\documentclass{amsart}
\usepackage{moreverb,url}
\usepackage{amssymb,amsmath,amsfonts}
\usepackage{bm}
\usepackage[font=scriptsize, labelfont=bf]{caption}
\usepackage{float}
\usepackage{varwidth}
\usepackage{booktabs}
\usepackage{afterpage}
\usepackage{mathrsfs}
\usepackage{subcaption}
\usepackage[squaren]{SIunits}
\usepackage{verbatim}
\usepackage[square,numbers]{natbib}
\usepackage{stmaryrd}
\usepackage{fancyhdr}
\usepackage{array}
\usepackage{syntax,etoolbox}
\usepackage{graphicx}
\usepackage{grffile}
\usepackage{hyperref}
\hypersetup{colorlinks=true,allcolors=blue}
\usepackage[square,numbers]{natbib}
\usepackage[left=1in, right=1in, bottom=1in,top=1in]{geometry}
\usepackage{lineno}
\usepackage{syntax,etoolbox}
\usepackage{algorithmicx}
\usepackage{algorithm}% http://ctan.org/pkg/algorithms
\usepackage{algpseudocode}% http://ctan.org/pkg/algorithmicx
\usepackage{tikz}
\allowdisplaybreaks

\newtheorem{theorem}{Theorem}[section]

\newtheorem{remark}{Remark}[section]

\newtheorem{innercustomgeneric}{\customgenericname}
\providecommand{\customgenericname}{}
\newcommand{\newcustomproblem}[2]{%
	\newenvironment{#1}[1]
	{%
		\renewcommand\customgenericname{#2}%
		\renewcommand\theinnercustomgeneric{##1}%
		\innercustomgeneric
	}
	{\endinnercustomgeneric}
}

\newcustomproblem{customprob}{Problem}

\newcommand*{\bqed}{\hfill\ensuremath{\blacksquare}}%

\newcommand{\vertiii}[1]{{\left\vert\kern-0.25ex\left\vert\kern-0.25ex\left\vert #1 
		\right\vert\kern-0.25ex\right\vert\kern-0.25ex\right\vert}}

\algblock{Input}{EndInput}

\algnotext{EndInput}

\algblock{Output}{EndOutput}

\algnotext{EndOutput}

\algblock{Iterate}{EndIterate}

\algnotext{EndIterate}

\makeatletter
\newcommand{\addresseshere}{%
	\enddoc@text\let\enddoc@text\relax
}
\makeatother

\begin{document}
	
	%%
	%% The title of the paper goes here.  Edit to your title.
	%%
	
	\title[Finite element CDA for an NSCH system]{A finite element continuous data assimilation framework for a Navier--Stokes--Cahn--Hilliard system}
	
	%%
	%% Now edit the following to give your name and address:
	%%

	\author[Tianyu Sun]{Tianyu Sun}
\address{T.S. Department of Mathematics and Institute for Scientific Computing and Applied Mathematics, Indiana University Bloomington, 729 East Third Street, Bloomington, Indiana, USA}
\email{ts19@iu.edu}

\today

\begin{abstract}
This paper studies a coupled two-dimensional Navier--Stokes--Cahn--Hilliard phase-field model augmented by a transported auxiliary field, and develops a continuous data assimilation (CDA) framework for recovering its trajectories from coarse-in-space observations.

We formulate a nudging-based CDA system for the coupled NSCH--auxiliary-field model, in which coarse measurements are incorporated through a general linear observation operator. The observation mechanism is described abstractly by an interpolant satisfying an $H^{2}$-type approximation property, which is compatible with coarse spatial observations obtained from mesh coarsening and reconstruction.

At the continuous level, we record two structural properties of the model: a formal energy law for the reference system and an evolution law for the phase mean in the assimilated dynamics. At the discrete level, we introduce a capped fully discrete finite element splitting scheme using continuous quadratic elements for the phase, chemical potential, velocity, and auxiliary field, together with continuous linear elements for the pressure. For this scheme, we prove one-step well-posedness and establish a stepwise a priori estimate for the capped method.

Numerical experiments illustrate the practical behavior of the proposed CDA approach. They demonstrate recovery from strongly mismatched initial conditions and show how synchronization depends on the observation resolution, boundary forcing, and feedback strength. A coarse-indistinguishability experiment further shows that identical coarse initial information may correspond to distinct fine-scale evolutions, while the assimilated dynamics select the trajectory determined by the supplied time-dependent observations.

\smallskip


\end{abstract}

\maketitle

\tableofcontents

\section{Introduction}

Phase-field models provide a flexible framework for describing multiphase hydrodynamics in which interfaces are represented by thin transition layers rather than sharp boundaries. The underlying diffuse-interface viewpoint goes back to the classical work of Cahn and Hilliard on interfacial free energy, and has since become a standard modeling tool for interfacial fluid phenomena; see, for example, \cite{CahnHilliard1958,AndersonMcFaddenWheeler1998}. Within this paradigm, the Navier--Stokes--Cahn--Hilliard (NSCH) system, often referred to as Model~H and its variants, provides a diffuse-interface description of incompressible two-phase flow that couples hydrodynamics with interfacial mixing and capillarity. Because the interface is represented through an order parameter rather than tracked explicitly, phase-field formulations are especially effective in situations involving topological changes, such as droplet merger and breakup, without the need for front tracking. Representative examples include microfluidic droplet formation and breakup~\cite{BaiHeYangZhouWang2017}. In biomechanics and hemodynamics, phase-field methods have also been used to describe thrombus formation and remodeling in flowing blood, incorporating spatially varying material properties and permeability within a unified energetic framework~\cite{ZhengYazdaniLiHumphreyKarniadakis2020,%
XuEtAl2017ClotShear,%
XuAlberXu2019ThreePhase}. Related developments connect phase-field hydrodynamics to complex-fluid effects such as viscoelasticity and transported microstructure; see, for example,~\cite{LinLiuZhang2005,GiorginiNdongmoNganaMedjoTemam2023}.

From the analytical viewpoint, the NSCH system and related diffuse-interface models have been studied extensively. For two-phase flows with matched densities, diffuse-interface models coupled to incompressible Navier--Stokes equations were analyzed by Abels~\cite{Abels2009Matched}, while for the variable-density setting strong and weak solvability results were established by Abels and collaborators~\cite{AbelsGarckeGruen2012,AbelsDepnerGarcke2013}. Degenerate-mobility variants have also been investigated, with the physically important feature that the order parameter remains in the admissible interval under suitable assumptions~\cite{AbelsDepnerGarcke2013Degenerate}. For the NSCH system itself, uniqueness and regularity theory with physically relevant potentials has been developed in several settings, including strong well-posedness in two dimensions and local strong theory in three dimensions~\cite{GiorginiMiranvilleTemam2019}. Long-time dynamics for CHNS-type systems have likewise been studied through attractors and asymptotic behavior~\cite{GalGrasselli2010,GiorginiTemam2022}. Broader background on the Cahn--Hilliard equation and its variants may be found in the surveys and monographs \cite{Miranville2017,Miranville2019}. Other works incorporate additional physical couplings and nonlinear effects, such as thermo-induced Marangoni forcing~\cite{Wu2017Marangoni} and chemotaxis-driven or singular-potential regimes in two dimensions~\cite{HeWu2021Chemotaxis}. Motivated in part by thrombus modeling, rigorous PDE analyses have also been obtained for phase-field blood--thrombus systems, including local strong well-posedness results in three dimensions~\cite{KimTawriTemam2022} and in two dimensions~\cite{GrasselliPoiatti2023}. These studies provide a useful foundation for investigating algorithmic questions in situations where only partial information about the evolving state is available.

On the numerical side, the CHNS system has generated a substantial literature devoted to efficient and stable discretizations. Among widely used approaches are convex-splitting and pressure-projection methods, decoupled energy-stable schemes, adaptive energy-stable methods, and second-order linear schemes that preserve discrete energy dissipation; see, for instance, \cite{HanWang2015,ShenYang2015,ChenShen2016,ChenZhao2020,ZhaoHan2021}. Such works show that preserving as much as possible of the energetic structure of the continuous model is central to robust computation. They also highlight the practical role of splitting, decoupling, and linearization in making phase-field fluid solvers computationally feasible. These considerations are particularly relevant in the present setting, where the coupled NSCH dynamics are further enriched by a transported auxiliary field and by assimilation terms driven by coarse observations.

A central practical difficulty is that accurate initial conditions and complete state observations are rarely accessible in applications. Data assimilation addresses this issue by incorporating observations into the model dynamics in order to recover, or track, the underlying trajectory. Among continuous-in-time approaches, the nudging framework of Azouani--Olson--Titi (AOT) has attracted considerable attention because of its relative simplicity at the PDE level and its connection with the idea that dissipative systems possess finitely many determining parameters~\cite{JonesTiti1993,AzouaniOlsonTiti2014}. A growing literature studies continuous and discrete data assimilation for fluid models, including computational validation for the two-dimensional Navier--Stokes equations~\cite{GeshoOlsonTiti2016}, discrete-in-time assimilation for the same system~\cite{FoiasMondainiTiti2016}, analytical results for convection models under partial observations~\cite{FarhatLunasinTiti2017}, and abridged or approximate assimilation procedures using reduced sets of measurements or regularized models~\cite{FarhatLunasinTiti2016Abridged,LariosPei2020}. For phase-field dynamics, CDA-type ideas have also been explored for the Cahn--Hilliard equation itself~\cite{DiegelRebholz2022}. More directly relevant to the present work, CDA has been developed for the two-dimensional Cahn--Hilliard--Navier--Stokes system, where convergence to the reference trajectory can be established under suitable resolution and nudging conditions~\cite{YouXia2022}. On the numerical side, data-assimilation modifications of NSCH have also been coupled with energy-stable time discretizations driven by coarse observed data~\cite{SongXiaKimLi2024}. At the same time, recent work has emphasized that nudging-type recovery may fail in non-dissipative regimes where determining parameters do not exist, underscoring the importance of dissipativity and observability in the underlying model~\cite{TitiVictor2025}.

The purpose of the present paper is to formulate and study a continuous data assimilation framework for a coupled NSCH-type phase-field model augmented by a transported auxiliary field. The additional field is intended to represent transported microstructural information and contributes an extra stress term in the momentum balance, making the system a natural extension of NSCH models arising in complex-fluid and thrombus-inspired settings~\cite{LinLiuZhang2005,ZhengYazdaniLiHumphreyKarniadakis2020,KimTawriTemam2022,GrasselliPoiatti2023,GiorginiNdongmoNganaMedjoTemam2023}. We restrict attention to two spatial dimensions and focus on the formulation of the CDA model, its basic structural properties, a practical capped fully discrete finite element realization, and the resulting computational behavior.

More precisely, we first introduce the coupled NSCH--$\psi$ system and the associated nudging-based CDA dynamics driven by coarse-in-space observations through a general linear observation operator. We then record two basic structural properties of the continuous system: a formal energy law for the reference dynamics and an evolution law for the phase mean in the assimilated system. The observation mechanism is described abstractly through an interpolant $I_H$ satisfying an $H^2$-type approximation property, which is compatible with coarse spatial observations obtained from mesh coarsening and reconstruction.

Next, we present a fully discrete finite element splitting scheme for the CDA system. The discretization uses continuous quadratic elements for the phase, chemical potential, velocity, and auxiliary field, together with continuous linear elements for the pressure. In order to match the numerical implementation and to ensure admissible coefficient evaluation, we introduce a capped phase in the fully discrete scheme. By combining skew-symmetric transport forms with this capped coefficient structure, we prove one-step well-posedness of the fully discrete problem and derive a stepwise a priori estimate for the capped scheme. We emphasize that this estimate is not a closed discrete free-energy law and does not by itself yield a full convergence theory; rather, it provides a basic stability bound consistent with the split discretization used in computation.

Finally, we perform a series of numerical experiments illustrating the behavior of the coupled CDA dynamics. These tests examine synchronization from strongly mismatched initial data, the influence of boundary forcing, the effect of the nudging coefficients, and the role of observation resolution. We also include a coarse-indistinguishability experiment showing that identical coarse initial information may correspond to distinct fine-scale evolutions, while the CDA dynamics select and track the trajectory determined by the supplied time-dependent observations.

\section{A phase-field model of NSCH type and continuous data assimilation}
\label{sec:NSCH-DA}

In this section, we introduce the coupled Navier--Stokes--Cahn--Hilliard (NSCH) system considered in this work,
together with an additional transported auxiliary phase field and the corresponding continuous data assimilation (CDA) algorithm.
Throughout, $\omega\subset\mathbb{R}^2$ denotes a bounded simply connected domain with $C^2$ boundary, and $T>0$ is a fixed final time. For $t\in(0,T)$, the unknowns are the velocity
$\vec{u}=(u_1,u_2)$, the pressure $p$, the phase-field order parameter $\phi\in[0,1]$, the chemical potential $\mu$, and an auxiliary
phase field $\vec{\psi}=(\psi_1, \psi_2)$.

We regard the density $\rho>0$ and the Reynolds number $\mathrm{Re}>0$ as fixed. The phase-dependent viscosity
$\eta(\phi)$, the mobility $\nu(\phi)$, and the permeability $\kappa(\phi)$ are assumed to satisfy
\begin{equation}
\label{eq:coeff-ass}
0<\eta_{\min}\le \eta(\phi)\le \eta_{\max},\qquad
0\le\nu_{\min}\le \nu(\phi)\le \nu_{\max},\qquad
0<\kappa_{\min}\le \kappa(\phi)\le \kappa_{\max},
\end{equation}
for all admissible values of $\phi$. Moreover, we assume $\eta,\nu,\kappa\in C^{1}(\mathbb{R})$ and that their derivatives are bounded on bounded sets.
In particular, for any $M>0$ there exists $C_M>0$ such that
\begin{equation}
\label{eq:coeff-ass-1}
|\eta'(s)|+|\nu'(s)|+|\kappa'(s)| \le C_M \qquad \text{for all } |s|\le M.
\end{equation}

We denote by $\lambda>0$ the gradient-energy coefficient and by $\gamma>0$ the bulk potential coefficient. The time-scale parameters
$\tau>0$ (Cahn--Hilliard mobility) and $\epsilon>0$ (diffusivity scale for $\vec{\psi}$) are also fixed.

Let $F:\mathbb{R}\to\mathbb{R}$ be a double-well potential and denote by $f:=F'$ its derivative. We assume $F\in C^2(\mathbb R)$ and that there exist $q\ge3$ and constants $C,c_F>0$, $C_F\ge0$ such that
\begin{equation}\label{as:f-growth}
|f(s)|\le C(1+|s|^{q-1})\qquad\forall\,s\in\mathbb R,
\end{equation}
and
\begin{equation}\label{as:coerciveF}
F(s)\ge c_F|s|^{q}-C_F\qquad\forall\,s\in\mathbb R.
\end{equation}

We consider the following coupled system on $\omega\times(0,T)$:
\begin{equation}
\label{eq:strong}
\left\{
\begin{array}{l}
\rho\,\mathrm{Re}\Big(\partial_t\vec{u}+(\vec{u}\cdot\nabla)\vec{u}\Big)
+\nabla p
-\nabla\cdot\!\big(\eta(\phi)\nabla\vec{u}\big)
\\
\qquad\qquad\qquad\qquad
=-\lambda\,\nabla\cdot(\nabla\phi\otimes\nabla\phi)
+\nabla\cdot\!\big(\nu(\phi)\,(\nabla\vec{\psi})^{T}\nabla\vec{\psi}\big)
-\frac{\eta(\phi)}{\kappa(\phi)}(1-\phi)\,\vec{u},
\\[2pt]
\nabla\cdot\vec{u}=0,
\\[2pt]
\partial_t\vec{\psi}+(\vec{u}\cdot\nabla)\vec{\psi}
=\epsilon\,\nabla\cdot\!\big(\nu(\phi)\nabla\vec{\psi}\big),
\\[2pt]
\partial_t\phi+\vec{u}\cdot\nabla\phi=\tau\,\Delta\mu,
\\[2pt]
\mu=-\lambda\,\Delta\phi+\lambda\gamma\,f(\phi)+\frac{1}{2}\nu'(\phi)\,|\nabla\vec{\psi}|^2,
\end{array}
\right.
\end{equation}
The term $-\lambda\nabla\cdot(\nabla\phi\otimes\nabla\phi)$ is the stress contribution generated by the gradient energy of $\phi$,
while $\nabla\cdot\!\big(\nu(\phi)(\nabla\vec{\psi})^{T}\nabla\vec{\psi}\big)$ is the analogous stress contribution generated by
the auxiliary field $\vec{\psi}$. The resistance term $\frac{\eta(\phi)}{\kappa(\phi)}(1-\phi)\vec{u}$ penalises
the flow in regions where $\phi$ is small.

We prescribe the boundary conditions
\begin{equation}
\label{eq:bc}
\vec{u}=\vec{0}\quad \textup{on }\partial \omega\times(0,T),\qquad
\partial_{\vec{\nu}}\phi=\partial_{\vec{\nu}}\mu=0\quad \textup{on }\partial \omega\times(0,T),\qquad
\nu(\phi)\,\partial_{\vec{\nu}}\vec{\psi}=0\quad \textup{on }\partial \omega\times(0,T),
\end{equation}
and the initial data
\begin{equation}
\label{eq:ic}
\vec{u}(\cdot,0)=\vec{u}_0,\qquad \phi(\cdot,0)=\phi_0,\qquad \vec{\psi}(\cdot,0)=\vec{\psi}_0,\qquad \textup{in }\omega.
\end{equation}

We introduce the space of smooth solenoidal test fields
\[
\mathcal V:=\{\vec w\in \vec C_0^\infty(\omega):\nabla\!\cdot \vec w=0\}.
\]
We denote by $\mathbf H$ the closure of $\mathcal V$ in $\vec L^2(\omega)$, and by $\mathbf V$ the closure of $\mathcal V$ in $H_0^1(\omega)$.
It is standard that these spaces admit the characterizations
\[
\mathbf H=\{\vec w\in \vec L^2(\omega):\nabla\!\cdot\vec w=0,\ \vec w\cdot\vec\nu=0\ \text{on }\partial \omega\},
\qquad
\mathbf V=\{\vec w\in H_0^1(\omega):\nabla\!\cdot\vec w=0\}.
\]
Let $P:\vec L^2(\omega)\to \mathbf H$ be the Helmholtz--Leray orthogonal projection and define
the Stokes operator $A:=-P\Delta$ with domain
\[
D(A):=H^2(\omega)\cap \mathbf V.
\]

%============================================================

For any sufficiently smooth $\phi$, the stress admits the identity
\begin{equation}\label{eq:korteweg-id}
-\lambda\,\nabla\cdot(\nabla\phi\otimes\nabla\phi)
=
\lambda(-\Delta\phi)\,\nabla\phi-\nabla\Big(\frac{\lambda}{2}|\nabla\phi|^2\Big).
\end{equation}
Using the definition of the chemical potential in \eqref{eq:strong},
\[
\mu=-\lambda\,\Delta\phi+\lambda\gamma\,f(\phi)+\frac12\nu'(\phi)\,|\nabla\vec{\psi}|^2,
\]
and the relation $f=F'$, we may rewrite \eqref{eq:korteweg-id} as
\begin{equation}\label{eq:korteweg-mu}
-\lambda\,\nabla\cdot(\nabla\phi\otimes\nabla\phi)
=
\mu\,\nabla\phi
-\frac12\,\nu'(\phi)\,|\nabla\vec{\psi}|^2\,\nabla\phi
-\nabla\Big(\lambda\gamma\,F(\phi)+\frac{\lambda}{2}|\nabla\phi|^2\Big).
\end{equation}
Consequently, introducing the modified pressure
\begin{equation}\label{eq:ptilde}
\tilde p \;:=\; p+\lambda\gamma\,F(\phi)+\frac{\lambda}{2}|\nabla\phi|^2,
\end{equation}
the system \eqref{eq:strong} is equivalently written in the form
\begin{equation}
\label{eq:strong-korteweg}
\left\{
\begin{array}{l}
\rho\,\mathrm{Re}\Big(\partial_t\vec{u}+(\vec{u}\cdot\nabla)\vec{u}\Big)
+\nabla \tilde p
-\nabla\cdot\!\big(\eta(\phi)\nabla\vec{u}\big)
\\
\qquad\qquad\qquad\qquad
=\mu\,\nabla\phi
-\dfrac{1}{2}\nu'(\phi)\,|\nabla\vec{\psi}|^2\,\nabla\phi
+\nabla\cdot\!\big(\nu(\phi)\,(\nabla\vec{\psi})^{T}\nabla\vec{\psi}\big)
-\frac{\eta(\phi)}{\kappa(\phi)}(1-\phi)\,\vec{u},
\\[2pt]
\nabla\cdot\vec{u}=0,
\\[2pt]
\partial_t\vec{\psi}+(\vec{u}\cdot\nabla)\vec{\psi}
=\epsilon\,\nabla\cdot\!\big(\nu(\phi)\nabla\vec{\psi}\big),
\\[2pt]
\partial_t\phi+\vec{u}\cdot\nabla\phi=\tau\,\Delta\mu,
\\[2pt]
\mu=-\lambda\,\Delta\phi+\lambda\gamma\,f(\phi)+\frac{1}{2}\nu'(\phi)\,|\nabla\vec{\psi}|^2.
\end{array}
\right.
\end{equation}
\medskip

The formulations \eqref{eq:strong} and \eqref{eq:strong-korteweg} are algebraically equivalent and therefore share the same existence and uniqueness theory for strong solutions under \eqref{eq:coeff-ass} and mild regularity assumptions on $f$ (see, e.g., \cite{KimTawriTemam2022}).  We now introduce the continuous data assimilation (CDA) algorithm associated with \eqref{eq:strong-korteweg}, following the nudging
framework of Azouani--Olson--Titi~\cite{AzouaniOlsonTiti2014}. The objective is to recover the unknown reference
trajectory from coarse-in-space observations by evolving an assimilated state and adding linear feedback terms that relax the
observed large scales toward the available measurements.

Let $h>0$ denote the observational resolution and let $I_H$ be a linear interpolant mapping sufficiently regular fields on $\omega$
into an observation space. We assume that there exist constants $c_0,c_I>0$, independent of $h$, such that
\begin{equation}\label{eq:Ih-props}
\|w-I_H w\|_{L^2(\omega)}^2 \le \frac{1}{4}c_0^2\,h^{4}\,\|w\|_{H^{2}(\omega)}^2
\qquad \forall\, w\in H^{2}(\omega),
\end{equation}

In our computations, $I_H$ is realized by transferring fine-grid finite element fields onto a uniformly coarsened mesh. 

\begin{customprob}{$\mathcal{P}$}
Let $(\vec u,\tilde p,\phi,\mu,\vec\psi)$ denote the reference solution of \eqref{eq:strong}--\eqref{eq:ic}.
The assimilated unknowns $(\vec v,\tilde\pi,\varphi,\xi,\vec\zeta)$ are defined as the solution of the nudged NSCH system
\begin{equation}
\label{eq:strong-CDA}
\left\{
\begin{array}{l}
\rho\,\mathrm{Re}\Big(\partial_t\vec{v}+(\vec{v}\cdot\nabla)\vec{v}\Big)
+\nabla \tilde \pi
-\nabla\cdot\!\big(\eta(\varphi)\nabla\vec{v}\big)
\\
\qquad\qquad\qquad\qquad
=\xi\,\nabla\varphi
-\dfrac{1}{2}\nu'(\varphi)\,|\nabla\vec \zeta|^2\,\nabla\varphi
+\nabla\cdot\!\big(\nu(\varphi)\,(\nabla\vec \zeta)^{T}\nabla\vec \zeta\big)
-\frac{\eta(\varphi)}{\kappa(\varphi)}(1-\varphi)\,\vec{v}
+\alpha_u\,I_H(\vec u-\vec v),
\\[2pt]
\nabla\cdot\vec{v}=0,
\\[2pt]
\partial_t\vec \zeta+(\vec v\cdot\nabla)\vec \zeta
=\epsilon\,\nabla\cdot\!\big(\nu(\varphi)\nabla\vec \zeta\big)
+\alpha_\psi\,I_H(\vec \psi-\vec \zeta),
\\[2pt]
\partial_t\varphi+\vec v\cdot\nabla\varphi
=\tau\,\Delta\xi+\alpha_\phi\,I_H(\phi-\varphi),
\\[2pt]
\xi=-\lambda\,\Delta\varphi+\lambda\gamma\,f(\varphi)+\frac{1}{2}\nu'(\varphi)\,|\nabla\vec \zeta|^2,
\end{array}
\right.
\end{equation}
supplemented with the same boundary conditions as \eqref{eq:bc} with $(\vec u,\phi,\mu,\vec\psi)$ replaced by
$(\vec v,\varphi,\xi,\vec \zeta)$ and initial data
\begin{equation}\label{eq:ic-CDA}
\vec v(\cdot,0)=\vec v_0,\qquad \varphi(\cdot,0)=\varphi_0,\qquad \vec \zeta(\cdot,0)=\vec \zeta_0,
\end{equation}
which may be chosen independently of the reference initial data. Here $\alpha_u,\alpha_\phi,\alpha_\psi\ge 0$ are the nudging parameters. The feedback acts only through the coarse observation operator $I_H$. 

\end{customprob}
Based on the formal energetic structure of \eqref{eq:strong-CDA}, we define the energy by

\begin{equation}\label{eq:E-def}
\mathcal E(t):=\int_\omega\Big(
\frac{\rho\mathrm{Re}}2|\vec v|^2+\frac\lambda2|\nabla\varphi|^2+\lambda\gamma F(\varphi)
+\frac12\nu(\varphi)|\nabla\vec\zeta|^2\Big)\,dx.
\end{equation}

%===========================================================

We begin by recording the formal energy law and the behavior of the spatial mean. These identities explain the dissipative structure of the reference system and the role of the nudging terms in the assimilated dynamics.

\begin{theorem}
\label{prop:formal-energy}
Let $(\vec u,\tilde p,\phi,\mu,\vec\psi)$ be a sufficiently smooth solution of
\eqref{eq:strong-korteweg}--\eqref{eq:ic} satisfying the boundary conditions \eqref{eq:bc}.
Then the energy
\[
\mathcal E_{\rm ref}(t):=\int_\omega\Big(
\frac{\rho\mathrm{Re}}2|\vec u|^2
+\frac\lambda2|\nabla\phi|^2
+\lambda\gamma F(\phi)
+\frac12\nu(\phi)|\nabla\vec\psi|^2\Big)\,dx
\]
satisfies
\begin{equation}
\label{eq:formal-energy-law}
\frac{d}{dt}\mathcal E_{\rm ref}(t)
+\int_\omega \eta(\phi)|\nabla\vec u|^2\,dx
+\int_\omega \frac{\eta(\phi)}{\kappa(\phi)}(1-\phi)|\vec u|^2\,dx
+\tau\|\nabla\mu\|_{L^2(\omega)}^2
+\epsilon\|\nabla\cdot(\nu(\phi)\nabla\vec\psi)\|_{L^2(\omega)}^2
=0.
\end{equation}
\end{theorem}

\begin{proof}
We test the momentum equation by $\vec u$ and integrate over $\omega$. Using
$\nabla\cdot\vec u=0$, $\vec u|_{\partial\omega}=0$, and the boundary conditions, we obtain
\begin{align}
\frac{\rho\mathrm{Re}}2\frac{d}{dt}\|\vec u\|_{L^2}^2
+\int_\omega \eta(\phi)|\nabla\vec u|^2\,dx
+\int_\omega \frac{\eta(\phi)}{\kappa(\phi)}(1-\phi)|\vec u|^2\,dx
&=(\mu\nabla\phi,\vec u)
-\frac12\big(\nu'(\phi)|\nabla\vec\psi|^2\nabla\phi,\vec u\big)
\nonumber\\
&\quad
-\big(\nu(\phi)(\nabla\vec\psi)^T\nabla\vec\psi,\nabla\vec u\big).
\label{eq:proof-mom}
\end{align}
Here we used
\[
((\vec u\cdot\nabla)\vec u,\vec u)=0,
\qquad
(\nabla\tilde p,\vec u)=0.
\]

Next, testing the phase equation by $\mu$ gives
\begin{equation}
\label{eq:proof-phase1}
(\partial_t\phi,\mu)+(\vec u\cdot\nabla\phi,\mu)
+\tau\|\nabla\mu\|_{L^2}^2=0.
\end{equation}
Using the constitutive relation for $\mu$, we compute
\begin{align}
(\partial_t\phi,\mu)
&=
\lambda(\nabla\phi,\nabla\partial_t\phi)
+\lambda\gamma(f(\phi),\partial_t\phi)
+\frac12\big(\nu'(\phi)|\nabla\vec\psi|^2,\partial_t\phi\big)
\nonumber\\
&=
\frac{d}{dt}\int_\omega
\Big(\frac\lambda2|\nabla\phi|^2+\lambda\gamma F(\phi)\Big)\,dx
+\frac12\int_\omega \nu'(\phi)|\nabla\vec\psi|^2\,\partial_t\phi\,dx.
\label{eq:proof-phase2}
\end{align}

For the auxiliary field, we test the $\vec\psi$-equation by
$-\nabla\cdot(\nu(\phi)\nabla\vec\psi)$ and integrate by parts. This yields
\begin{align}
\frac12\frac{d}{dt}\int_\omega \nu(\phi)|\nabla\vec\psi|^2\,dx
-\frac12\int_\omega \nu'(\phi)|\nabla\vec\psi|^2\,\partial_t\phi\,dx
+\epsilon\|\nabla\cdot(\nu(\phi)\nabla\vec\psi)\|_{L^2}^2
\nonumber\\
\qquad
+\big((\vec u\cdot\nabla)\vec\psi,\,-\nabla\cdot(\nu(\phi)\nabla\vec\psi)\big)=0.
\label{eq:proof-psi1}
\end{align}
A direct integration by parts shows that
\begin{equation}
\label{eq:proof-psi-conv}
\big((\vec u\cdot\nabla)\vec\psi,\,-\nabla\cdot(\nu(\phi)\nabla\vec\psi)\big)
=
\big(\nu(\phi)(\nabla\vec\psi)^T\nabla\vec\psi,\nabla\vec u\big)
-\frac12\big(\nu'(\phi)|\nabla\vec\psi|^2\nabla\phi,\vec u\big).
\end{equation}

Finally, summing \eqref{eq:proof-mom}, \eqref{eq:proof-phase1}--\eqref{eq:proof-phase2},
and \eqref{eq:proof-psi1}--\eqref{eq:proof-psi-conv}, all mixed coupling terms cancel:
\[
(\mu\nabla\phi,\vec u)-(\vec u\cdot\nabla\phi,\mu)=0,
\]
and the $\nu'(\phi)|\nabla\vec\psi|^2$-terms also cancel exactly.
Therefore we obtain \eqref{eq:formal-energy-law}.
\end{proof}

\begin{theorem}
\label{prop:mean}
Assume $(\vec u,\tilde p,\phi,\mu,\vec\psi)$ is a sufficiently smooth solution of
\eqref{eq:strong-korteweg} and $(\vec v,\tilde\pi,\varphi,\xi,\vec\zeta)$ is a sufficiently smooth
solution of \eqref{eq:strong-CDA}. Assume also that the observation operator $I_H$ satisfies
\eqref{eq:Ih-props}. Then:
\begin{enumerate}
\item The reference phase mass is conserved:
\[
\int_\omega \phi(t)\,dx=\int_\omega \phi_0\,dx
\qquad \text{for all } t\in[0,T].
\]

\item The assimilated phase mass satisfies
\[
\frac{d}{dt}\int_\omega \varphi(t)\,dx
=
\alpha_\phi\int_\omega I_H(\phi-\varphi)\,dx.
\]

\item Let $e_\phi:=\phi-\varphi$. Then the mean error satisfies
\[
\frac{d}{dt}\int_\omega e_\phi\,dx+\alpha_\phi\int_\omega e_\phi\,dx
=
\alpha_\phi\int_\omega (e_\phi-I_H e_\phi)\,dx.
\]
Consequently,
\[
\left|
\frac{d}{dt}\big(\overline\phi-\overline\varphi\big)
+\alpha_\phi(\overline\phi-\overline\varphi)
\right|
\le
\frac{c_0}{2}\,\alpha_\phi\,|\omega|^{-1/2} h^2
\|e_\phi\|_{H^2(\omega)},
\]
where
\[
\overline\phi:=|\omega|^{-1}\int_\omega\phi\,dx,
\qquad
\overline\varphi:=|\omega|^{-1}\int_\omega\varphi\,dx.
\]
\end{enumerate}
\end{theorem}

\begin{proof}
Integrating the reference phase equation over $\omega$ and using
$\partial_{\nu}\mu=0$ on $\partial\omega$, together with $\vec u\cdot\vec\nu=0$ on $\partial\omega$ and
$\nabla\cdot\vec u=0$, we obtain
\[
\frac{d}{dt}\int_\omega \phi\,dx
=
\tau\int_\omega \Delta\mu\,dx-\int_\omega \vec u\cdot\nabla\phi\,dx=0.
\]
This proves the conservation of the reference phase mass.

Similarly, integrating the assimilated phase equation gives
\[
\frac{d}{dt}\int_\omega \varphi\,dx
=
\tau\int_\omega \Delta\xi\,dx-\int_\omega \vec v\cdot\nabla\varphi\,dx
+\alpha_\phi\int_\omega I_H(\phi-\varphi)\,dx
=
\alpha_\phi\int_\omega I_H(\phi-\varphi)\,dx.
\]
Thus, with $e_\phi=\phi-\varphi$,
\[
\frac{d}{dt}\int_\omega e_\phi\,dx
=
-\alpha_\phi\int_\omega I_H e_\phi\,dx,
\]
and therefore
\[
\frac{d}{dt}\int_\omega e_\phi\,dx+\alpha_\phi\int_\omega e_\phi\,dx
=
\alpha_\phi\int_\omega (e_\phi-I_H e_\phi)\,dx.
\]

Finally, by Cauchy--Schwarz and \eqref{eq:Ih-props},
\[
\left|\int_\omega (e_\phi-I_H e_\phi)\,dx\right|
\le
|\omega|^{1/2}\|e_\phi-I_H e_\phi\|_{L^2(\omega)}
\le
\frac{c_0}{2}|\omega|^{1/2}h^2\|e_\phi\|_{H^2(\omega)}.
\]
Dividing by $|\omega|$ yields
\[
\left|
\frac{d}{dt}\big(\overline\phi-\overline\varphi\big)
+\alpha_\phi(\overline\phi-\overline\varphi)
\right|
\le
\frac{c_0}{2}\,\alpha_\phi\,|\omega|^{-1/2} h^2
\|e_\phi\|_{H^2(\omega)}.
\]
This completes the proof.
\end{proof}

\begin{remark}
The previous theorem shows that, unlike the reference phase mass, the assimilated phase mean is generally not conserved unless the observation operator has an additional mean-preserving property. Thus the coarse-observation mechanism can influence the large-scale phase balance even when the reference dynamics conserve mass.
\end{remark}

\section{Finite element discretization and one-step well-posedness}
\label{sec:fem}

We now introduce a fully discrete finite element splitting scheme used in the numerical experiments. Our goal here is not to develop a full numerical convergence theory, but rather to record the discretization and establish one-step solvability. Let $\mathcal T_h$ be a conforming triangulation of $\omega$. For each triangle $K\in\mathcal T_h$, let $P_r(K)$ denote the space of scalar polynomials on $K$ of total degree at most $r$. In particular, $P_1(K)$ and $P_2(K)$ are the standard linear and quadratic Lagrange finite elements on $K$. We introduce the finite element spaces
\[
V_h:=\{\,\vartheta_h\in C^0(\overline\omega):\ \vartheta_h|_K\in P_2(K)\ \forall K\in\mathcal T_h\,\},
\]
\[
Q_h:=\{\,q_h\in C^0(\overline\omega):\ q_h|_K\in P_1(K)\ \forall K\in\mathcal T_h\,\},
\]
\[
Q_h^0:=\Bigl\{q_h\in Q_h:\int_\omega q_h\,dx=0\Bigr\},
\]
\[
\vec V_h:=\{\vec w_h\in [V_h]^2:\ \vec w_h=\vec 0 \text{ on }\partial\omega\},
\qquad
\vec W_h:=[V_h]^2.
\]
\begin{customprob}{$\mathcal{P}_{h}$}
\label{prob:disc-NSCH}

Let $t^n:=n\Delta t$ for $n=0,\dots,N$. Assume that coarse observations of the reference fields
$(\vec u^n,\phi^n,\vec\psi^n)$ are available through $I_H(\vec u^n)$, $I_H(\phi^n)$, and $I_H(\vec\psi^n)$.

The CDA finite element unknowns are
\[
(\vec v_h^{\,n},\tilde\pi_h^{\,n},\varphi_h^{\,n},\xi_h^{\,n},\vec\zeta_h^{\,n})
\in \vec V_h\times Q_h^0\times V_h\times V_h\times \vec W_h .
\]
Given initial data $(\vec v_h^{\,0},\varphi_h^{\,0},\vec\zeta_h^{\,0})$ and an initial pressure guess
$\tilde\pi_h^{\,0}\in Q_h^0$, compute
$(\vec v_h^{\,n+1},\tilde\pi_h^{\,n+1},\varphi_h^{\,n+1},\xi_h^{\,n+1},\vec\zeta_h^{\,n+1})$
for $n=0,\dots,N-1$ by the following splitting procedure.
The no-slip condition for $\vec v_h$ is imposed strongly, while Neumann conditions for
$\varphi_h$, $\xi_h$, and $\vec\zeta_h$ are enforced naturally in the weak forms.

For $\vec a\in \vec V_h$, $r,s\in V_h$, $\vec z,\vec\theta\in \vec W_h$, and
$\vec b,\vec w\in \vec V_h$, define
\[
b_\phi(\vec a;r,s)
:=
-\big(r\vec a,\nabla s\big)
-\frac12\big((\nabla\!\cdot\vec a)\,r,s\big),
\]
\[
b_\psi(\vec a;\vec z,\vec\theta)
:=
\big((\vec a\cdot\nabla)\vec z,\vec\theta\big)
+\frac12\big((\nabla\!\cdot\vec a)\,\vec z,\vec\theta\big),
\]
and
\[
c(\vec a;\vec b,\vec w)
:=
\big((\vec a\cdot\nabla)\vec b,\vec w\big)
+\frac12\big((\nabla\!\cdot\vec a)\,\vec b,\vec w\big).
\]
Here
\[
\big((\nabla\!\cdot\vec a)\,\vec z,\vec\theta\big)
:=\sum_{j=1}^{2}\big((\nabla\!\cdot\vec a)\,z_j,\theta_j\big).
\]

Define the truncation operator
\[
\mathcal T(s):=\min\{1,\max\{0,s\}\}, \qquad s\in\mathbb R.
\]
For each time level \(n\), we set
\[
\widehat\varphi_h^{\,n}:=\mathcal T(\varphi_h^{\,n}).
\]
Thus \(\varphi_h^{\,n}\in V_h\) remains the finite element unknown, while
\(\widehat\varphi_h^{\,n}\in[0,1]\) is the capped phase used in the
evaluation of the phase-dependent coefficients. This reflects the numerical
implementation in which the phase is clipped to \([0,1]\) at each time step
before forming the coefficients.

\medskip
\noindent (i) Cahn--Hilliard step.
Given \((\vec v_h^{\,n},\varphi_h^{\,n},\vec\zeta_h^{\,n})\), find \((\varphi_h^{\,n+1},\xi_h^{\,n+1})\in V_h\times V_h\) such that, for all
\((\vartheta_h,\chi_h)\in V_h\times V_h\),
\begin{equation}
\label{eq:disc-CH-weak-CDA-final}
\begin{aligned}
\Big(\frac{\varphi_h^{\,n+1}-\varphi_h^{\,n}}{\Delta t},\vartheta_h\Big)
+b_\phi(\vec v_h^{\,n};\varphi_h^{\,n+1},\vartheta_h)
+\tau\,(\nabla\xi_h^{\,n+1},\nabla\vartheta_h)
&=\big(\alpha_\phi\,I_H(\phi^{n}-\varphi_h^{\,n}),\vartheta_h\big),
\\
(\xi_h^{\,n+1},\chi_h)
-\lambda(\nabla\varphi_h^{\,n+1},\nabla\chi_h)
&=\lambda\gamma\big(f(\widehat\varphi_h^{\,n}),\chi_h\big)
+\frac12\big(\nu'(\widehat\varphi_h^{\,n})|\nabla\vec\zeta_h^{\,n}|^2,\chi_h\big).
\end{aligned}
\end{equation}

\medskip
\noindent (ii) Auxiliary-field step.
Given \((\vec v_h^{\,n},\varphi_h^{\,n+1},\vec\zeta_h^{\,n})\), find \(\vec\zeta_h^{\,n+1}\in \vec W_h\) such that, for all \(\vec\theta_h\in \vec W_h\),
\begin{equation}
\label{eq:disc-psi-weak-CDA-final}
\Big(\frac{\vec\zeta_h^{\,n+1}-\vec\zeta_h^{\,n}}{\Delta t},\vec\theta_h\Big)
+b_\psi(\vec v_h^{\,n};\vec\zeta_h^{\,n+1},\vec\theta_h)
+\epsilon\big(\nu(\widehat\varphi_h^{\,n+1})\nabla\vec\zeta_h^{\,n+1},\nabla\vec\theta_h\big)
=\big(\alpha_\psi\,I_H(\vec\psi^{\,n}-\vec\zeta_h^{\,n}),\vec\theta_h\big).
\end{equation}

\medskip
\noindent (iii) Navier--Stokes step and pressure correction.
Given \((\vec v_h^{\,n},\tilde\pi_h^{\,n},\tilde\pi_h^{\,n-1},\varphi_h^{\,n+1},\xi_h^{\,n+1},\vec\zeta_h^{\,n+1})\),
find \(\vec v_h^{\,n+1}\in \vec V_h\) such that, for all \(\vec w_h\in \vec V_h\),
\begin{equation}
\label{eq:disc-NS-weak-CDA-final}
\begin{aligned}
&\Big(\frac{\rho\,\mathrm{Re}}{\Delta t}\vec v_h^{\,n+1},\vec w_h\Big)
+\rho\,\mathrm{Re}\,c(\vec v_h^{\,n};\vec v_h^{\,n+1},\vec w_h)
+\big(\eta(\widehat\varphi_h^{\,n+1})\mathbf \nabla(\vec v_h^{\,n+1}),\mathbf \nabla(\vec w_h)\big)
+\Big(\frac{\eta(\widehat\varphi_h^{\,n+1})}{\kappa(\widehat\varphi_h^{\,n+1})}(1-\widehat\varphi_h^{\,n+1})\vec v_h^{\,n+1},\vec w_h\Big)
\\
&\qquad
=\Big(\frac{\rho\,\mathrm{Re}}{\Delta t}\vec v_h^{\,n},\vec w_h\Big)
+\big(2\tilde\pi_h^{\,n}-\tilde\pi_h^{\,n-1},\nabla\cdot \vec w_h\big)
+\big(\xi_h^{\,n+1}\nabla\varphi_h^{\,n+1},\vec w_h\big)
-\frac12\big(\nu'(\widehat\varphi_h^{\,n+1})|\nabla\vec\zeta_h^{\,n+1}|^2\nabla\varphi_h^{\,n+1},\vec w_h\big)
\\
&\qquad\qquad
-\big(\nu(\widehat\varphi_h^{\,n+1})(\nabla\vec\zeta_h^{\,n+1})^{T}\nabla\vec\zeta_h^{\,n+1}, \nabla \vec w_h\big)
+\big(\alpha_u\,I_H(\vec u^{\,n}-\vec v_h^{\,n}),\vec w_h\big).
\end{aligned}
\end{equation}

Finally, compute $\tilde\pi_h^{\,n+1}\in Q_h^0$ from the Poisson correction: find
$\tilde\pi_h^{\,n+1}\in Q_h^0$ such that, for all $q_h\in Q_h^0$,
\begin{equation}
\label{eq:disc-press-CDA-final}
(\nabla \tilde\pi_h^{\,n+1},\nabla q_h)
=\Big(\frac{\rho\,\mathrm{Re}}{\Delta t}\nabla\cdot\vec v_h^{\,n+1},q_h\Big)
-(\nabla \tilde\pi_h^{\,n},\nabla q_h).
\end{equation}
\bqed
\end{customprob}

Before stating the one-step solvability result, we note that schemes of this
general type have been widely used in the numerical treatment of diffuse-interface
fluid models. For the Cahn--Hilliard--Navier--Stokes system and related
phase-field flow equations, unique solvability, linearized solvability, and
discrete stability have been proved for convex-splitting, projection-based,
and auxiliary-variable discretizations; see, for example,
\cite{GaoWang2014,HanWang2015,ChenZhao2020,SongXiaKimLi2024}.
Classical pressure-correction and fractional-step ideas underlying the velocity
update also go back to the projection methods of Chorin and Temam
\cite{Chorin1968,Temam1969}; see also \cite{GiraultRaviart1986,Temam1977}
for general background.
In the present setting, once the quantities at time level $n$ are fixed, each
subproblem in Problem~\ref{prob:disc-NSCH} is linear in the new unknowns, the
transport forms are skew-symmetric, and the capped phase
$\widehat\varphi_h^{\,n}$ keeps the coefficient evaluations within the
admissible interval $[0,1]$. This leads to the following one-step
well-posedness statement.

\begin{theorem}
\label{prop:disc-wp}
Assume that $\eta,\nu,\kappa$ satisfy \eqref{eq:coeff-ass}. Let
\[
(\vec v_h^{\,n},\tilde\pi_h^{\,n},\tilde\pi_h^{\,n-1},\varphi_h^{\,n},\vec\zeta_h^{\,n})
\in \vec V_h\times Q_h^0\times Q_h^0\times V_h\times \vec W_h
\]
be given, and define
\[
\widehat\varphi_h^{\,n}:=\mathcal T(\varphi_h^{\,n}).
\]
Then the subproblems in Problem~\ref{prob:disc-NSCH} admit unique solutions
\[
(\varphi_h^{\,n+1},\xi_h^{\,n+1})\in V_h\times V_h,\qquad
\vec\zeta_h^{\,n+1}\in \vec W_h,\qquad
\vec v_h^{\,n+1}\in \vec V_h,\qquad
\tilde\pi_h^{\,n+1}\in Q_h^0.
\]
\end{theorem}

\begin{proof}
Since all finite element spaces are finite dimensional and each subproblem is linear in the new unknowns, it is enough to prove uniqueness for the associated homogeneous problems.

Assume that
\[
(\varphi_{h,1}^{\,n+1},\xi_{h,1}^{\,n+1}),
\qquad
(\varphi_{h,2}^{\,n+1},\xi_{h,2}^{\,n+1})
\]
are two solutions of \eqref{eq:disc-CH-weak-CDA-final}. Define
\[
\Phi_h:=\varphi_{h,1}^{\,n+1}-\varphi_{h,2}^{\,n+1},
\qquad
\Xi_h:=\xi_{h,1}^{\,n+1}-\xi_{h,2}^{\,n+1}.
\]
Then $(\Phi_h,\Xi_h)\in V_h\times V_h$ satisfies
\begin{equation}
\label{eq:CH-hom-final}
\begin{aligned}
\Big(\frac{\Phi_h}{\Delta t},\vartheta_h\Big)
+b_\phi(\vec v_h^{\,n};\Phi_h,\vartheta_h)
+\tau(\nabla\Xi_h,\nabla\vartheta_h)&=0,
\\
(\Xi_h,\chi_h)-\lambda(\nabla\Phi_h,\nabla\chi_h)&=0,
\end{aligned}
\end{equation}
for all $(\vartheta_h,\chi_h)\in V_h\times V_h$.
Choosing $\vartheta_h=\Phi_h$ and $\chi_h=\dfrac{\tau}{\lambda}\Xi_h$, and using
\[
b_\phi(\vec v_h^{\,n};\Phi_h,\Phi_h)=0,
\]
we obtain
\[
\frac1{\Delta t}\|\Phi_h\|_{L^2(\omega)}^2
+\tau(\nabla\Xi_h,\nabla\Phi_h)=0,
\]
and
\[
\frac{\tau}{\lambda}\|\Xi_h\|_{L^2(\omega)}^2
-\tau(\nabla\Phi_h,\nabla\Xi_h)=0.
\]
Adding these identities yields
\[
\frac1{\Delta t}\|\Phi_h\|_{L^2(\omega)}^2
+\frac{\tau}{\lambda}\|\Xi_h\|_{L^2(\omega)}^2=0.
\]
Hence $\Phi_h=0$ and $\Xi_h=0$. Therefore the Cahn--Hilliard subproblem is uniquely solvable.

Now assume that $\vec\zeta_{h,1}^{\,n+1}$ and $\vec\zeta_{h,2}^{\,n+1}$ are two solutions of
\eqref{eq:disc-psi-weak-CDA-final}, and define
\[
\vec Z_h:=\vec\zeta_{h,1}^{\,n+1}-\vec\zeta_{h,2}^{\,n+1}\in \vec W_h.
\]
Since both solutions are computed with the same $\varphi_h^{\,n+1}$, the same capped phase
\[
\widehat\varphi_h^{\,n+1}:=\mathcal T(\varphi_h^{\,n+1})
\]
appears in the coefficient. Thus $\vec Z_h$ satisfies
\[
\Big(\frac{\vec Z_h}{\Delta t},\vec\theta_h\Big)
+b_\psi(\vec v_h^{\,n};\vec Z_h,\vec\theta_h)
+\epsilon\big(\nu(\widehat\varphi_h^{\,n+1})\nabla\vec Z_h,\nabla\vec\theta_h\big)=0
\qquad \forall\,\vec\theta_h\in \vec W_h.
\]
Taking $\vec\theta_h=\vec Z_h$ and using
\[
b_\psi(\vec v_h^{\,n};\vec Z_h,\vec Z_h)=0,
\]
we obtain
\[
\frac1{\Delta t}\|\vec Z_h\|_{L^2(\omega)}^2
+\epsilon\int_\omega \nu(\widehat\varphi_h^{\,n+1})|\nabla\vec Z_h|^2\,dx=0.
\]
Because $0\le \widehat\varphi_h^{\,n+1}\le 1$ and \eqref{eq:coeff-ass} holds on admissible values,
\[
\nu(\widehat\varphi_h^{\,n+1})\ge \nu_{\min}\ge 0.
\]
Hence
\[
\frac1{\Delta t}\|\vec Z_h\|_{L^2(\omega)}^2
+\epsilon\nu_{\min}\|\nabla\vec Z_h\|_{L^2(\omega)}^2=0,
\]
and therefore $\vec Z_h=0$. Thus the auxiliary-field subproblem is uniquely solvable.

Next assume that $\vec v_{h,1}^{\,n+1}$ and $\vec v_{h,2}^{\,n+1}$ are two solutions of
\eqref{eq:disc-NS-weak-CDA-final}, and define
\[
\vec U_h:=\vec v_{h,1}^{\,n+1}-\vec v_{h,2}^{\,n+1}\in \vec V_h.
\]
Since both solutions are computed with the same $\varphi_h^{\,n+1}$, $\xi_h^{\,n+1}$, and $\vec\zeta_h^{\,n+1}$, and therefore with the same capped phase
\[
\widehat\varphi_h^{\,n+1}:=\mathcal T(\varphi_h^{\,n+1}),
\]
the difference $\vec U_h$ satisfies
\begin{align*}
&\Big(\frac{\rho\,\mathrm{Re}}{\Delta t}\vec U_h,\vec w_h\Big)
+\rho\,\mathrm{Re}\,c(\vec v_h^{\,n};\vec U_h,\vec w_h)
+\big(\eta(\widehat\varphi_h^{\,n+1})\mathbf \nabla(\vec U_h),\mathbf \nabla(\vec w_h)\big)
\\
&\qquad
+\Big(\frac{\eta(\widehat\varphi_h^{\,n+1})}{\kappa(\widehat\varphi_h^{\,n+1})}
(1-\widehat\varphi_h^{\,n+1})\vec U_h,\vec w_h\Big)=0
\qquad \forall\,\vec w_h\in \vec V_h.
\end{align*}
Choosing $\vec w_h=\vec U_h$ and using
\[
c(\vec v_h^{\,n};\vec U_h,\vec U_h)=0,
\]
we get
\[
\frac{\rho\,\mathrm{Re}}{\Delta t}\|\vec U_h\|_{L^2(\omega)}^2
+\int_\omega \eta(\widehat\varphi_h^{\,n+1})|\mathbf\nabla \vec U_h|^2\,dx
+\int_\omega \frac{\eta(\widehat\varphi_h^{\,n+1})}{\kappa(\widehat\varphi_h^{\,n+1})}
(1-\widehat\varphi_h^{\,n+1})|\vec U_h|^2\,dx=0.
\]
Since
\[
\eta(\widehat\varphi_h^{\,n+1})\ge \eta_{\min}>0,\qquad
\kappa(\widehat\varphi_h^{\,n+1})\ge \kappa_{\min}>0,\qquad
0\le 1-\widehat\varphi_h^{\,n+1}\le 1,
\]
all three terms are nonnegative. Hence $\vec U_h=0$, and the velocity subproblem is uniquely solvable.

Finally, suppose \(\tilde\pi_{h,1}^{\,n+1}\) and \(\tilde\pi_{h,2}^{\,n+1}\) are two solutions of
\eqref{eq:disc-press-CDA-final} in \(Q_h^0\), and define
\[
\Pi_h:=\tilde\pi_{h,1}^{\,n+1}-\tilde\pi_{h,2}^{\,n+1}\in Q_h^0.
\]
Then
\[
(\nabla \Pi_h,\nabla q_h)=0
\qquad \forall\,q_h\in Q_h^0.
\]
Taking \(q_h=\Pi_h\), we obtain
\[
\|\nabla \Pi_h\|_{L^2(\omega)}^2=0.
\]
Thus \(\Pi_h\) is constant on \(\omega\). Since \(\Pi_h\in Q_h^0\), it follows that \(\Pi_h=0\). Therefore the pressure correction is uniquely solvable.
\end{proof}

Before stating the stepwise estimate, we briefly place the argument in the
context of existing stability analyses for related phase-field fluid schemes.
For the Cahn--Hilliard--Navier--Stokes system, several distinct mechanisms have
been used to derive discrete stability. Han and Wang combined a convex-splitting
treatment of the Cahn--Hilliard part with a pressure-projection step for the
Navier--Stokes part and obtained a modified discrete energy law~\cite{HanWang2015}.
Shen and Yang constructed decoupled energy-stable schemes based on stabilization
and convex splitting~\cite{ShenYang2015}, while Chen and Shen developed a
decoupled fully discrete adaptive scheme that is unconditionally energy
stable~\cite{ChenShen2016}. Chen and Zhao proposed a second-order linear method
that preserves an energy dissipation law and whose linear systems are uniquely
solvable~\cite{ChenZhao2020}. More recently, Zhao and Han reformulated the
Cahn--Hilliard--Navier--Stokes system into a constraint gradient-flow form,
which leads to second-order decoupled energy-stable schemes in the original
variables~\cite{ZhaoHan2021}. In a broader gradient-flow setting, Shen, Xu, and
Yang introduced the scalar auxiliary variable approach, which has become a
standard device for constructing efficient energy-stable time
discretizations~\cite{ShenXuYang2018}. Motivated by these developments, we do
not seek here a closed discrete free-energy law for the split capped scheme in
Problem~\ref{prob:disc-NSCH}. Instead, by testing the Cahn--Hilliard substep
with the discrete chemical potential, the auxiliary-field substep with the
updated auxiliary variable, and the velocity substep with the updated velocity,
and then exploiting the skew-symmetry of the transport forms together with
Young's inequality, we obtain the following stepwise a priori estimate.

\begin{theorem}
\label{thm:stepwise-stability-capped}
Assume that $\eta,\nu,\kappa$ satisfy \eqref{eq:coeff-ass}, and let
\[
(\vec v_h^{\,n},\tilde\pi_h^{\,n},\varphi_h^{\,n},\xi_h^{\,n},\vec\zeta_h^{\,n})
\in \vec V_h\times Q_h^0\times V_h\times V_h\times \vec W_h,
\qquad n=0,1,\dots,N,
\]
be generated by Problem~\ref{prob:disc-NSCH}. For each $n$, define
\[
\widehat\varphi_h^{\,n}:=\mathcal T(\varphi_h^{\,n})
=\min\{1,\max\{0,\varphi_h^{\,n}\}\}.
\]

Set
\[
\mathcal E_h^n
:=
\frac{\rho\,\mathrm{Re}}{2}\|\vec v_h^{\,n}\|_{L^2(\omega)}^2
+\frac12\|\varphi_h^{\,n}\|_{L^2(\omega)}^2
+\frac12\|\vec\zeta_h^{\,n}\|_{L^2(\omega)}^2,
\]
\begin{align*}
\mathcal D_h^{n+1}
:={}&
\frac{\tau}{4\lambda}\|\xi_h^{\,n+1}\|_{L^2(\omega)}^2
+\frac12\big(\eta(\widehat\varphi_h^{\,n+1})\nabla \vec v_h^{\,n+1},
\nabla \vec v_h^{\,n+1}\big)
\\
&\quad
+\frac12\Big(
\frac{\eta(\widehat\varphi_h^{\,n+1})}{\kappa(\widehat\varphi_h^{\,n+1})}
(1-\widehat\varphi_h^{\,n+1})\vec v_h^{\,n+1},
\vec v_h^{\,n+1}\Big)
+\epsilon\big(\nu(\widehat\varphi_h^{\,n+1})\nabla \vec\zeta_h^{\,n+1},
\nabla \vec\zeta_h^{\,n+1}\big),
\end{align*}
and
\begin{align*}
\mathcal S_h^{n}
:={}&
\alpha_\phi^2\|I_H(\phi^{n}-\varphi_h^{\,n})\|_{L^2(\omega)}^2
+\alpha_\psi^2\|I_H(\vec\psi^{\,n}-\vec\zeta_h^{\,n})\|_{L^2(\omega)}^2
+\alpha_u^2\|I_H(\vec u^{\,n}-\vec v_h^{\,n})\|_{L^2(\omega)}^2
\\
&\quad
+\|f(\widehat\varphi_h^{\,n})\|_{L^2(\omega)}^2
+\|\nu'(\widehat\varphi_h^{\,n})\,|\nabla \vec\zeta_h^{\,n}|^2\|_{L^2(\omega)}^2
+\|2\tilde\pi_h^{\,n}-\tilde\pi_h^{\,n-1}\|_{L^2(\omega)}^2
\\
&\quad
+\|\xi_h^{\,n+1}\nabla\varphi_h^{\,n+1}\|_{L^2(\omega)}^2
+\|\nu'(\widehat\varphi_h^{\,n+1})|\nabla\vec\zeta_h^{\,n+1}|^2
\nabla\varphi_h^{\,n+1}\|_{L^2(\omega)}^2
\\
&\quad
+\|\nu(\widehat\varphi_h^{\,n+1})
(\nabla\vec\zeta_h^{\,n+1})^T\nabla\vec\zeta_h^{\,n+1}\|_{L^2(\omega)}^2 .
\end{align*}

Then there exist constants $C>0$ and $\Delta t_0>0$, depending only on
$\omega$, $\rho$, $\mathrm{Re}$, $\tau$, $\lambda$, $\epsilon$,
$\eta_{\min}$, $\eta_{\max}$, $\nu_{\max}$, and $\kappa_{\min}$,
such that for every $0<\Delta t\le \Delta t_0$ and every $n=0,\dots,N-1$,
\begin{equation}
\label{eq:stepwise-stab-capped}
\mathcal E_h^{n+1}
+\Delta t\,\mathcal D_h^{n+1}
\le
(1+C\Delta t)\mathcal E_h^n
+
C\Delta t\,\mathcal S_h^{n}.
\end{equation}
Consequently, if $T=N\Delta t$, then for every $m\le N$,
\begin{equation}
\label{eq:discrete-gronwall-bound-capped}
\mathcal E_h^m
+\Delta t\sum_{n=0}^{m-1}\mathcal D_h^{n+1}
\le
e^{CT}\Bigg(
\mathcal E_h^0
+
C\Delta t\sum_{n=0}^{m-1}\mathcal S_h^{n}
\Bigg).
\end{equation}
\end{theorem}

\begin{proof}
Throughout the proof, $C>0$ denotes a generic constant independent of $h$ and $\Delta t$.

We first estimate the Cahn--Hilliard step. Taking $\vartheta_h=\varphi_h^{\,n+1}$ in
\eqref{eq:disc-CH-weak-CDA-final}, using
\[
b_\phi(\vec v_h^{\,n};\varphi_h^{\,n+1},\varphi_h^{\,n+1})=0,
\]
and the identity
\[
2(a-b,a)=\|a\|^2-\|b\|^2+\|a-b\|^2,
\]
we obtain
\begin{align}
\frac{1}{2\Delta t}
\Big(
\|\varphi_h^{\,n+1}\|_{L^2(\omega)}^2-\|\varphi_h^{\,n}\|_{L^2(\omega)}^2
+\|\varphi_h^{\,n+1}-\varphi_h^{\,n}\|_{L^2(\omega)}^2
\Big)
+\tau(\nabla\xi_h^{\,n+1},\nabla\varphi_h^{\,n+1})
\notag\\
=
\alpha_\phi\big(I_H(\phi^n-\varphi_h^{\,n}),\varphi_h^{\,n+1}\big).
\label{eq:proof-CH-cap-1}
\end{align}
Next, choosing $\chi_h=\dfrac{\tau}{\lambda}\xi_h^{\,n+1}$ in the constitutive relation gives
\begin{align}
\frac{\tau}{\lambda}\|\xi_h^{\,n+1}\|_{L^2(\omega)}^2
-\tau(\nabla\varphi_h^{\,n+1},\nabla\xi_h^{\,n+1})
&=
\tau\gamma\big(f(\widehat\varphi_h^{\,n}),\xi_h^{\,n+1}\big)
\notag\\
&\quad
+\frac{\tau}{2\lambda}\big(\nu'(\widehat\varphi_h^{\,n})|\nabla\vec\zeta_h^{\,n}|^2,\xi_h^{\,n+1}\big).
\label{eq:proof-CH-cap-2}
\end{align}
Adding \eqref{eq:proof-CH-cap-1} and \eqref{eq:proof-CH-cap-2} yields
\begin{align}
\frac{1}{2\Delta t}
\Big(
\|\varphi_h^{\,n+1}\|_{L^2(\omega)}^2-\|\varphi_h^{\,n}\|_{L^2(\omega)}^2
+\|\varphi_h^{\,n+1}-\varphi_h^{\,n}\|_{L^2(\omega)}^2
\Big)
+\frac{\tau}{\lambda}\|\xi_h^{\,n+1}\|_{L^2(\omega)}^2
\notag\\
=
\alpha_\phi\big(I_H(\phi^n-\varphi_h^{\,n}),\varphi_h^{\,n+1}\big)
+\tau\gamma\big(f(\widehat\varphi_h^{\,n}),\xi_h^{\,n+1}\big)
+\frac{\tau}{2\lambda}\big(\nu'(\widehat\varphi_h^{\,n})|\nabla\vec\zeta_h^{\,n}|^2,\xi_h^{\,n+1}\big).
\label{eq:proof-CH-cap-3}
\end{align}
By Young's inequality,
\begin{align}
\alpha_\phi\big(I_H(\phi^n-\varphi_h^{\,n}),\varphi_h^{\,n+1}\big)
&\le
\frac12\|\varphi_h^{\,n+1}\|_{L^2(\omega)}^2
+
C\alpha_\phi^2\|I_H(\phi^n-\varphi_h^{\,n})\|_{L^2(\omega)}^2,
\label{eq:proof-CH-cap-4}
\\
\tau\gamma\big(f(\widehat\varphi_h^{\,n}),\xi_h^{\,n+1}\big)
&\le
\frac{\tau}{4\lambda}\|\xi_h^{\,n+1}\|_{L^2(\omega)}^2
+
C\|f(\widehat\varphi_h^{\,n})\|_{L^2(\omega)}^2,
\label{eq:proof-CH-cap-5}
\\
\frac{\tau}{2\lambda}\big(\nu'(\widehat\varphi_h^{\,n})|\nabla\vec\zeta_h^{\,n}|^2,\xi_h^{\,n+1}\big)
&\le
\frac{\tau}{4\lambda}\|\xi_h^{\,n+1}\|_{L^2(\omega)}^2
+
C\|\nu'(\widehat\varphi_h^{\,n})|\nabla\vec\zeta_h^{\,n}|^2\|_{L^2(\omega)}^2.
\label{eq:proof-CH-cap-6}
\end{align}
Substituting \eqref{eq:proof-CH-cap-4}--\eqref{eq:proof-CH-cap-6} into
\eqref{eq:proof-CH-cap-3}, we obtain
\begin{align}
\frac{1}{2\Delta t}
\Big(
\|\varphi_h^{\,n+1}\|_{L^2(\omega)}^2-\|\varphi_h^{\,n}\|_{L^2(\omega)}^2
+\|\varphi_h^{\,n+1}-\varphi_h^{\,n}\|_{L^2(\omega)}^2
\Big)
+\frac{\tau}{2\lambda}\|\xi_h^{\,n+1}\|_{L^2(\omega)}^2
\notag\\
\le
\frac12\|\varphi_h^{\,n+1}\|_{L^2(\omega)}^2
+
C\Big(
\alpha_\phi^2\|I_H(\phi^n-\varphi_h^{\,n})\|_{L^2(\omega)}^2
+\|f(\widehat\varphi_h^{\,n})\|_{L^2(\omega)}^2
+\|\nu'(\widehat\varphi_h^{\,n})|\nabla\vec\zeta_h^{\,n}|^2\|_{L^2(\omega)}^2
\Big).
\label{eq:proof-CH-cap-final}
\end{align}

We now estimate the auxiliary-field step. Taking $\vec\theta_h=\vec\zeta_h^{\,n+1}$ in
\eqref{eq:disc-psi-weak-CDA-final} and using
\[
b_\psi(\vec v_h^{\,n};\vec\zeta_h^{\,n+1},\vec\zeta_h^{\,n+1})=0,
\]
we obtain
\begin{align}
\frac{1}{2\Delta t}
\Big(
\|\vec\zeta_h^{\,n+1}\|_{L^2(\omega)}^2-\|\vec\zeta_h^{\,n}\|_{L^2(\omega)}^2
+\|\vec\zeta_h^{\,n+1}-\vec\zeta_h^{\,n}\|_{L^2(\omega)}^2
\Big)
+\epsilon\big(\nu(\widehat\varphi_h^{\,n+1})\nabla\vec\zeta_h^{\,n+1},\nabla\vec\zeta_h^{\,n+1}\big)
\notag\\
=
\alpha_\psi\big(I_H(\vec\psi^{\,n}-\vec\zeta_h^{\,n}),\vec\zeta_h^{\,n+1}\big).
\label{eq:proof-psi-cap-1}
\end{align}
By Young's inequality,
\begin{equation}
\alpha_\psi\big(I_H(\vec\psi^{\,n}-\vec\zeta_h^{\,n}),\vec\zeta_h^{\,n+1}\big)
\le
\frac12\|\vec\zeta_h^{\,n+1}\|_{L^2(\omega)}^2
+
C\alpha_\psi^2\|I_H(\vec\psi^{\,n}-\vec\zeta_h^{\,n})\|_{L^2(\omega)}^2.
\label{eq:proof-psi-cap-2}
\end{equation}
Hence
\begin{align}
\frac{1}{2\Delta t}
\Big(
\|\vec\zeta_h^{\,n+1}\|_{L^2(\omega)}^2-\|\vec\zeta_h^{\,n}\|_{L^2(\omega)}^2
+\|\vec\zeta_h^{\,n+1}-\vec\zeta_h^{\,n}\|_{L^2(\omega)}^2
\Big)
+\epsilon\big(\nu(\widehat\varphi_h^{\,n+1})\nabla\vec\zeta_h^{\,n+1},\nabla\vec\zeta_h^{\,n+1}\big)
\notag\\
\le
\frac12\|\vec\zeta_h^{\,n+1}\|_{L^2(\omega)}^2
+
C\alpha_\psi^2\|I_H(\vec\psi^{\,n}-\vec\zeta_h^{\,n})\|_{L^2(\omega)}^2.
\label{eq:proof-psi-cap-final}
\end{align}

Next we estimate the velocity step. Taking $\vec w_h=\vec v_h^{\,n+1}$ in
\eqref{eq:disc-NS-weak-CDA-final}, using
\[
c(\vec v_h^{\,n};\vec v_h^{\,n+1},\vec v_h^{\,n+1})=0,
\]
and again the identity $2(a-b,a)=\|a\|^2-\|b\|^2+\|a-b\|^2$, we obtain
\begin{align}
\frac{\rho\,\mathrm{Re}}{2\Delta t}
\Big(
\|\vec v_h^{\,n+1}\|_{L^2(\omega)}^2-\|\vec v_h^{\,n}\|_{L^2(\omega)}^2
+\|\vec v_h^{\,n+1}-\vec v_h^{\,n}\|_{L^2(\omega)}^2
\Big)
\notag\\
+\big(\eta(\widehat\varphi_h^{\,n+1})\nabla\vec v_h^{\,n+1},\nabla\vec v_h^{\,n+1}\big)
+\Big(
\frac{\eta(\widehat\varphi_h^{\,n+1})}{\kappa(\widehat\varphi_h^{\,n+1})}
(1-\widehat\varphi_h^{\,n+1})\vec v_h^{\,n+1},\vec v_h^{\,n+1}\Big)
\notag\\
=
\big(2\tilde\pi_h^{\,n}-\tilde\pi_h^{\,n-1},\nabla\!\cdot \vec v_h^{\,n+1}\big)
+\big(\xi_h^{\,n+1}\nabla\varphi_h^{\,n+1},\vec v_h^{\,n+1}\big)
\notag\\
\quad
-\frac12\big(\nu'(\widehat\varphi_h^{\,n+1})|\nabla\vec\zeta_h^{\,n+1}|^2\nabla\varphi_h^{\,n+1},\vec v_h^{\,n+1}\big)
\notag\\
\quad
-\big(\nu(\widehat\varphi_h^{\,n+1})(\nabla\vec\zeta_h^{\,n+1})^T\nabla\vec\zeta_h^{\,n+1},
\nabla \vec v_h^{\,n+1}\big)
+\alpha_u\big(I_H(\vec u^{\,n}-\vec v_h^{\,n}),\vec v_h^{\,n+1}\big).
\label{eq:proof-v-cap-1}
\end{align}

Since $\vec v_h^{\,n+1}\in \vec V_h\subset H_0^1(\omega)^2$, the Poincar\'e inequality gives
\[
\|\vec v_h^{\,n+1}\|_{L^2(\omega)}
\le C_P\|\nabla\vec v_h^{\,n+1}\|_{L^2(\omega)}.
\]
Also,
\[
\eta(\widehat\varphi_h^{\,n+1})\ge \eta_{\min}>0,
\qquad
0\le 1-\widehat\varphi_h^{\,n+1}\le 1.
\]
Therefore, by Cauchy--Schwarz, Poincar\'e, and Young's inequality,
\begin{align}
\big(2\tilde\pi_h^{\,n}-\tilde\pi_h^{\,n-1},\nabla\!\cdot \vec v_h^{\,n+1}\big)
&\le
\frac14\big(\eta(\widehat\varphi_h^{\,n+1})\nabla\vec v_h^{\,n+1},
\nabla\vec v_h^{\,n+1}\big)
+
C\|2\tilde\pi_h^{\,n}-\tilde\pi_h^{\,n-1}\|_{L^2(\omega)}^2,
\label{eq:proof-v-cap-2}
\\
\big(\xi_h^{\,n+1}\nabla\varphi_h^{\,n+1},\vec v_h^{\,n+1}\big)
&\le
\frac14\big(\eta(\widehat\varphi_h^{\,n+1})\nabla\vec v_h^{\,n+1},
\nabla\vec v_h^{\,n+1}\big)
+
C\|\xi_h^{\,n+1}\nabla\varphi_h^{\,n+1}\|_{L^2(\omega)}^2,
\label{eq:proof-v-cap-3}
\\
\frac12\big|\big(\nu'(\widehat\varphi_h^{\,n+1})|\nabla\vec\zeta_h^{\,n+1}|^2
\nabla\varphi_h^{\,n+1},\vec v_h^{\,n+1}\big)\big|
&\le
\frac14\big(\eta(\widehat\varphi_h^{\,n+1})\nabla\vec v_h^{\,n+1},
\nabla\vec v_h^{\,n+1}\big)
\notag\\
&\quad
+
C\|\nu'(\widehat\varphi_h^{\,n+1})|\nabla\vec\zeta_h^{\,n+1}|^2
\nabla\varphi_h^{\,n+1}\|_{L^2(\omega)}^2,
\label{eq:proof-v-cap-4}
\\
\big|\big(\nu(\widehat\varphi_h^{\,n+1})(\nabla\vec\zeta_h^{\,n+1})^T\nabla\vec\zeta_h^{\,n+1},
\nabla \vec v_h^{\,n+1}\big)\big|
&\le
\frac14\big(\eta(\widehat\varphi_h^{\,n+1})\nabla\vec v_h^{\,n+1},
\nabla\vec v_h^{\,n+1}\big)
\notag\\
&\quad
+
C\|\nu(\widehat\varphi_h^{\,n+1})(\nabla\vec\zeta_h^{\,n+1})^T
\nabla\vec\zeta_h^{\,n+1}\|_{L^2(\omega)}^2,
\label{eq:proof-v-cap-5}
\\
\alpha_u\big(I_H(\vec u^{\,n}-\vec v_h^{\,n}),\vec v_h^{\,n+1}\big)
&\le
\frac14\big(\eta(\widehat\varphi_h^{\,n+1})\nabla\vec v_h^{\,n+1},
\nabla\vec v_h^{\,n+1}\big)
+
C\alpha_u^2\|I_H(\vec u^{\,n}-\vec v_h^{\,n})\|_{L^2(\omega)}^2.
\label{eq:proof-v-cap-6}
\end{align}
Substituting \eqref{eq:proof-v-cap-2}--\eqref{eq:proof-v-cap-6} into
\eqref{eq:proof-v-cap-1}, we obtain
\begin{align}
\frac{\rho\,\mathrm{Re}}{2\Delta t}
\Big(
\|\vec v_h^{\,n+1}\|_{L^2(\omega)}^2-\|\vec v_h^{\,n}\|_{L^2(\omega)}^2
+\|\vec v_h^{\,n+1}-\vec v_h^{\,n}\|_{L^2(\omega)}^2
\Big)
\notag\\
+\frac12\big(\eta(\widehat\varphi_h^{\,n+1})\nabla\vec v_h^{\,n+1},
\nabla\vec v_h^{\,n+1}\big)
+\Big(
\frac{\eta(\widehat\varphi_h^{\,n+1})}{\kappa(\widehat\varphi_h^{\,n+1})}
(1-\widehat\varphi_h^{\,n+1})\vec v_h^{\,n+1},\vec v_h^{\,n+1}\Big)
\notag\\
\le
C\Big(
\|2\tilde\pi_h^{\,n}-\tilde\pi_h^{\,n-1}\|_{L^2(\omega)}^2
+\|\xi_h^{\,n+1}\nabla\varphi_h^{\,n+1}\|_{L^2(\omega)}^2
\notag\\
\qquad
+\|\nu'(\widehat\varphi_h^{\,n+1})|\nabla\vec\zeta_h^{\,n+1}|^2
\nabla\varphi_h^{\,n+1}\|_{L^2(\omega)}^2
\notag\\
\qquad
+\|\nu(\widehat\varphi_h^{\,n+1})(\nabla\vec\zeta_h^{\,n+1})^T
\nabla\vec\zeta_h^{\,n+1}\|_{L^2(\omega)}^2
+\alpha_u^2\|I_H(\vec u^{\,n}-\vec v_h^{\,n})\|_{L^2(\omega)}^2
\Big).
\label{eq:proof-v-cap-final}
\end{align}

We now add \eqref{eq:proof-CH-cap-final}, \eqref{eq:proof-psi-cap-final}, and
\eqref{eq:proof-v-cap-final}, and multiply by $\Delta t$. Recalling the definitions of
$\mathcal E_h^n$, $\mathcal D_h^{n+1}$, and $\mathcal S_h^n$, we infer
\begin{equation}
\label{eq:proof-preabsorb-cap}
\mathcal E_h^{n+1}-\mathcal E_h^n
+\Delta t\,\mathcal D_h^{n+1}
\le
C\Delta t\,\mathcal E_h^{n+1}
+
C\Delta t\,\mathcal S_h^n.
\end{equation}
Indeed, the terms
\[
\frac{\Delta t}{2}\|\varphi_h^{\,n+1}\|_{L^2(\omega)}^2
\qquad\text{and}\qquad
\frac{\Delta t}{2}\|\vec\zeta_h^{\,n+1}\|_{L^2(\omega)}^2
\]
arising from \eqref{eq:proof-CH-cap-final} and \eqref{eq:proof-psi-cap-final}
are bounded by $C\Delta t\,\mathcal E_h^{n+1}$.

Choose $\Delta t_0>0$ so that $C\Delta t_0\le \frac12$. Then for every
$0<\Delta t\le \Delta t_0$, the term $C\Delta t\,\mathcal E_h^{n+1}$ can be absorbed
into the left-hand side of \eqref{eq:proof-preabsorb-cap}, yielding
\[
\mathcal E_h^{n+1}
+\Delta t\,\mathcal D_h^{n+1}
\le
(1+C\Delta t)\mathcal E_h^n
+
C\Delta t\,\mathcal S_h^n,
\]
which is exactly \eqref{eq:stepwise-stab-capped}. Iterating this estimate and applying
the discrete Gr\"onwall lemma proves \eqref{eq:discrete-gronwall-bound-capped}.
\end{proof}

\begin{remark}
Theorem~\ref{thm:stepwise-stability-capped} provides a stepwise stability bound
for the capped fully discrete scheme. In particular, it gives a useful a priori
control on the propagation of the discrete solution and supports the numerical
implementation employed in the experiments below.
\end{remark}

\begin{remark}
Related convergence and error estimates for diffuse-interface and
Cahn--Hilliard fluid schemes can be found in
\cite{Feng2006NSCH,FengHeLiu2007,DiegelWangWangWise2017,
LiuChenWangWise2017,ChenEtAl2022CHSD,ChenMaoShen2020,
ChenEtAl2024FHCHNS}.
These works place Theorem~\ref{thm:stepwise-stability-capped} in the broader
context of fully discrete phase-field approximations. For the capped scheme
considered here, the truncation$\widehat\varphi_h^{\,n}=\mathcal T(\varphi_h^{\,n})$
adds an extra layer to the analysis.
\end{remark}

%===========================================================
\section{Numerical experiments for the NSCH system with continuous data assimilation}
\label{sec:numerics}

We present a set of two-dimensional numerical experiments illustrating the behavior of the coupled NSCH--$\psi$
system and the effectiveness of the continuous data assimilation (CDA) nudging terms. All simulations are performed
with \texttt{FreeFEM++} using the splitting scheme described in Problem~\ref{prob:disc-NSCH}.

%-----------------------------------------------------------

We take $\omega=(0,1)\times(0,1)$. Let $\mathcal T_H$ denote the fine triangulation used for the forward solves, and
let $\mathcal T_h$ be a uniformly coarsened mesh used to represent the observational resolution. In the implementation,
the observation operator $I_H$ is realized by nodal interpolation onto $\mathcal T_h$ and then evaluating the coarse-grid field on $\mathcal T_H$.

We use continuous $P_2$ elements on $\mathcal T_H$ for $(\phi,\mu)$ and for each component of $\vec\psi$ and $\vec u$,
and continuous $P_1$ elements for the pressure variable.

Let us choose $\nu(\phi)=\lambda_e(1-\phi)$. We report the following components:
\[
E_{\rm kin}(t)=\frac{\rho}{2}\int_\omega |\vec u|^2\,dx,\qquad
E_{\rm mix}(t)=\int_\omega\Big(\frac{\lambda}{2}|\nabla\phi|^2+4\lambda\gamma\,\phi^2(1-\phi)^2\Big)\,dx,
\]
and an elastic contribution for $\vec\psi$ weighted by $1-\phi$,
\[
E_{\rm el}(t)=\frac{\lambda_e}{2}\int_\omega (1-\phi)\,|\nabla\vec\psi|^2\,dx,
\qquad
E_{\rm tot}=E_{\rm kin}+E_{\rm mix}+E_{\rm el}
\].

We track the $L^2$-errors between the reference state and the CDA state:
\[
e_u(t)=\|\vec u(t)-\vec v(t)\|_{L^2(\omega)},\quad
e_\phi(t)=\|\phi(t)-\varphi(t)\|_{L^2(\omega)},\quad
e_\psi(t)=\|\vec\psi(t)-\vec\zeta(t)\|_{L^2(\omega)}.
\]
For visualization, we plot the errors on a logarithmic scale versus time. We also plot an $L^2$-difference for the
pressure iterate, but we emphasize that pressure is only defined up to an additive constant and is obtained here
via a pressure-correction update. Thus, its error can exhibit persistent oscillations even when the primary variables
synchronize.

%-----------------------------------------------------------
\paragraph{Test 1: performance demonstration}
\label{subsec:da-test}

We consider the square domain $\omega=(0,1)^2$ on a $64\times64$ on the fine mesh $\mathcal T_H$ and $32\times32$ on the coarse mesh $\mathcal T_h$. We solve the coupled system with the no-slip condition
$\vec u=\vec 0$ on $\partial\omega$, homogeneous Neumann conditions for $\phi$, $\mu$, and $\vec\psi$, and no external forcing.
For this test, set
$\Delta t=10^{-2}$, $T=100$, $\mathrm{Re}=3000$, $\tau=10^{-4}$, $\epsilon=5\times 10^{-3}$,
$\rho=1$, $\lambda=1$, $\gamma=1$, $\lambda_e=0.5$, $k_{\mathrm{per}}=1$,
and $\eta(\phi)=\eta_p\phi+\eta_f(1-\phi)$ with $(\eta_f,\eta_p)=(10^{-2},10^{-1})$.
The observation operator $I_H$ is implemented by nodal interpolation of fine-grid fields to a uniformly coarsened mesh
$\mathcal T_H$.

The reference initial conditions are
\[
\phi_0(x,y)=\frac12+\frac12\tanh\!\Big(\frac{R_0-\sqrt{(x-x_0)^2+(y-y_0)^2}}{\sqrt{2\gamma}}\Big),
\qquad (x_0,y_0)=(0.5,0.5),\quad R_0=0.18,
\]
together with
\[
\vec\psi_0(x,y)=(-y,\ x),\qquad \vec u_0(x,y)\equiv (0,0),\qquad p_0(x,y)\equiv 0.
\]
For the assimilated run, we intentionally start from mismatched initial data:
\[
\varphi_0(x,y)=\frac12+\frac12\tanh\!\Big(\frac{R_0^{\mathrm{da}}-\sqrt{(x-x_0^{\mathrm{da}})^2+(y-y_0^{\mathrm{da}})^2}}{\sqrt{2\gamma}}\Big),
\qquad (x_0^{\mathrm{da}},y_0^{\mathrm{da}})=(0.65,0.35),\quad R_0^{\mathrm{da}}=0.22,
\]
\[
\vec\zeta_0(x,y)=\bigl(-(\cos\theta\,y+\sin\theta\,x),\ \cos\theta\,x-\sin\theta\,y\bigr),
\qquad \theta=0.6,
\]
\[
\vec v_0(x,y)=\bigl(0.20\sin(\pi y),\ 0\bigr),\qquad \pi_0(x,y)\equiv 0.
\]

We first report the energy evolution of the reference solution. In this parameter regime, the total energy decreases, and the solution relaxes toward a steady configuration.

\begin{figure}[H]
\centering
\includegraphics[width=0.48\textwidth]{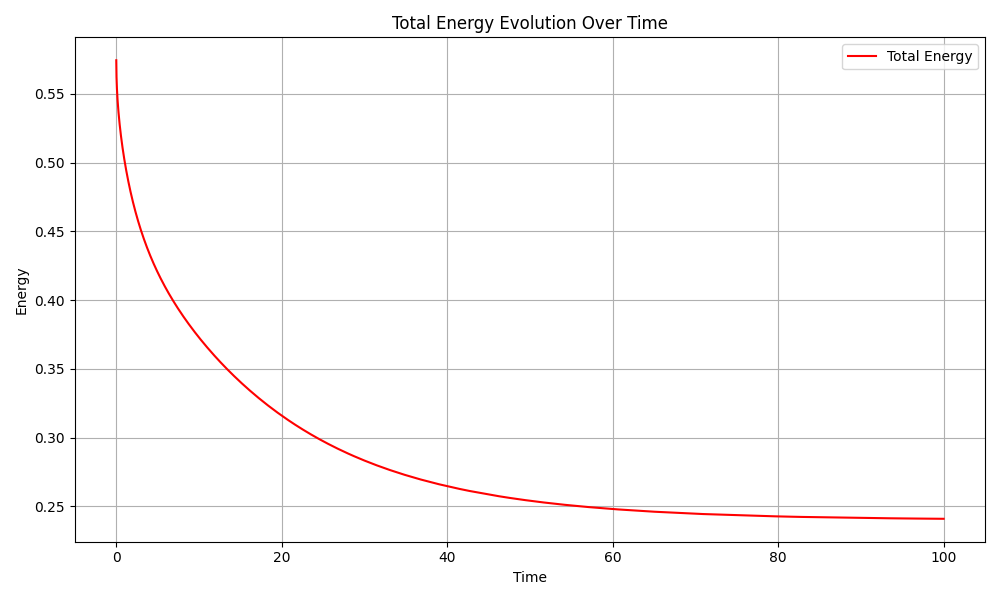}\hfill
\includegraphics[width=0.48\textwidth]{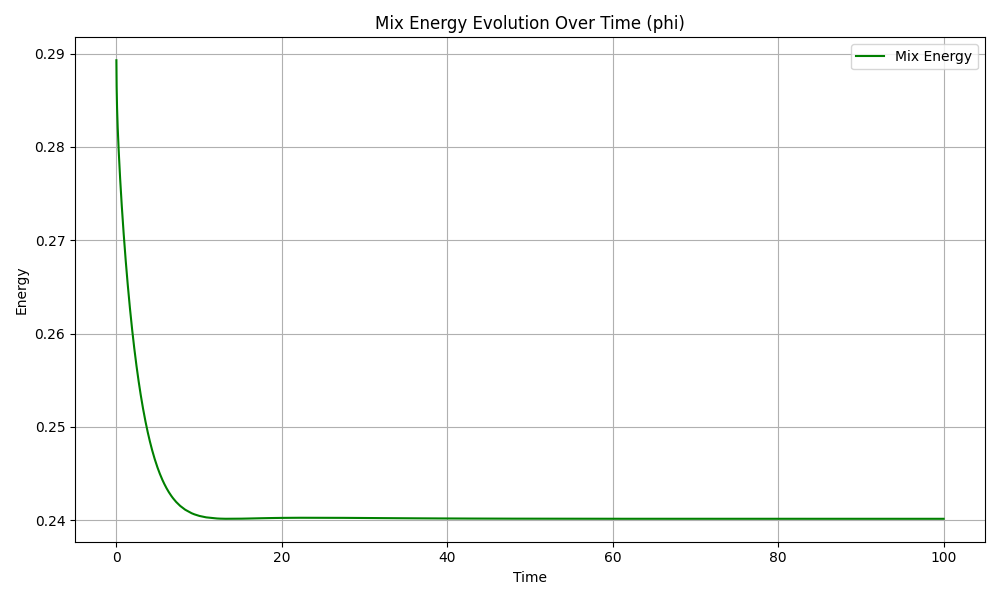}\\[0.5ex]
\includegraphics[width=0.48\textwidth]{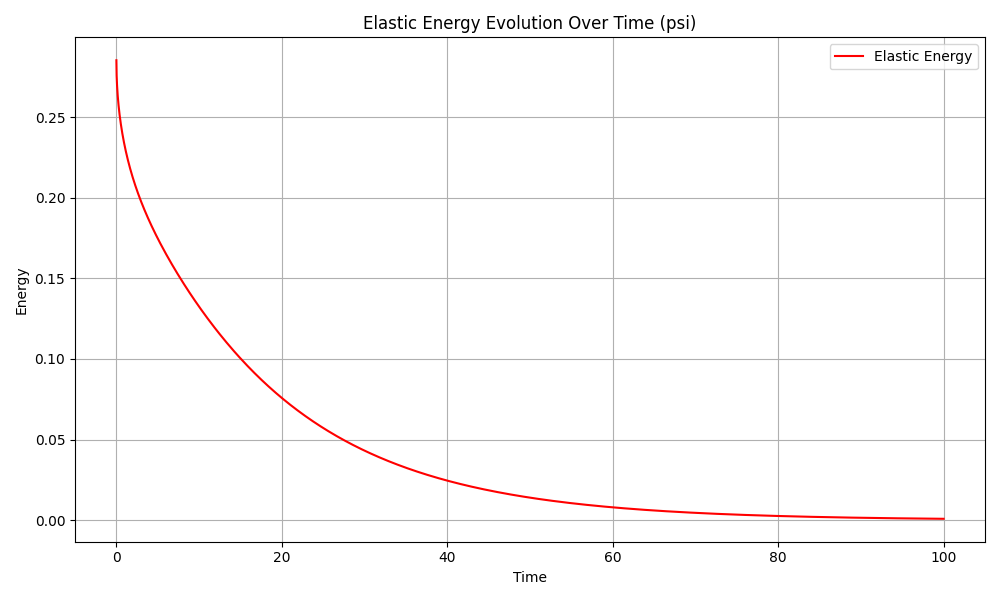}\hfill
\includegraphics[width=0.48\textwidth]{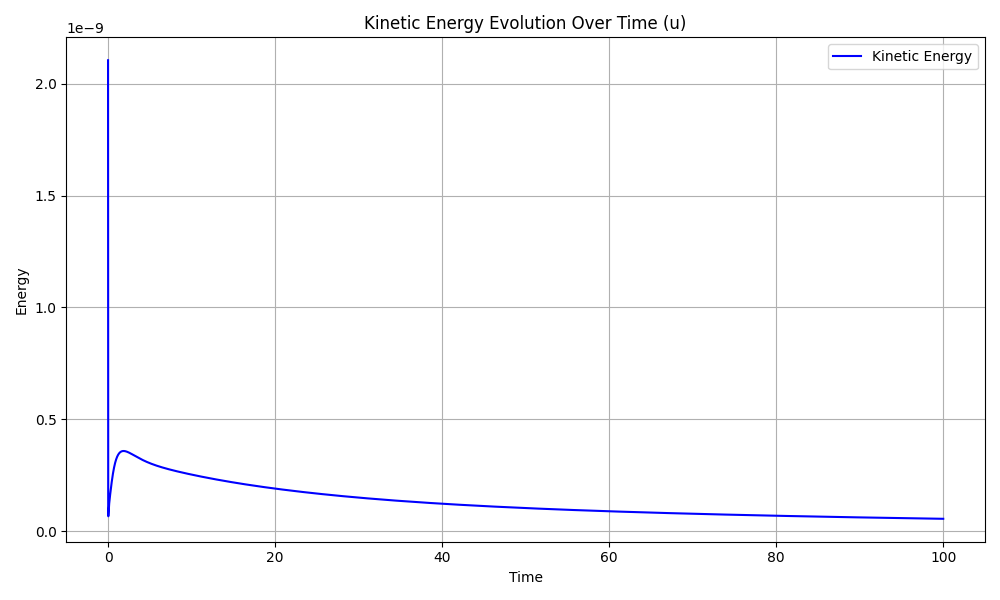}
\caption{Reference solution: total energy and its components.}
\label{fig:da-energy}
\end{figure}

We next show the synchronization errors for the CDA run. The $L^2$-errors
$e_u(t)=\|\vec u-\vec v\|_{L^2(\omega)}$, $e_\phi(t)=\|\phi-\varphi\|_{L^2(\omega)}$, and
$e_\psi(t)=\|\vec\psi-\vec\zeta\|_{L^2(\omega)}$ decay rapidly and then plateau at a small discretization-limited level.

\begin{figure}[H]
\centering
\includegraphics[width=0.32\textwidth]{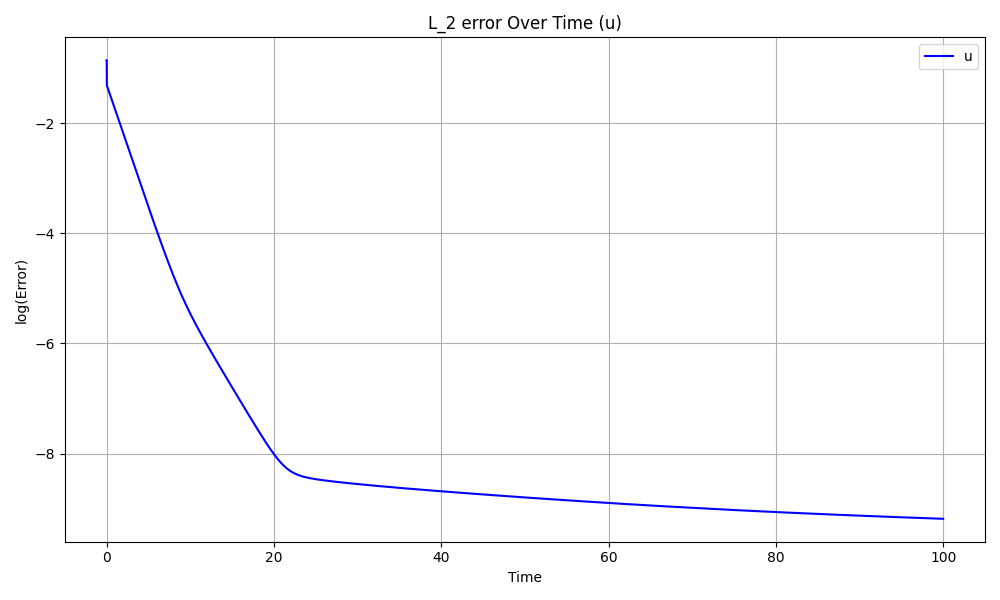}\hfill
\includegraphics[width=0.32\textwidth]{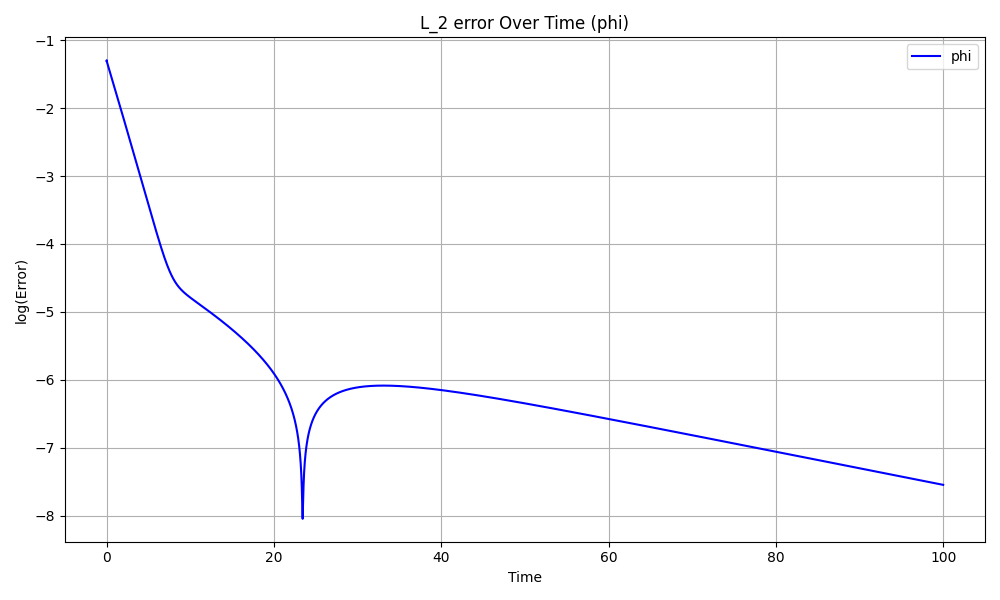}\hfill
\includegraphics[width=0.32\textwidth]{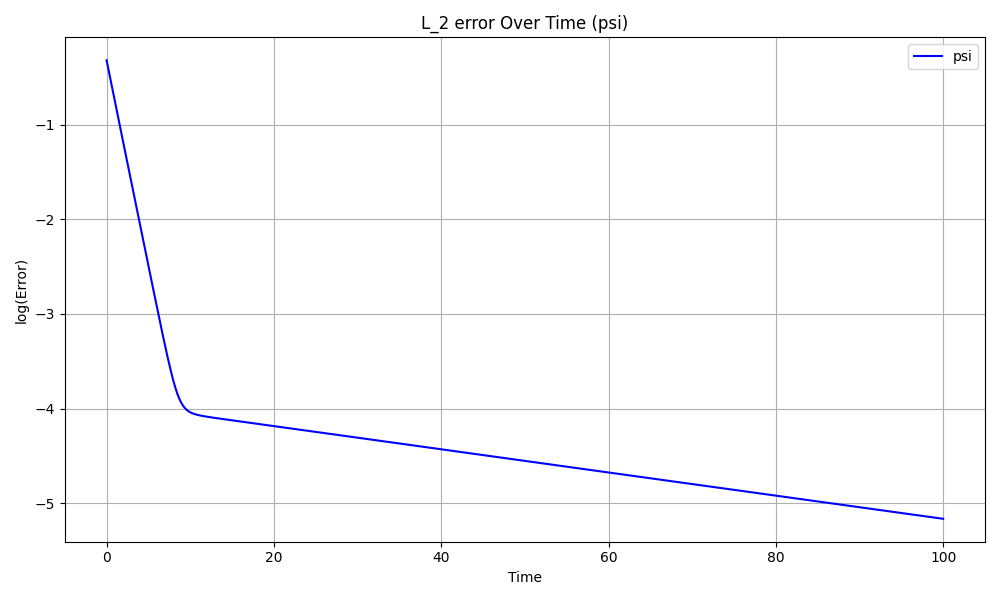}
\caption{CDA run (Test~1, $\alpha_u=\alpha_\phi=\alpha_\psi=1$): logarithmic $L^2$-errors in $\vec u$, $\phi$, and $\vec\psi$.}
\label{fig:da-error}
\end{figure}

Then, we set the nudging parameter to zero to turn off the data assimilation. The mismatch persists, and the errors do not decay to zero.

\begin{figure}[H]
\centering
\includegraphics[width=0.32\textwidth]{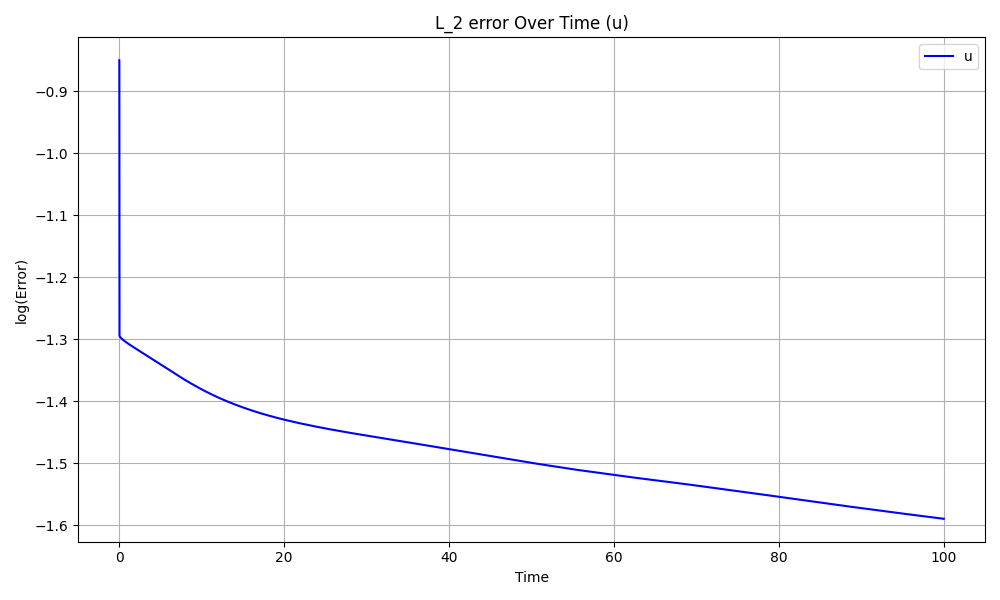}\hfill
\includegraphics[width=0.32\textwidth]{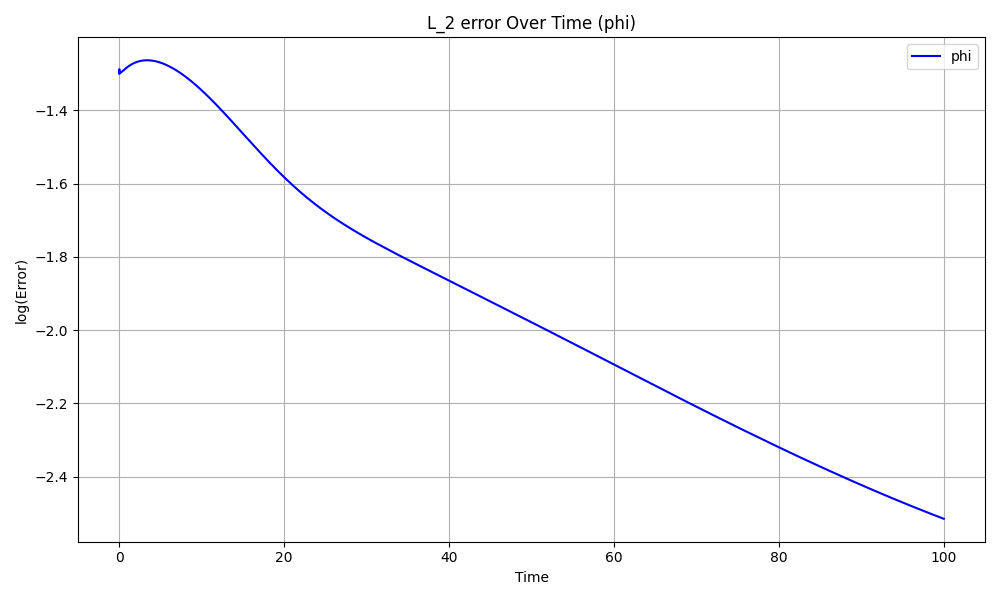}\hfill
\includegraphics[width=0.32\textwidth]{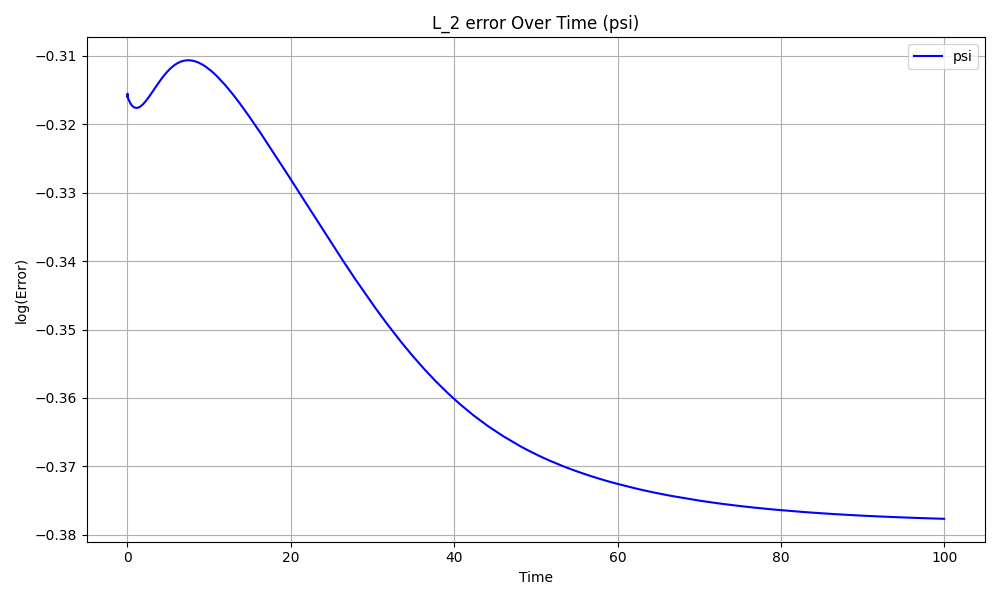}
\caption{No-nudging run (Test~1, $\alpha_u=\alpha_\phi=\alpha_\psi=0$): logarithmic $L^2$-errors in $\vec u$, $\phi$, and $\vec\psi$.}
\label{fig:noda-error}
\end{figure}

We finally compare the $\varphi$  snapshots. With CDA nudging the assimilated state $\varphi$ recovers the correct
interface geometry and location. The solution relaxes to a different steady state determined by its own
initial condition without nudging.

\begin{figure}[H]
\centering
\setlength{\tabcolsep}{1pt}
\renewcommand{\arraystretch}{0}
\scriptsize
\begin{tabular}{@{}c*{7}{c}@{}}

\raisebox{0.4\height}{\makebox[0pt][c]{\rotatebox{90}{\textbf{Reference}}}} &
\includegraphics[width=0.155\textwidth]{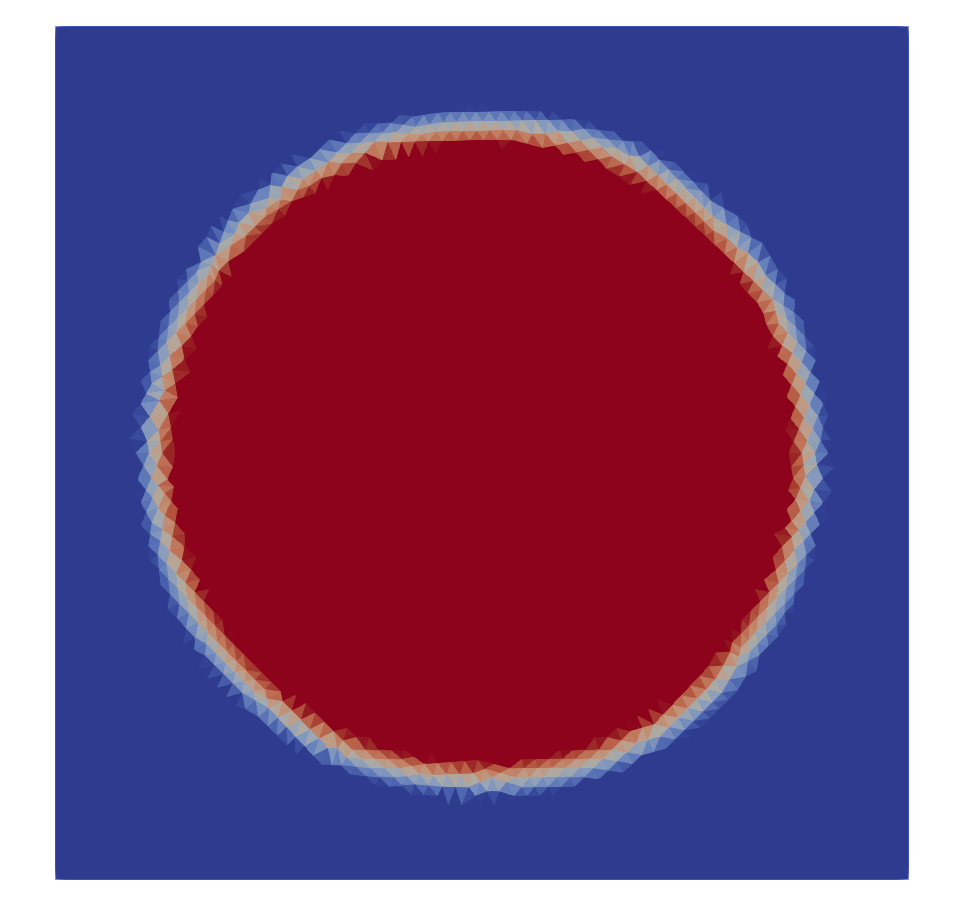} &
\includegraphics[width=0.155\textwidth]{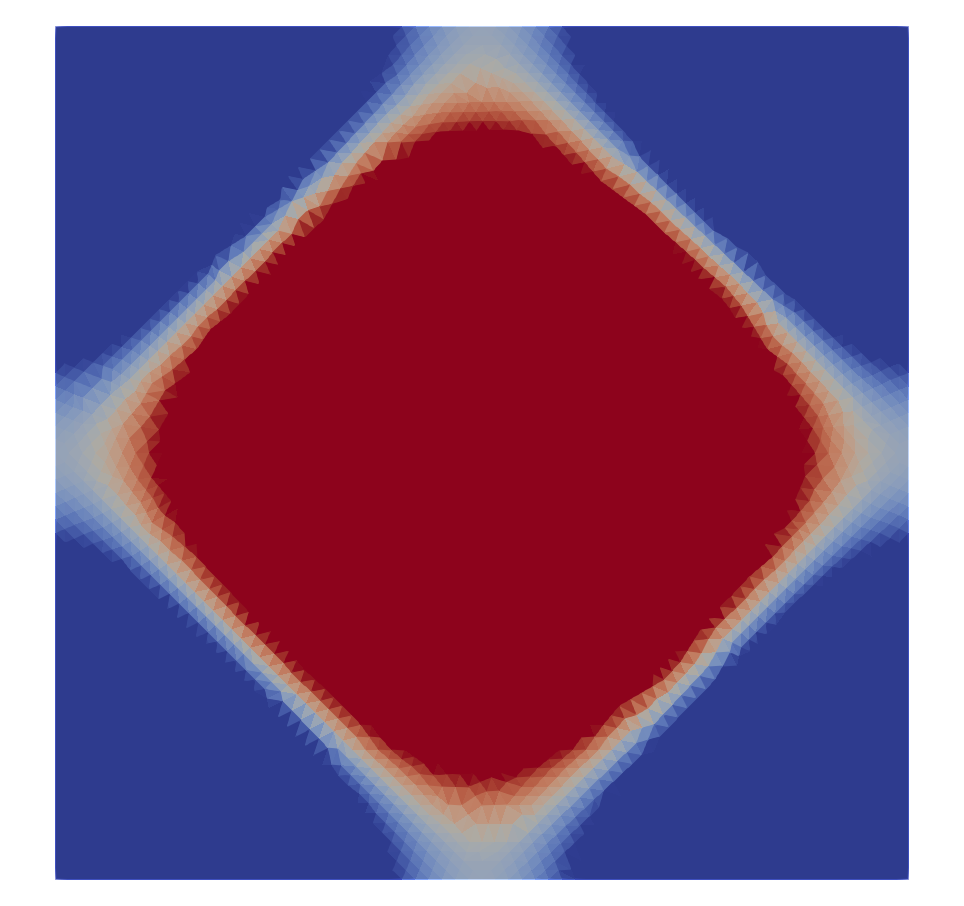} &
\includegraphics[width=0.155\textwidth]{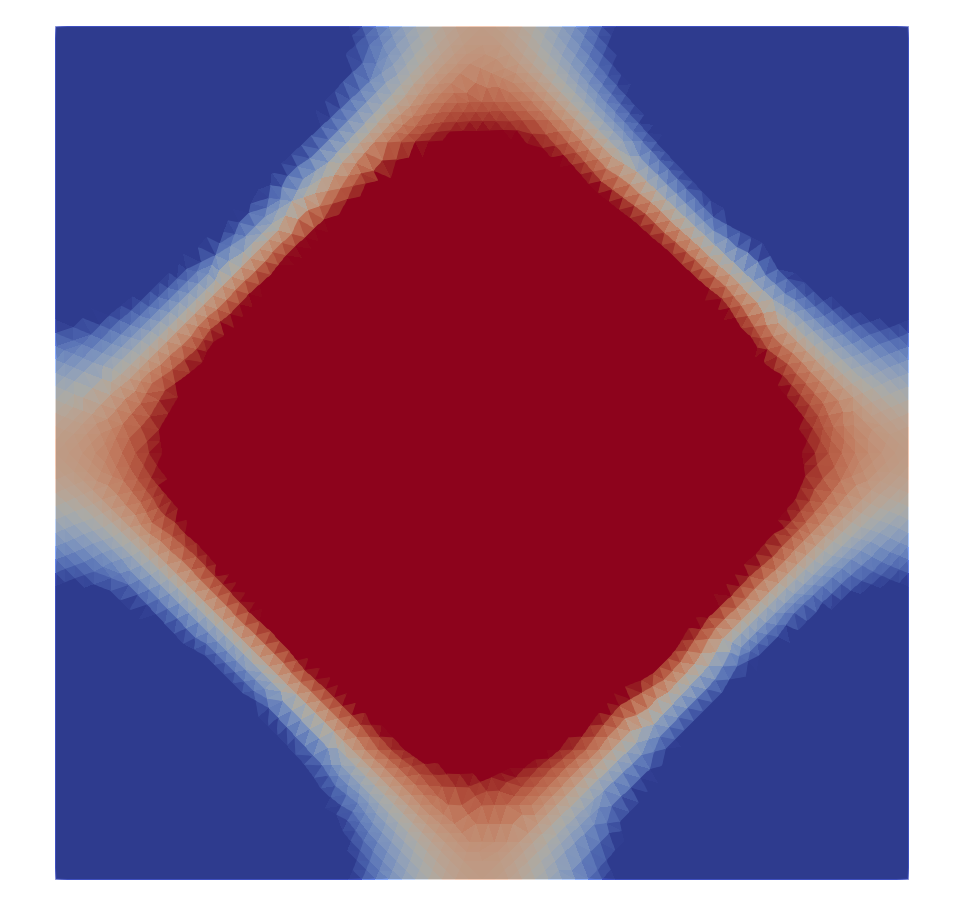} &
\includegraphics[width=0.155\textwidth]{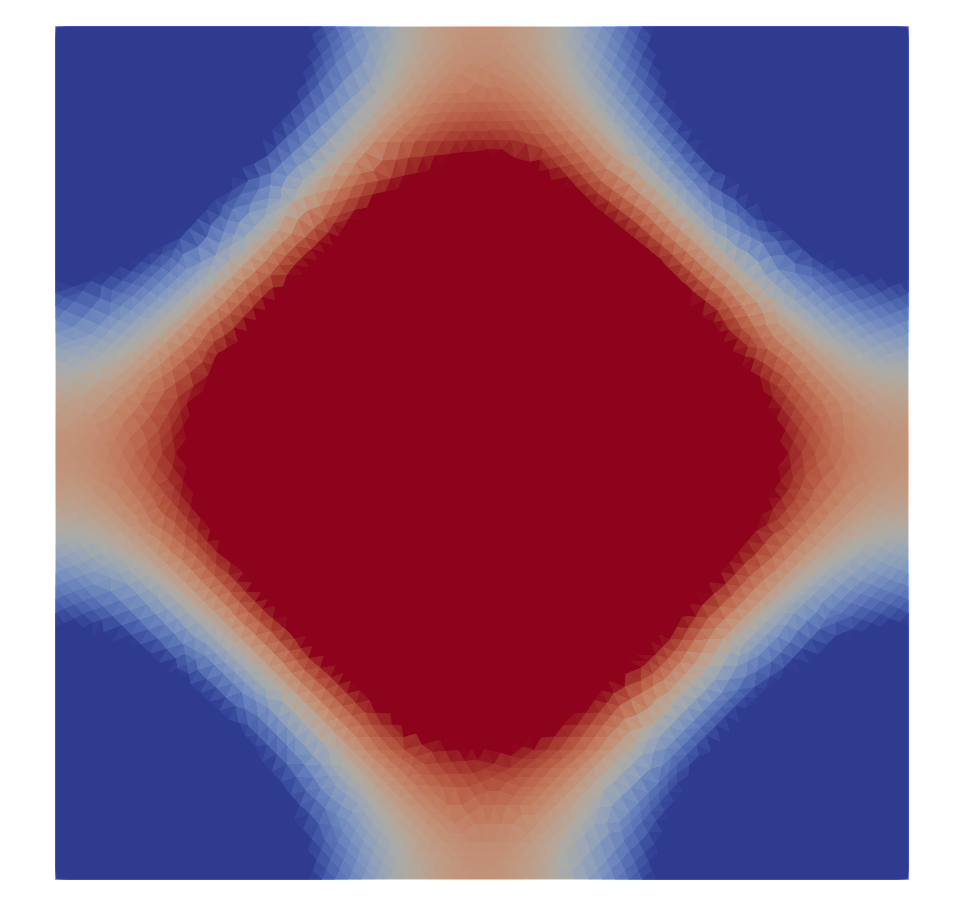} &
\includegraphics[width=0.155\textwidth]{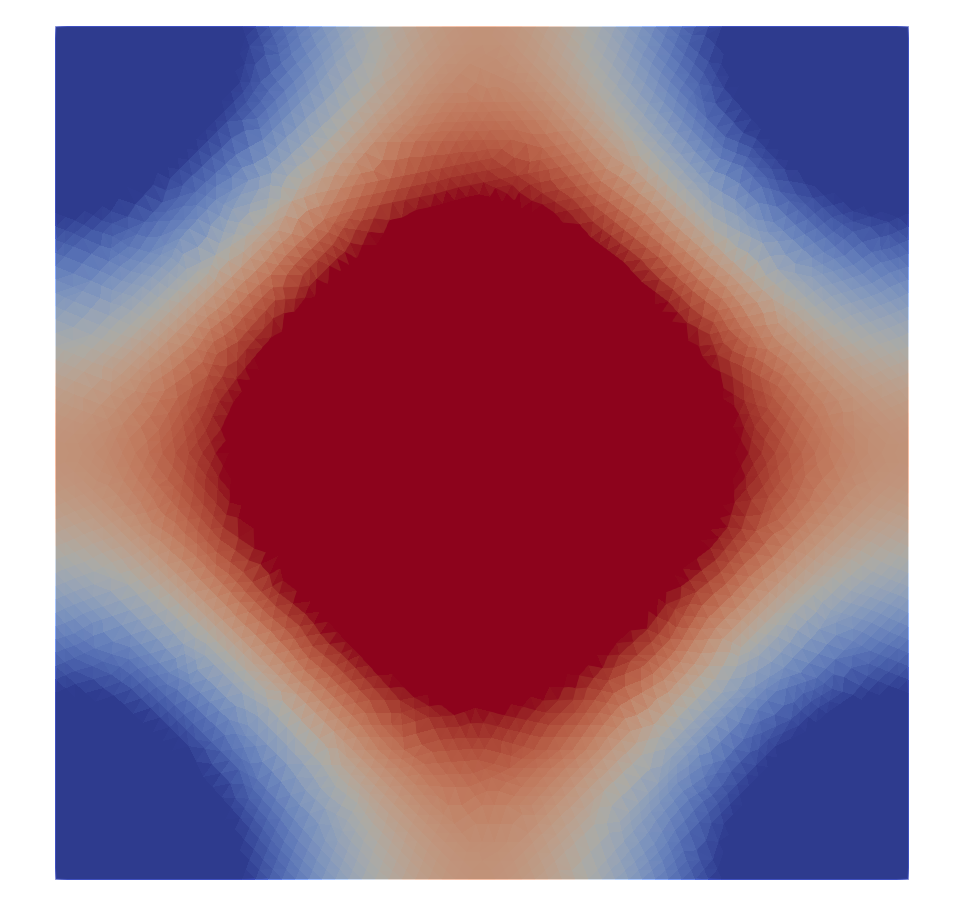} &
\includegraphics[width=0.155\textwidth]{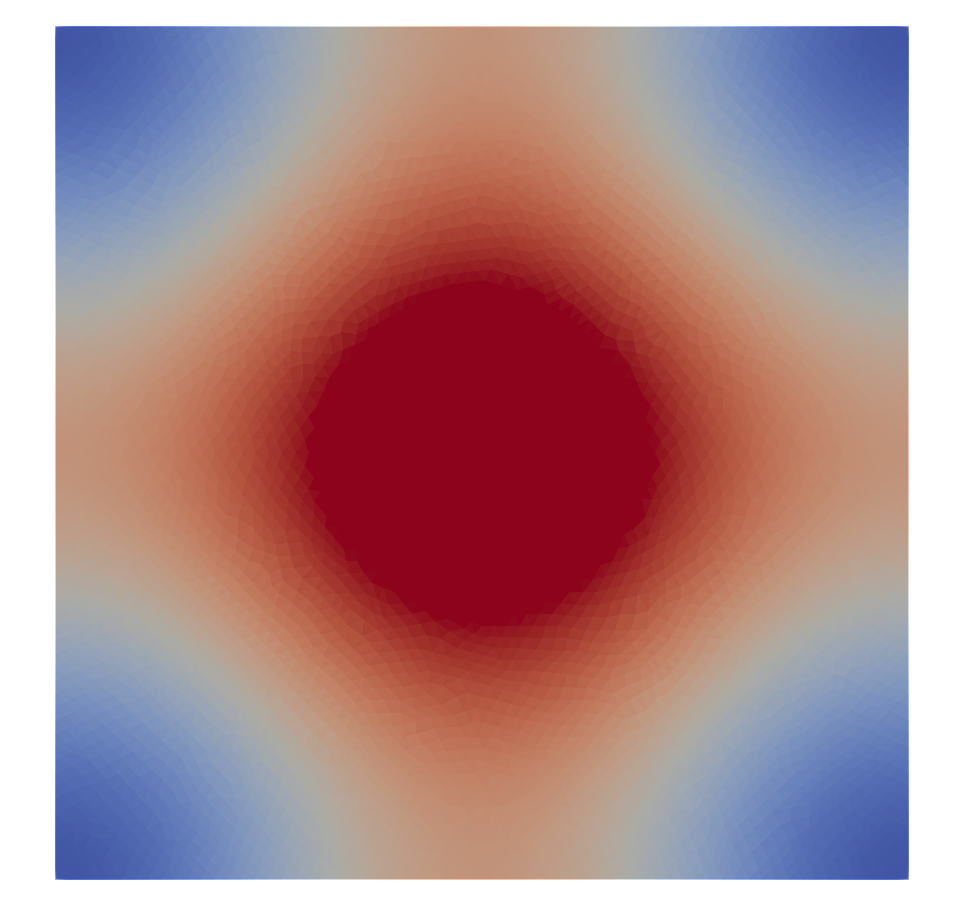} 
& \includegraphics[width=0.038\textwidth]{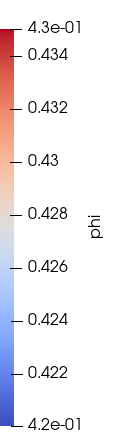}\\[2mm]

\raisebox{1.2\height}{\makebox[0pt][c]{\rotatebox{90}{\textbf{CDA}}}} &
\includegraphics[width=0.155\textwidth]{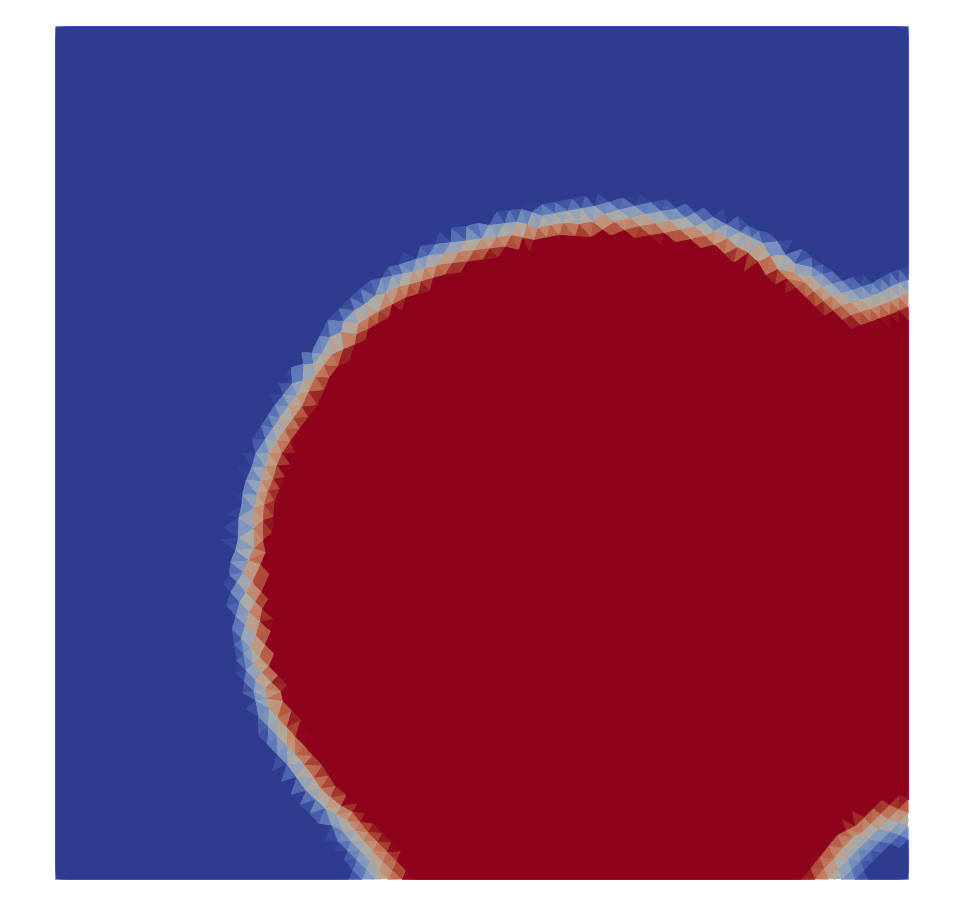} &
\includegraphics[width=0.155\textwidth]{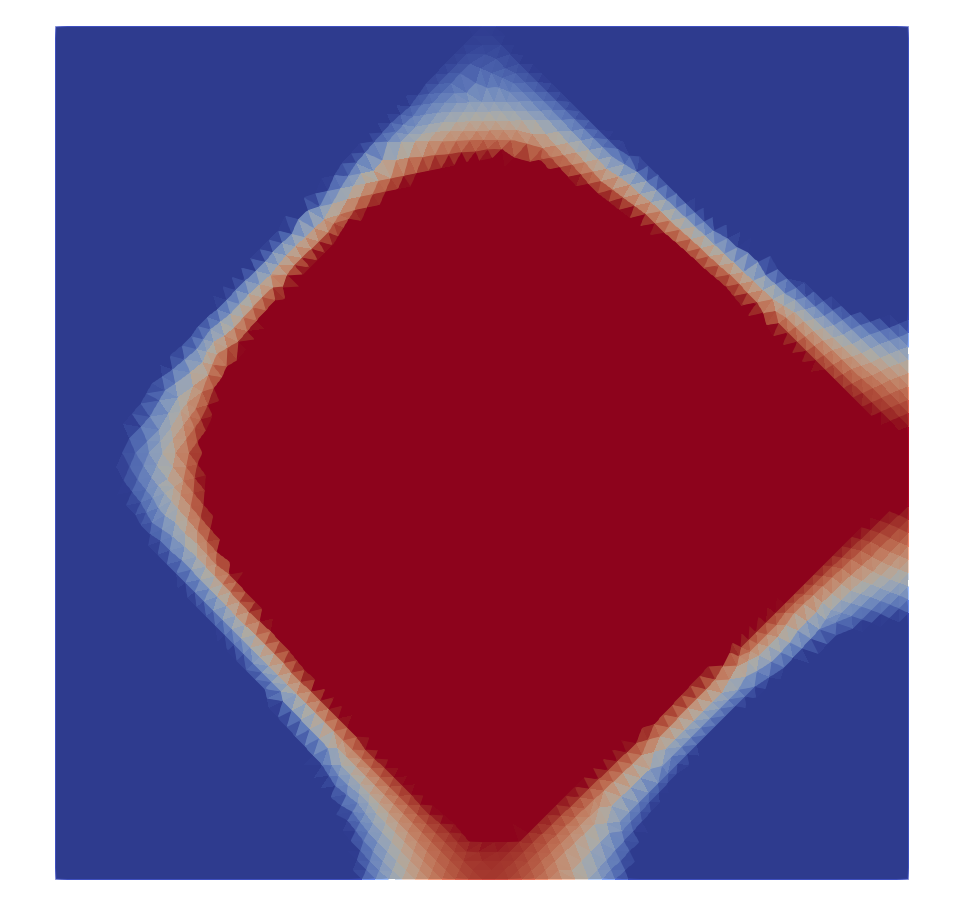} &
\includegraphics[width=0.155\textwidth]{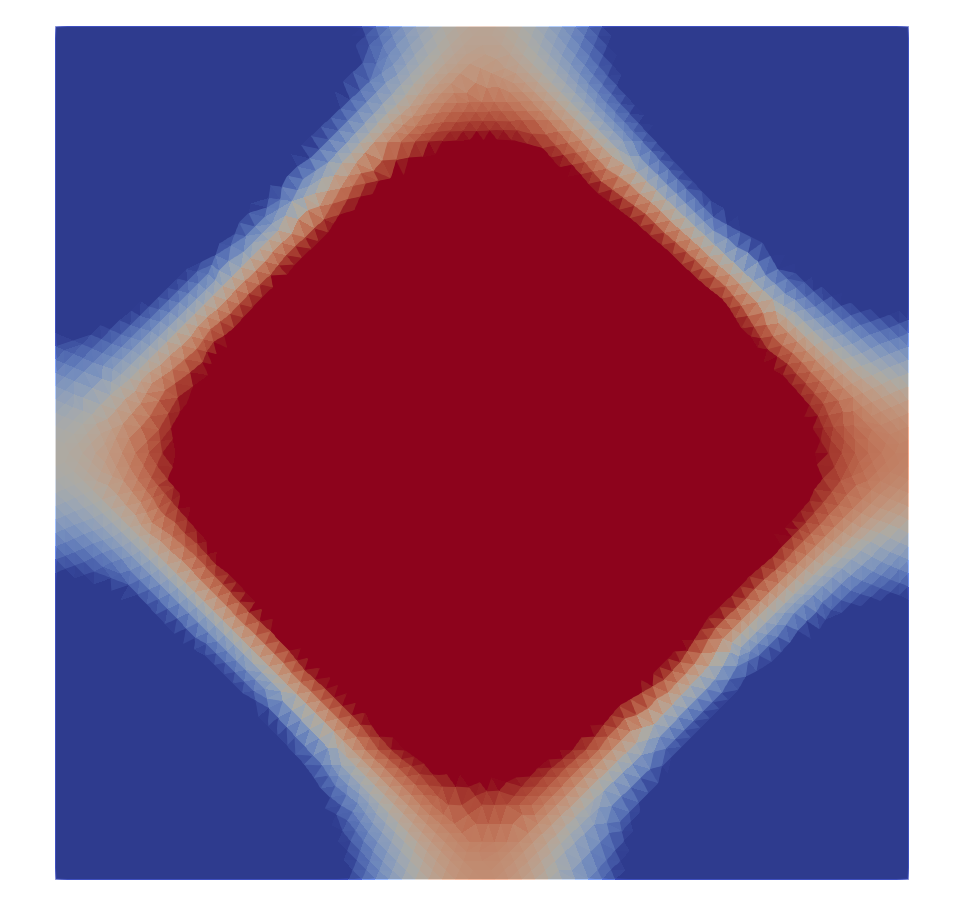} &
\includegraphics[width=0.155\textwidth]{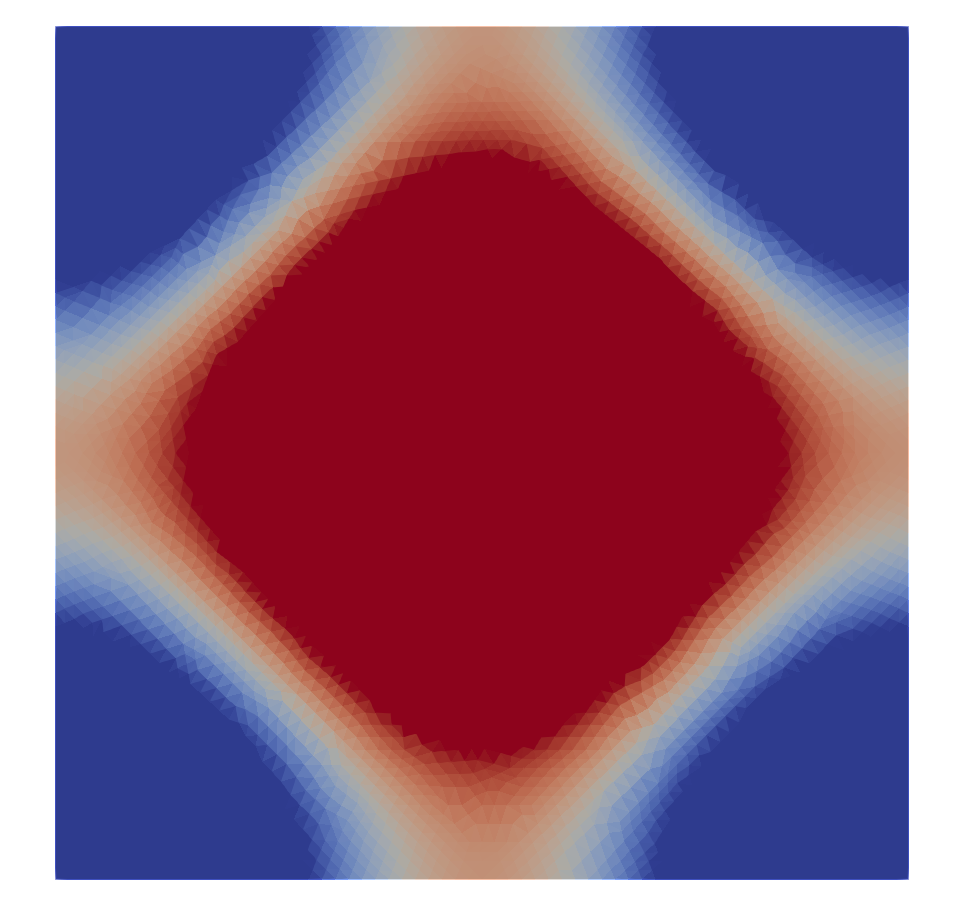} &
\includegraphics[width=0.155\textwidth]{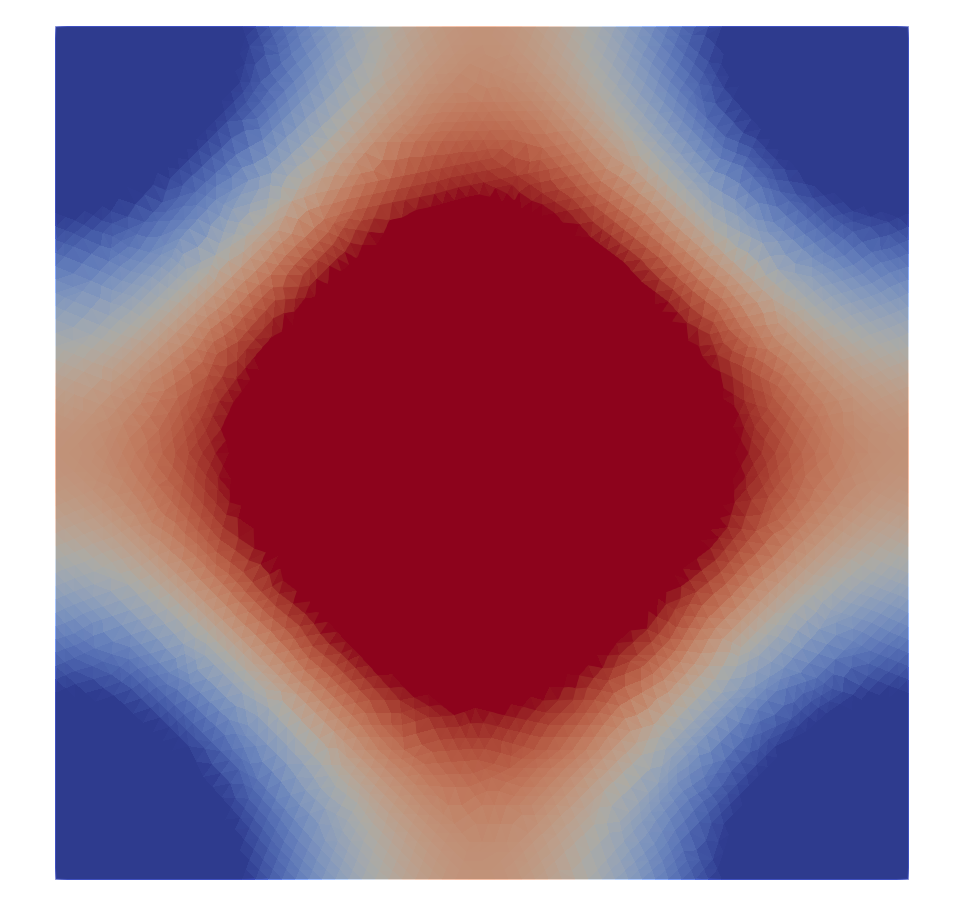} &
\includegraphics[width=0.155\textwidth]{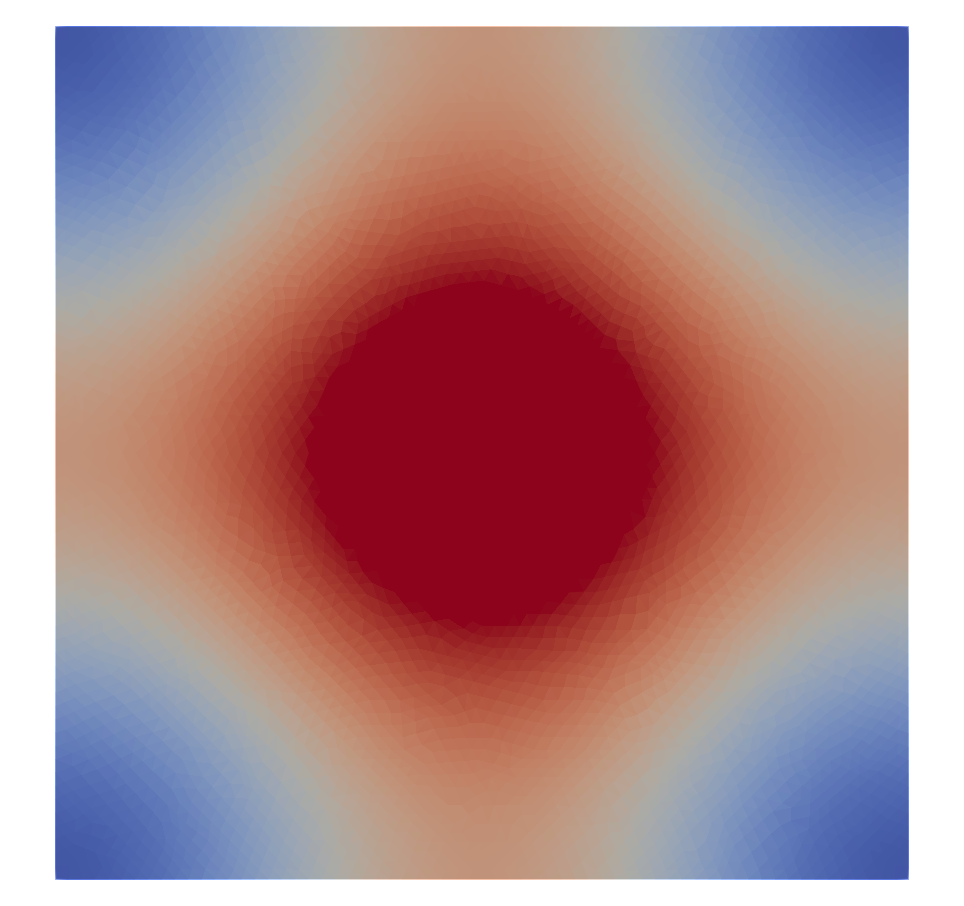} 
& \includegraphics[width=0.04\textwidth]{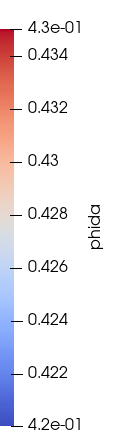}\\[2mm]

\raisebox{0.25\height}{\makebox[0pt][c]{\rotatebox{90}{\textbf{No nudging}}}} &
\includegraphics[width=0.155\textwidth]{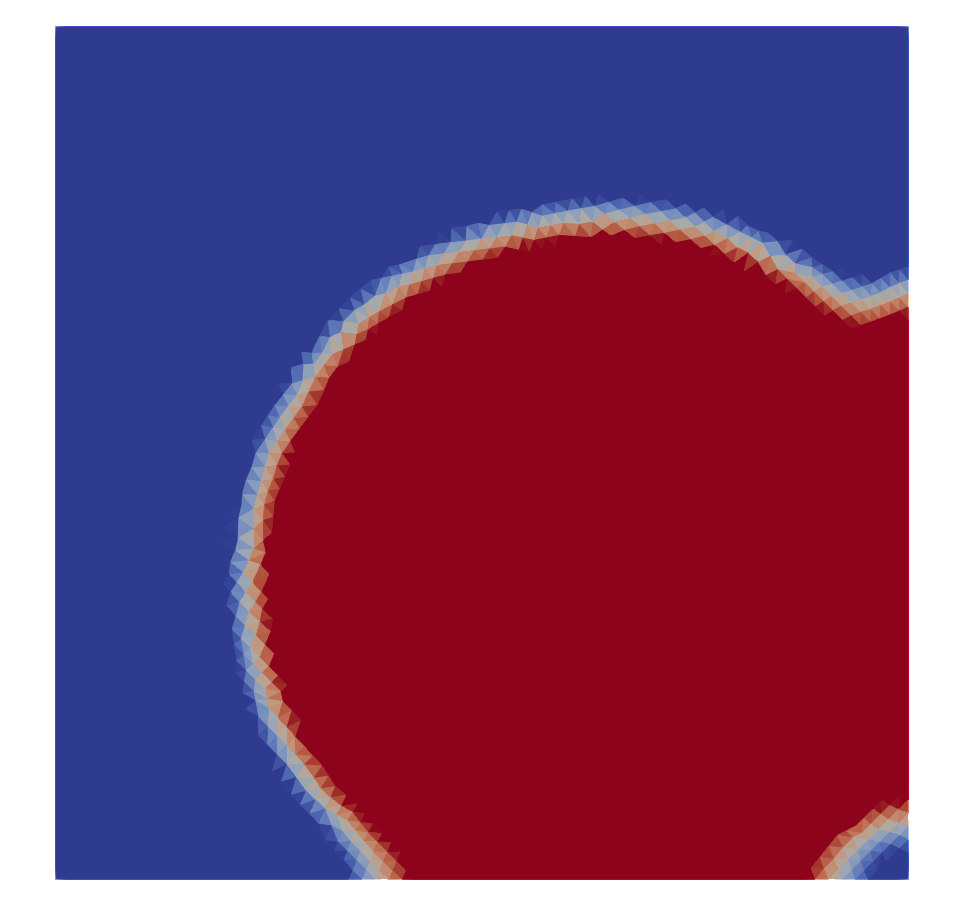} &
\includegraphics[width=0.155\textwidth]{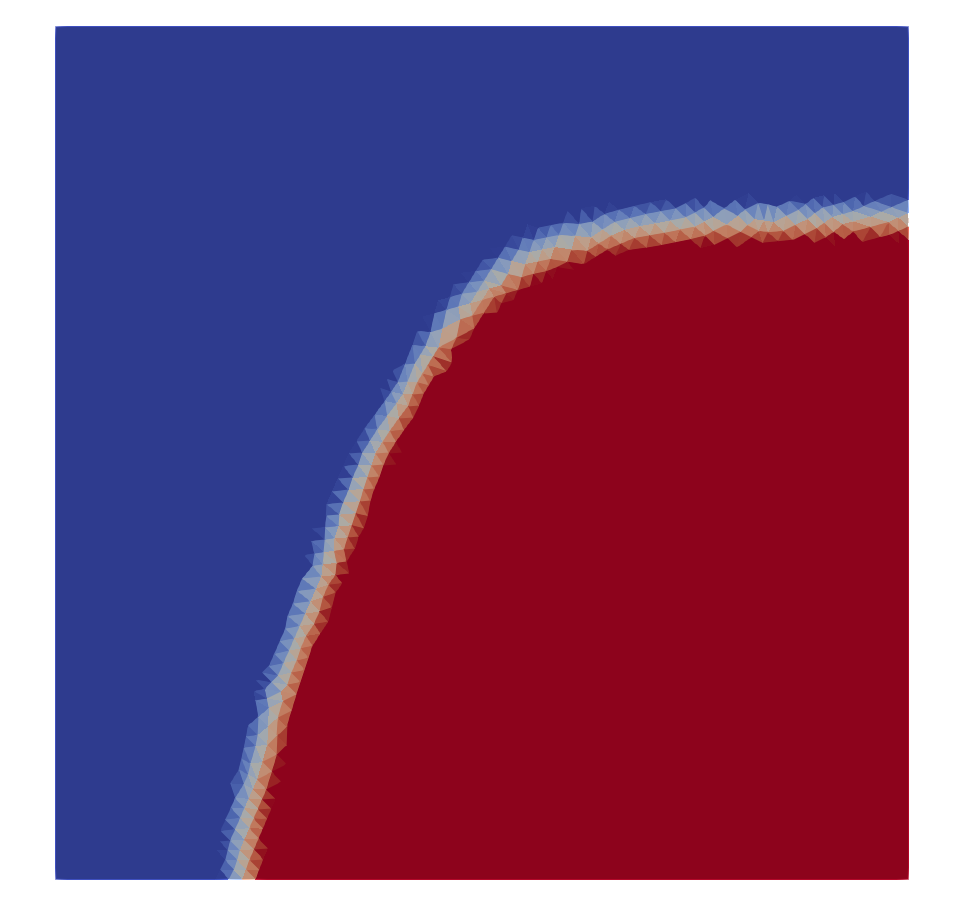} &
\includegraphics[width=0.155\textwidth]{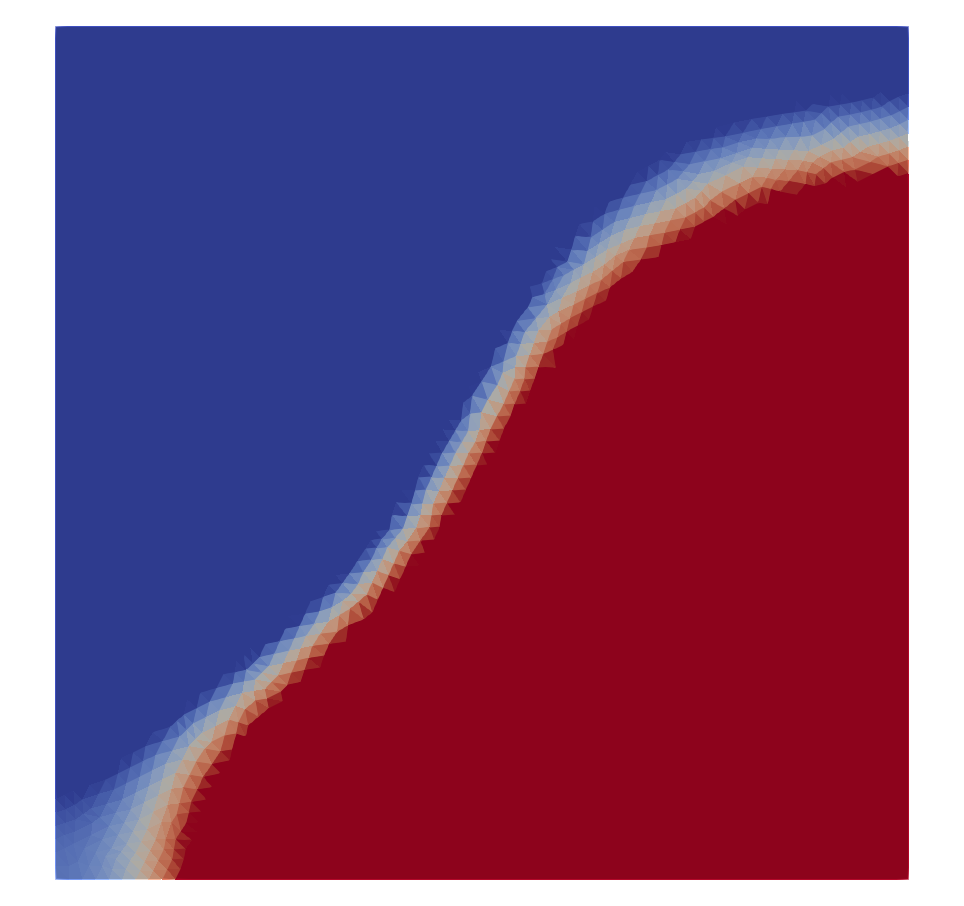} &
\includegraphics[width=0.155\textwidth]{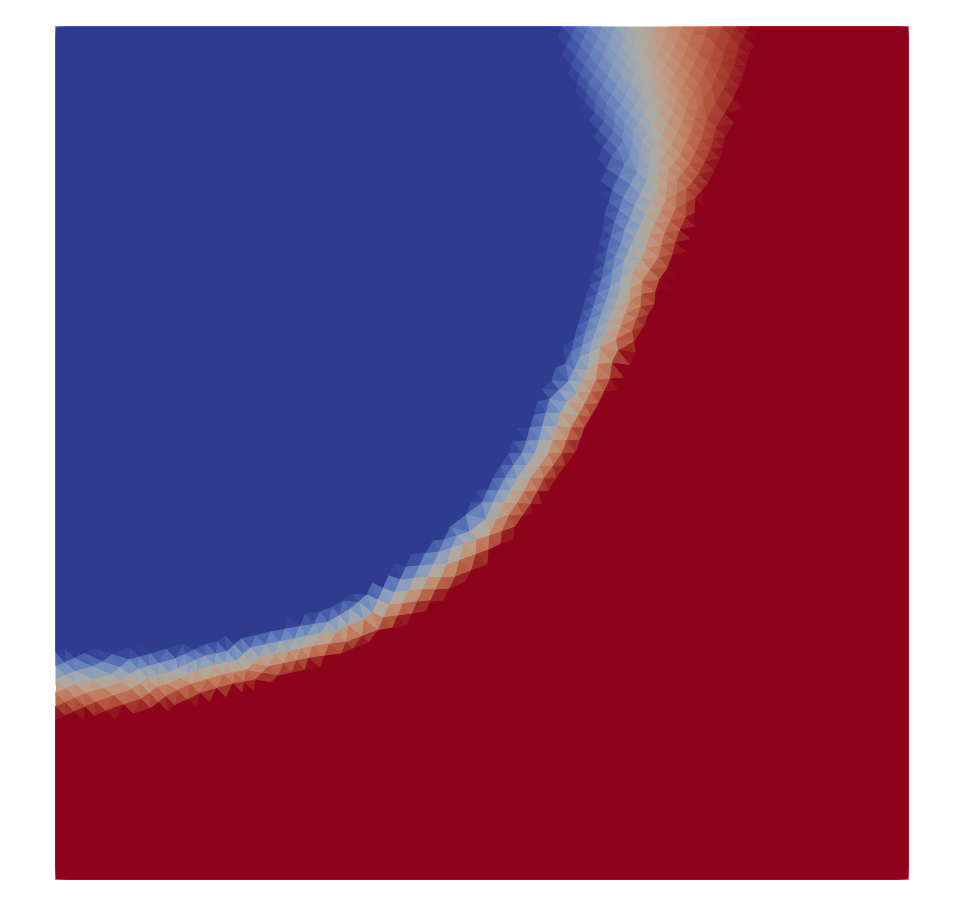} &
\includegraphics[width=0.155\textwidth]{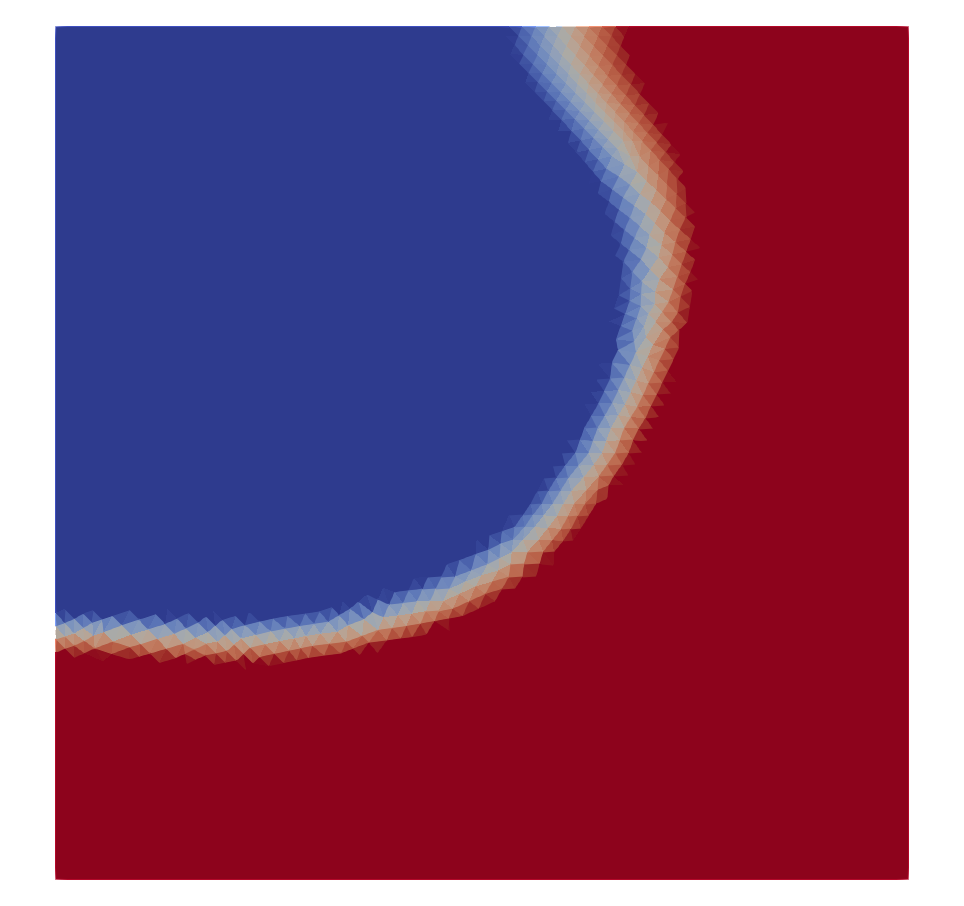} &
\includegraphics[width=0.155\textwidth]{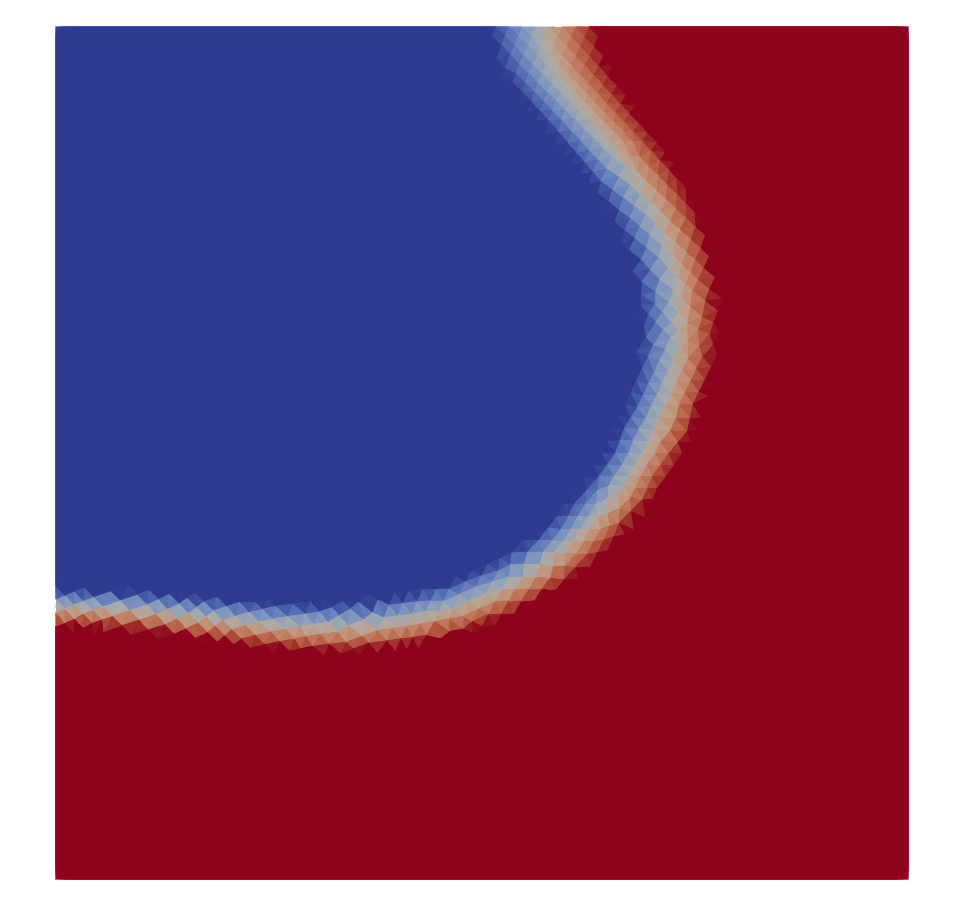}
& \includegraphics[width=0.04\textwidth]{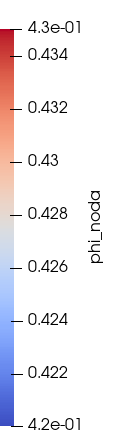}\\[1mm]

& $t=0$ & $t=2$ & $t=4$ & $t=6$ & $t=8$ & $t=10$
\end{tabular}
\caption{Order parameter snapshots. Top row: reference $\phi(t)$.
Middle row: assimilated $\varphi(t)$ with $\alpha_u=\alpha_\phi=\alpha_\psi=1$.
Bottom row: solution without nudging ($\alpha_u=\alpha_\phi=\alpha_\psi=0$).}
\label{fig:phi-compare-da-noda}
\end{figure}

\paragraph{Test 2: effects of initial conditions}
\label{subsec:opposite-test}

We use the same setup as in Test~1 (domain, boundary conditions, discretization, parameters, and observation operator $I_H$),
and we fix $\alpha_u=\alpha_\phi=\alpha_\psi=1$. The only difference is the assimilated initial data: we initialize the CDA state
from an opposite configuration,
\[
\varphi_0 = 1-\phi_0,\qquad \vec\zeta_0 = -\vec\psi_0=(y,-x),\qquad 
\vec v_0(x,y)=\bigl(0.20\sin(\pi y),\,0\bigr),\qquad \pi_0(x,y)\equiv 0.
\]
This provides a stringent test in which the initial droplet phase is inverted and the auxiliary field is sign-reversed.

Figure~\ref{fig:opposite-error} reports the synchronization errors at log scale. Despite the severe mismatch at $t=0$, all errors decay rapidly and then plateau at a discretization-limited level.

\begin{figure}[H]
\centering
\includegraphics[width=0.32\textwidth]{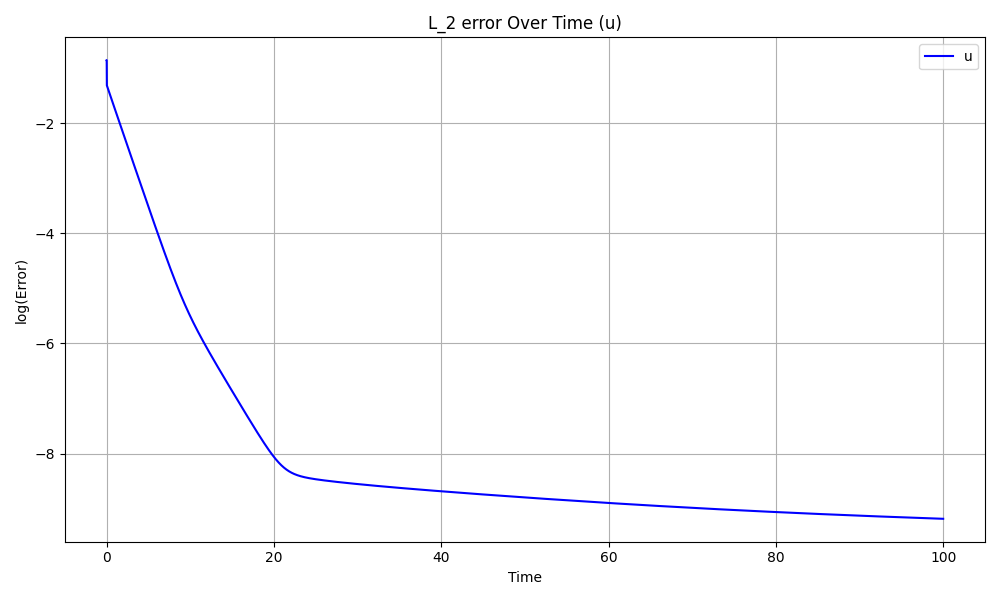}\hfill
\includegraphics[width=0.32\textwidth]{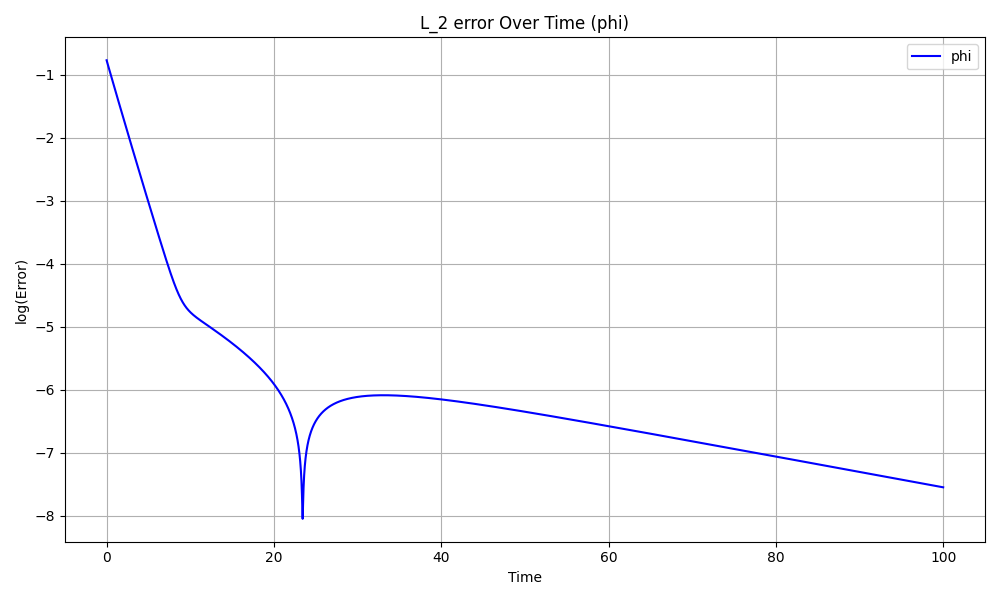}\hfill
\includegraphics[width=0.32\textwidth]{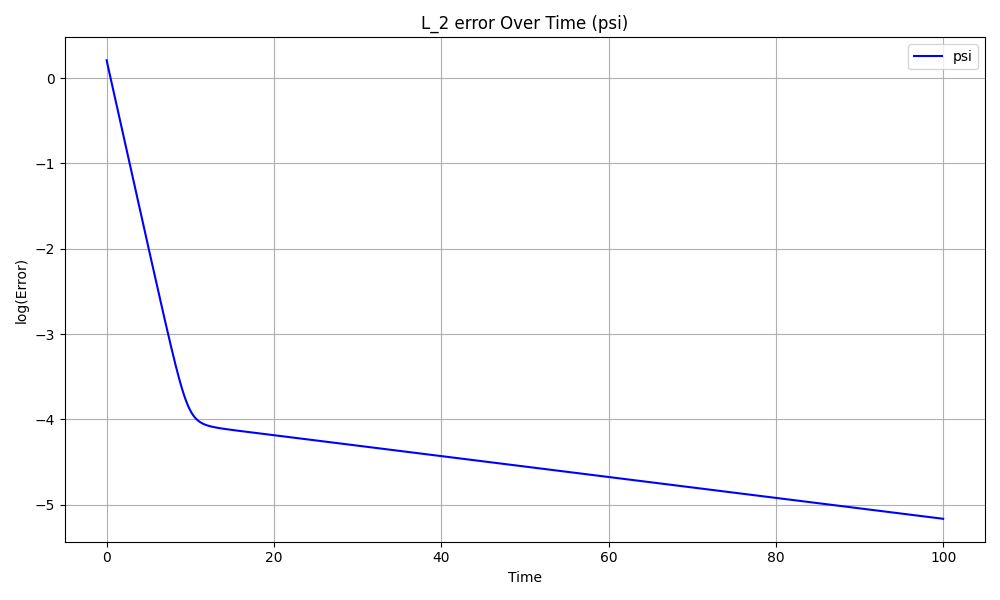}
\caption{Opposite-initialization test: logarithmic $L^2$-errors in $\vec u$, $\phi$, and $\vec\psi$.}
\label{fig:opposite-error}
\end{figure}

Figure~\ref{fig:phi-compare-opposite} compares order-parameter snapshots for the reference and assimilated runs at early times. The assimilated state, initialized from $\varphi_0=1-\phi_0$, rapidly aligns with the reference interface.

\begin{figure}[H]
\centering
\setlength{\tabcolsep}{1pt}
\renewcommand{\arraystretch}{0}
\scriptsize
\begin{tabular}{@{}c*{7}{c}@{}}

\raisebox{0.4\height}{\makebox[0pt][c]{\rotatebox{90}{\textbf{Reference}}}} &
\includegraphics[width=0.155\textwidth]{fig/da/phi0t.png} &
\includegraphics[width=0.155\textwidth]{fig/da/phi2t.png} &
\includegraphics[width=0.155\textwidth]{fig/da/phi4t.png} &
\includegraphics[width=0.155\textwidth]{fig/da/phi6t.png} &
\includegraphics[width=0.155\textwidth]{fig/da/phi8t.png} &
\includegraphics[width=0.155\textwidth]{fig/da/phi10t.png}
& \includegraphics[width=0.04\textwidth]{fig/da/bar.png}\\[2mm]

\raisebox{1.2\height}{\makebox[0pt][c]{\rotatebox{90}{\textbf{CDA}}}} &
\includegraphics[width=0.155\textwidth]{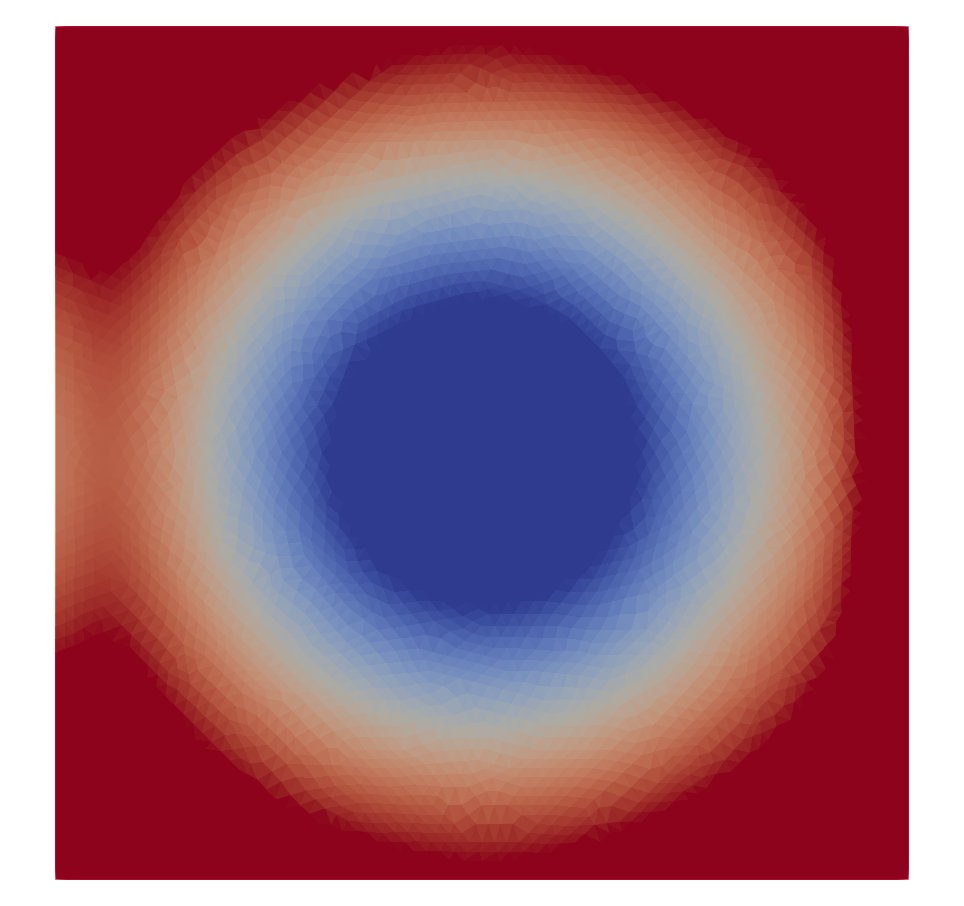} &
\includegraphics[width=0.155\textwidth]{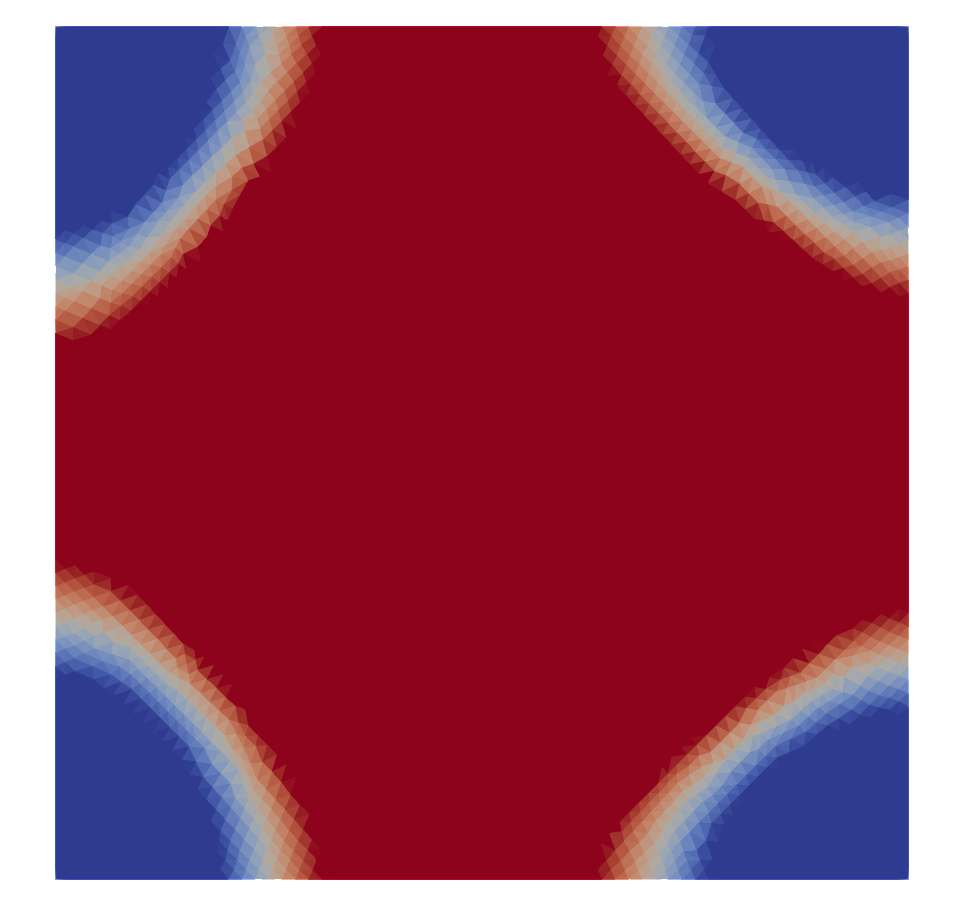} &
\includegraphics[width=0.155\textwidth]{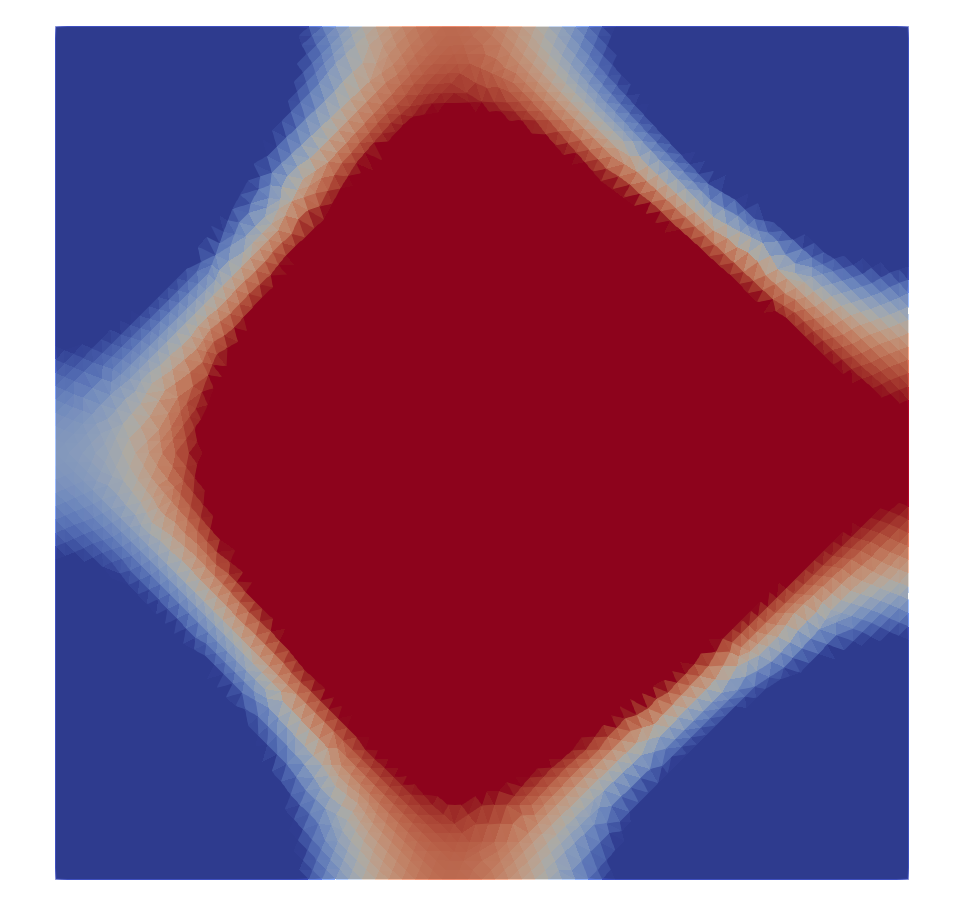} &
\includegraphics[width=0.155\textwidth]{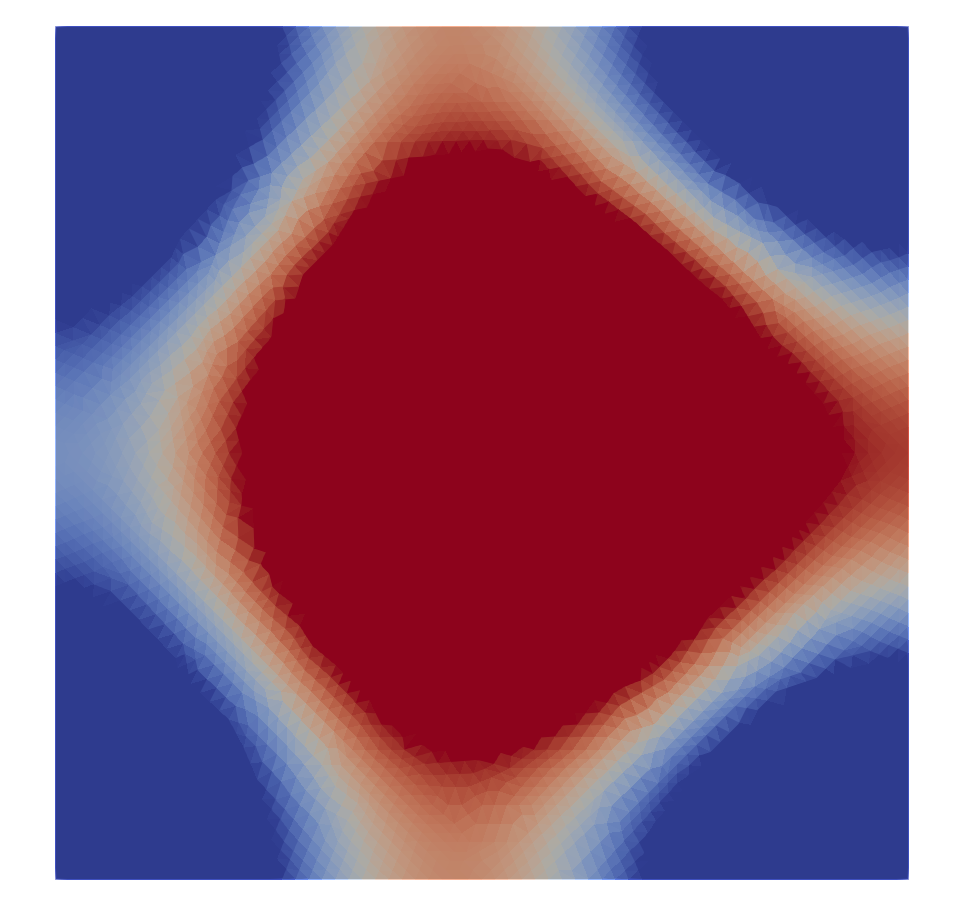} &
\includegraphics[width=0.155\textwidth]{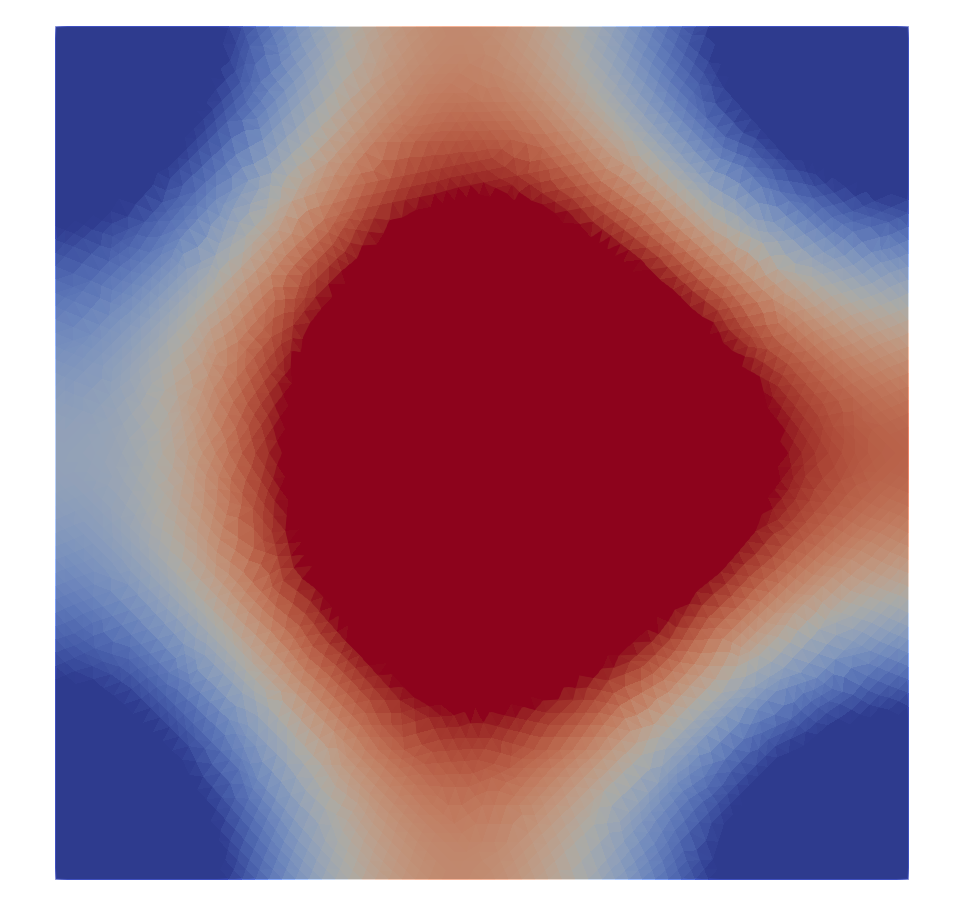} &
\includegraphics[width=0.155\textwidth]{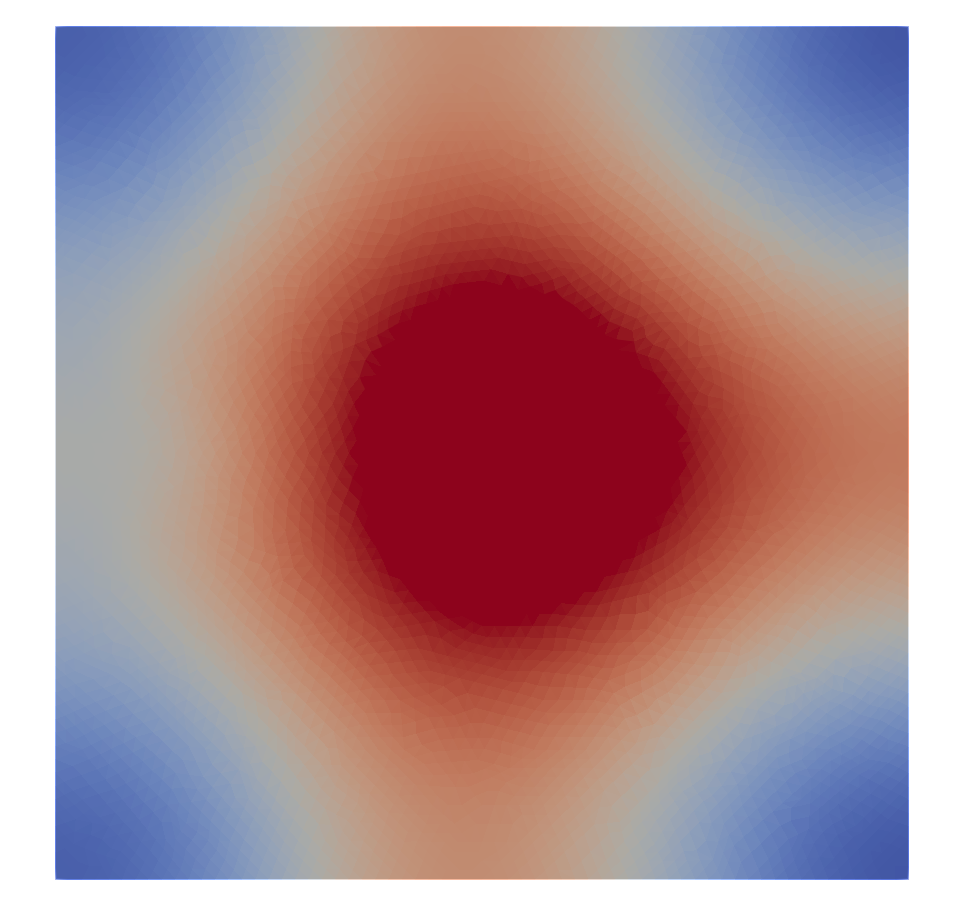}
& \includegraphics[width=0.04\textwidth]{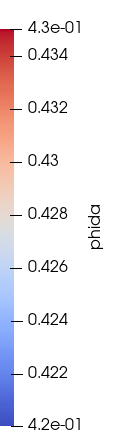}\\[1mm]

& $t=0$ & $t=2$ & $t=4$ & $t=6$ & $t=8$ & $t=10$
\end{tabular}
\caption{Order parameter snapshots. Top row: reference $\phi(t)$. Bottom row: assimilated $\varphi(t)$ initialized by $\varphi_0=1-\phi_0$.}
\label{fig:phi-compare-opposite}
\end{figure}

\paragraph{Test 3: effects of boundary conditions}
\label{subsec:shear-test}

We consider the same setup as in Test~1 (domain, discretization, observation operator $I_H$, and initial data), but now impose a
shear flow by prescribing a moving-lid boundary condition on the top boundary:
\[
\vec u=(U_0,0)\ \text{on }\{y=1\},\qquad
\vec u=\vec 0\ \text{on }\{y=0\}\cup\{x=0\}\cup\{x=1\},
\qquad U_0=2.
\]
The nudging parameters are fixed as $\alpha_u=\alpha_\phi=\alpha_\psi=1$, and we set $\mathrm{Re}=100$ for this test.

Figure~\ref{fig:shear-energy} reports the evolution of the total energy and its components. In contrast to Test~1, the moving lid injects kinetic energy into the system, and therefore, the total energy is not monotone. The kinetic energy increases as the shear develops, while the mixing energy rapidly relaxes to an approximately steady value, and the elastic contribution decays quickly.

\begin{figure}[H]
\centering
\includegraphics[width=0.48\textwidth]{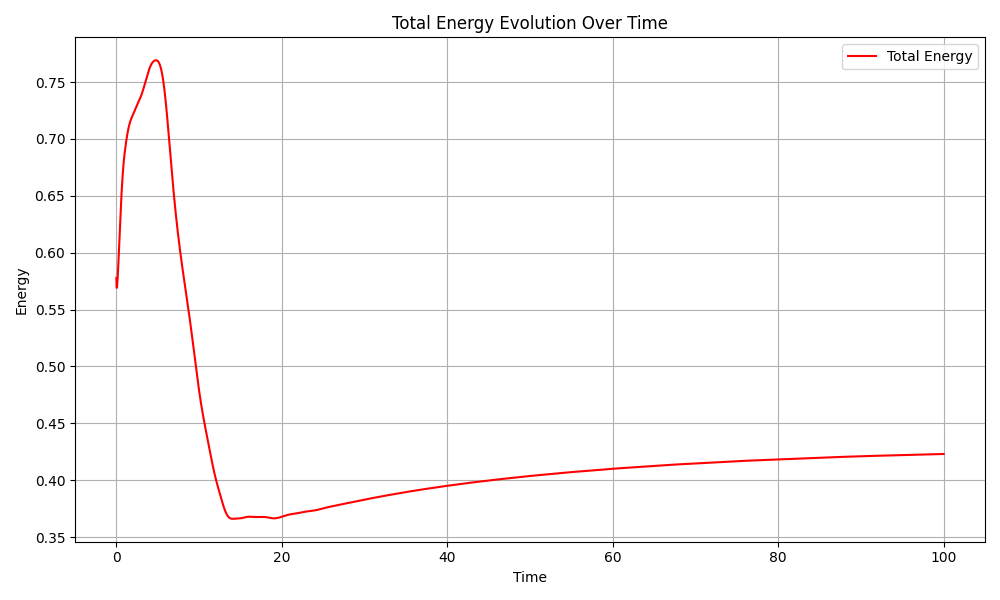}\hfill
\includegraphics[width=0.48\textwidth]{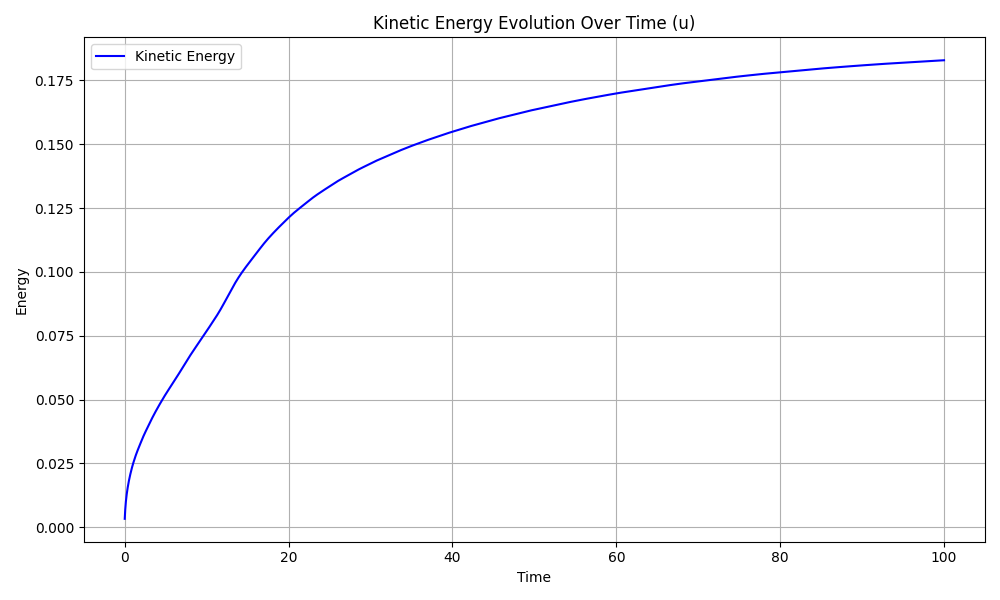}\\[0.5ex]
\includegraphics[width=0.48\textwidth]{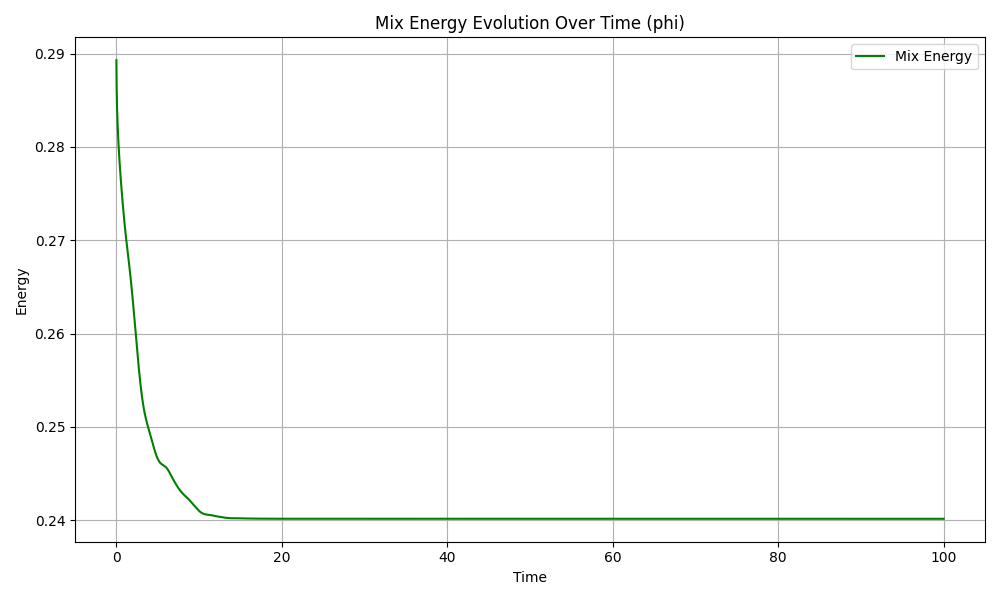}\hfill
\includegraphics[width=0.48\textwidth]{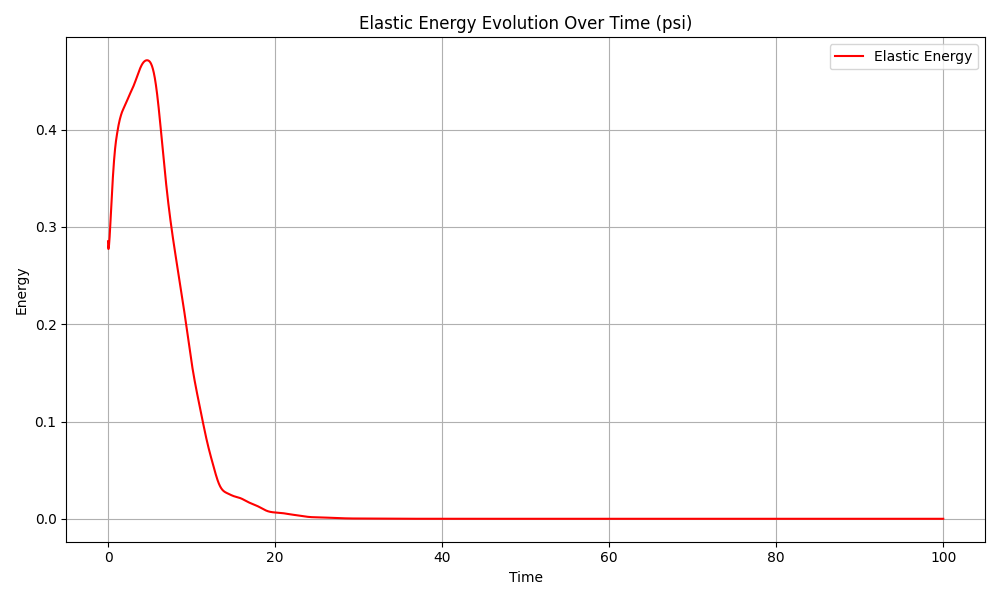}
\caption{Shear-driven test: total energy and its components.}
\label{fig:shear-energy}
\end{figure}

We next show the synchronization errors on a logarithmic scale. As shown in Figure~\ref{fig:shear-error}, the errors decay rapidly and then plateau at a discretization-limited level.

\begin{figure}[H]
\centering
\includegraphics[width=0.32\textwidth]{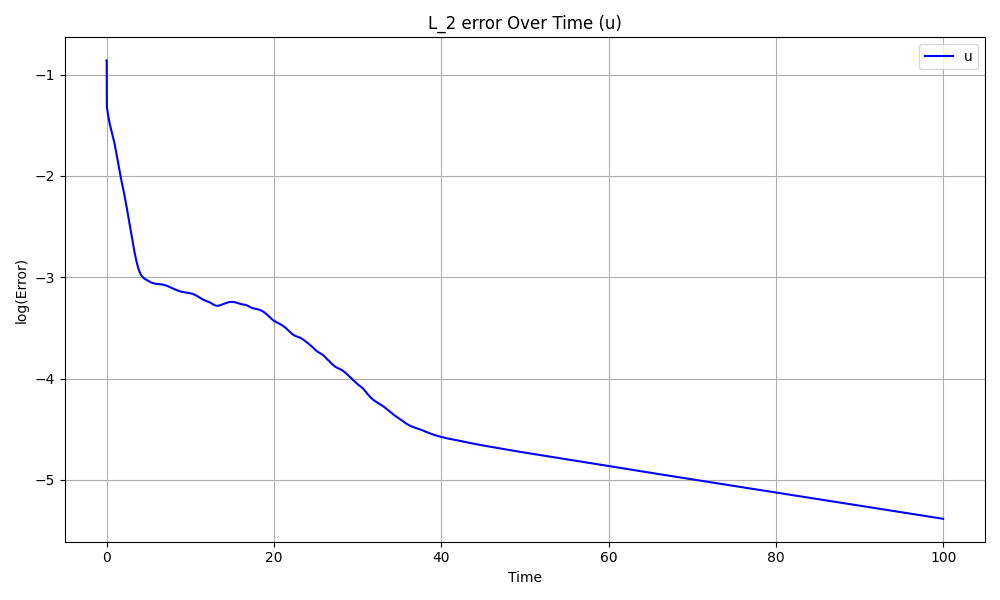}\hfill
\includegraphics[width=0.32\textwidth]{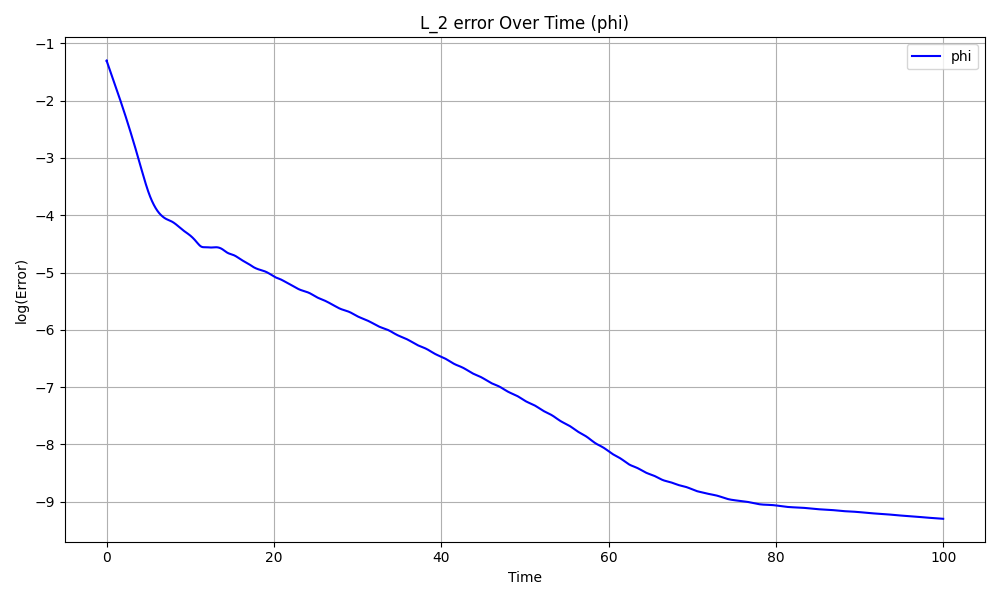}\hfill
\includegraphics[width=0.32\textwidth]{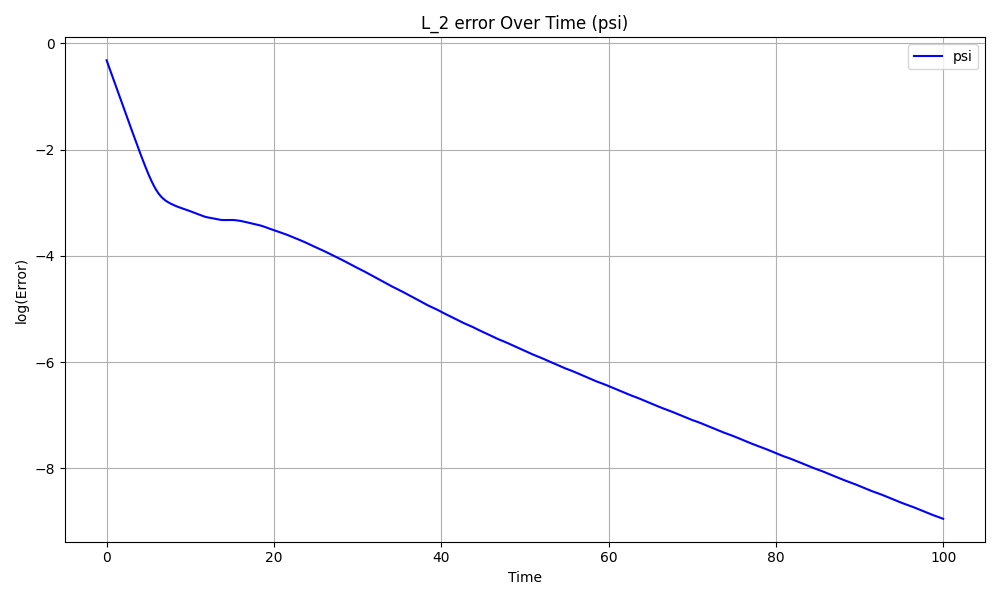}
\caption{Shear-driven test: logarithmic $L^2$-errors in $\vec u$, $\phi$, and $\vec\psi$.}
\label{fig:shear-error}
\end{figure}

We finally compare order-parameter snapshots for the reference and assimilated runs at early times. Despite the initial mismatch in droplet
location and radius, the assimilated interface rapidly aligns with the reference interface under the imposed shear.

\begin{figure}[H]
\centering
\setlength{\tabcolsep}{1pt}
\renewcommand{\arraystretch}{0}
\scriptsize
\begin{tabular}{@{}c*{7}{c}@{}}

\raisebox{0.4\height}{\makebox[0pt][c]{\rotatebox{90}{\textbf{Reference}}}} &
\includegraphics[width=0.155\textwidth]{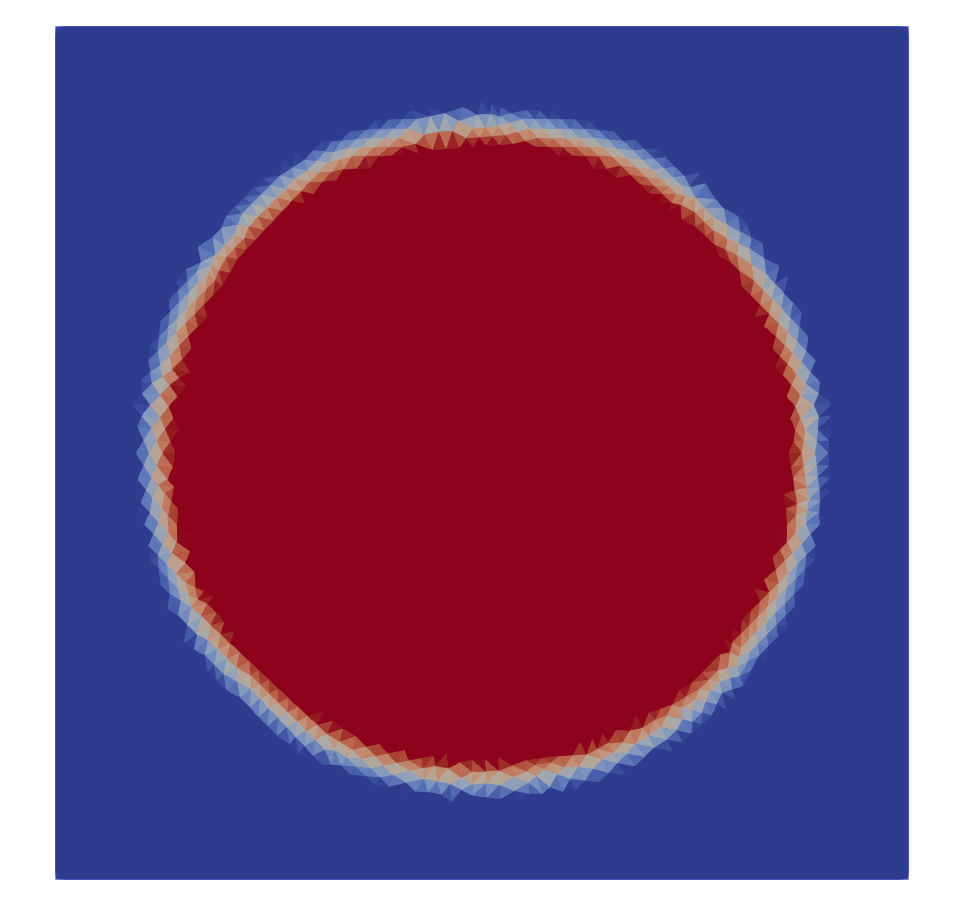} &
\includegraphics[width=0.155\textwidth]{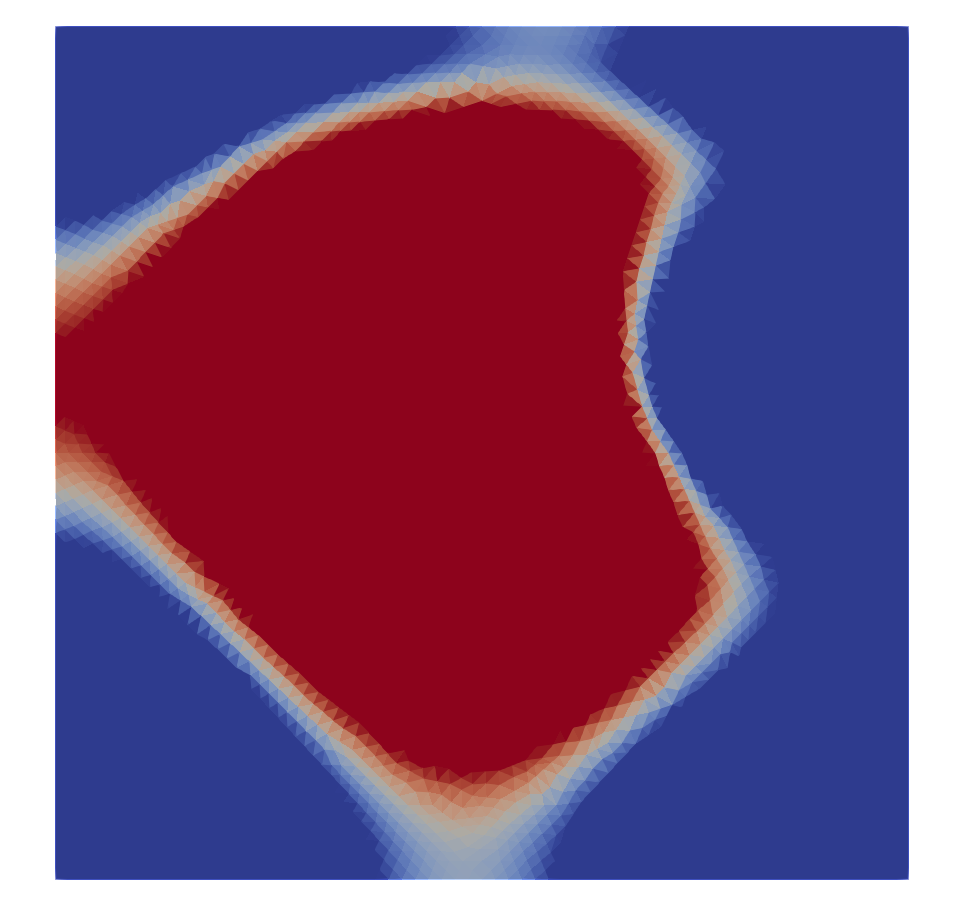} &
\includegraphics[width=0.155\textwidth]{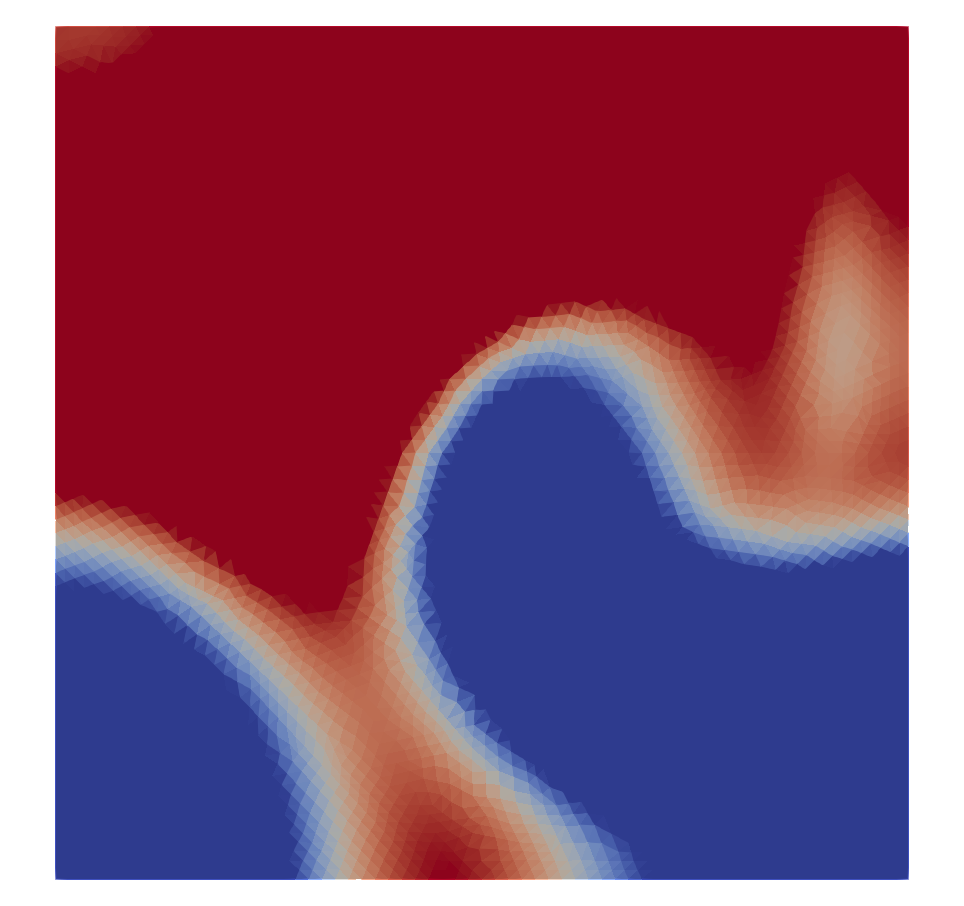} &
\includegraphics[width=0.155\textwidth]{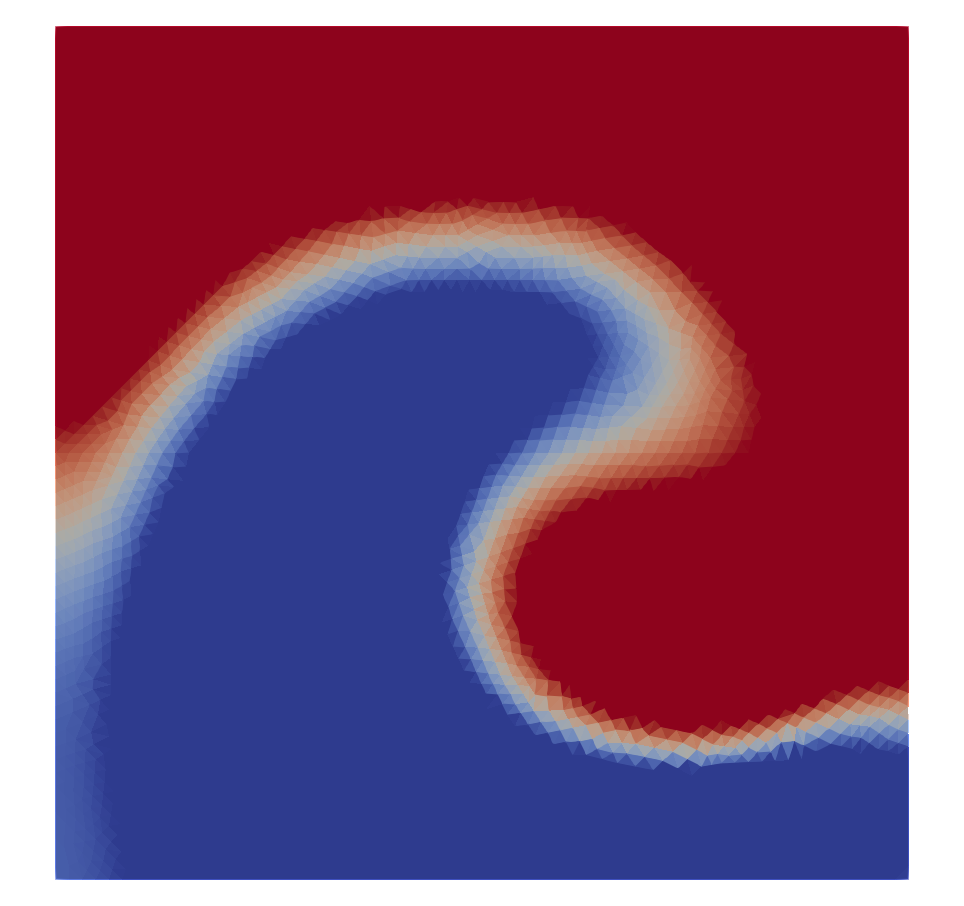} &
\includegraphics[width=0.155\textwidth]{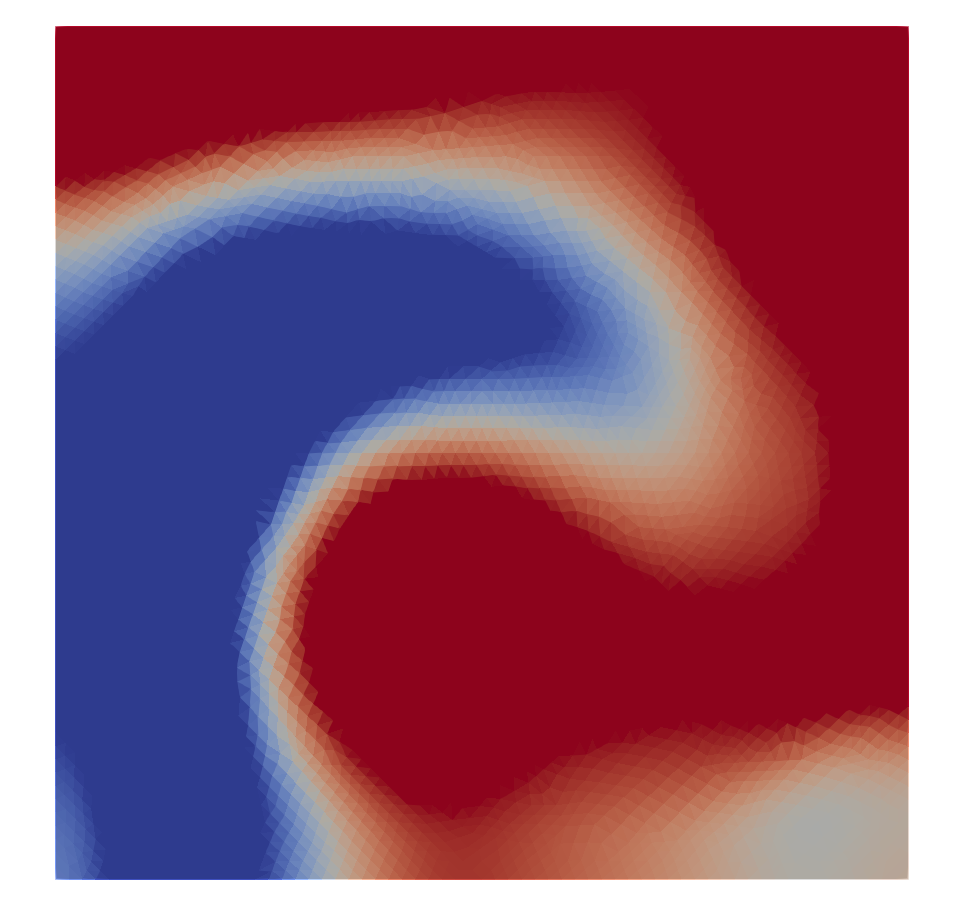} &
\includegraphics[width=0.155\textwidth]{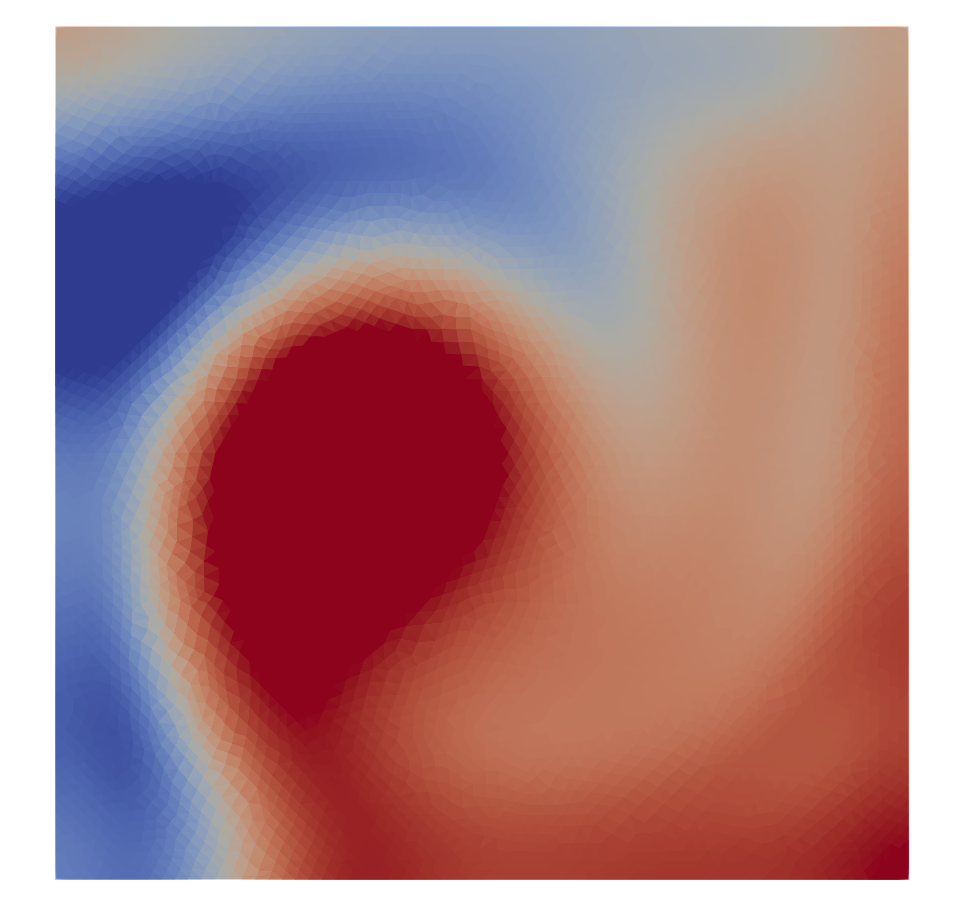}
& \includegraphics[width=0.04\textwidth]{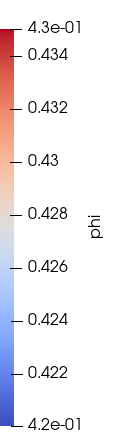}\\[2mm]

\raisebox{1.2\height}{\makebox[0pt][c]{\rotatebox{90}{\textbf{CDA}}}} &
\includegraphics[width=0.155\textwidth]{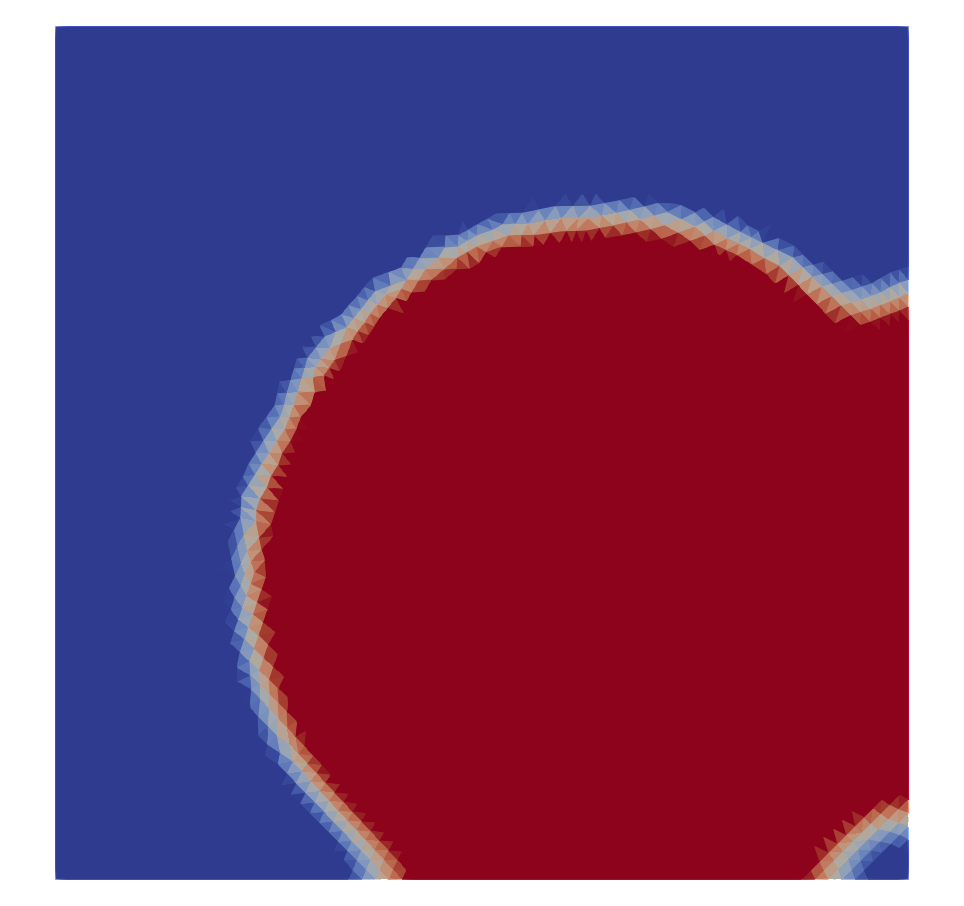} &
\includegraphics[width=0.155\textwidth]{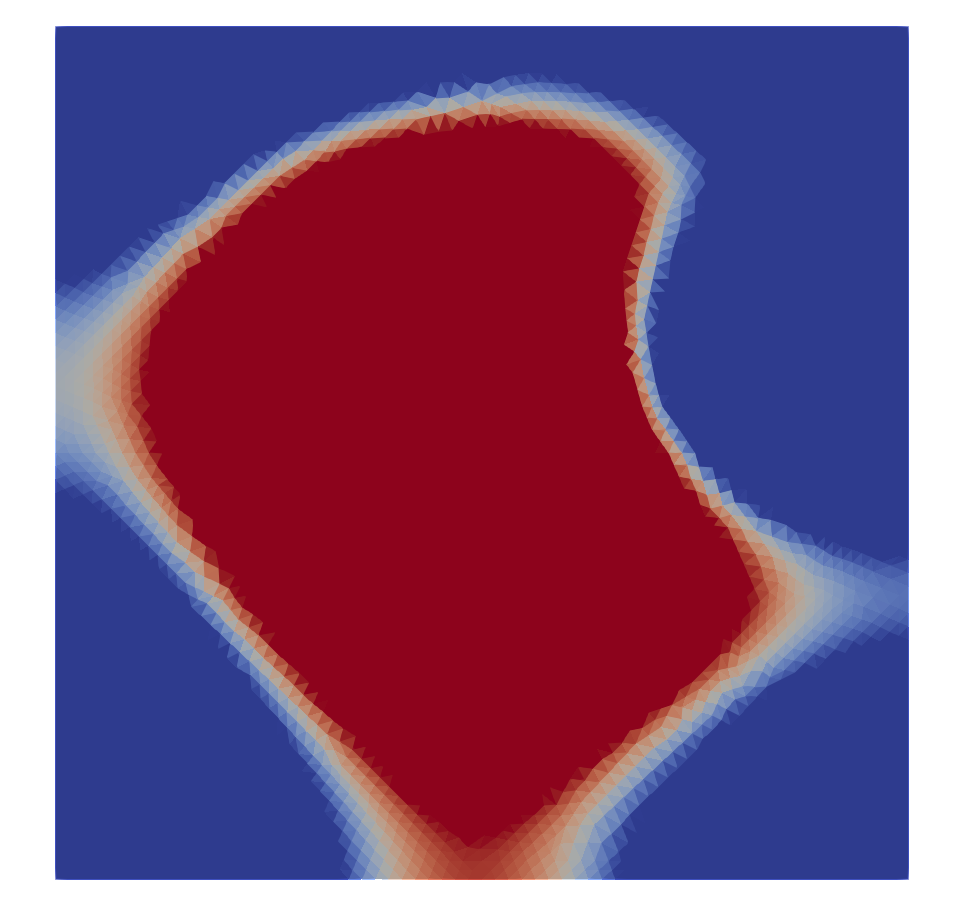} &
\includegraphics[width=0.155\textwidth]{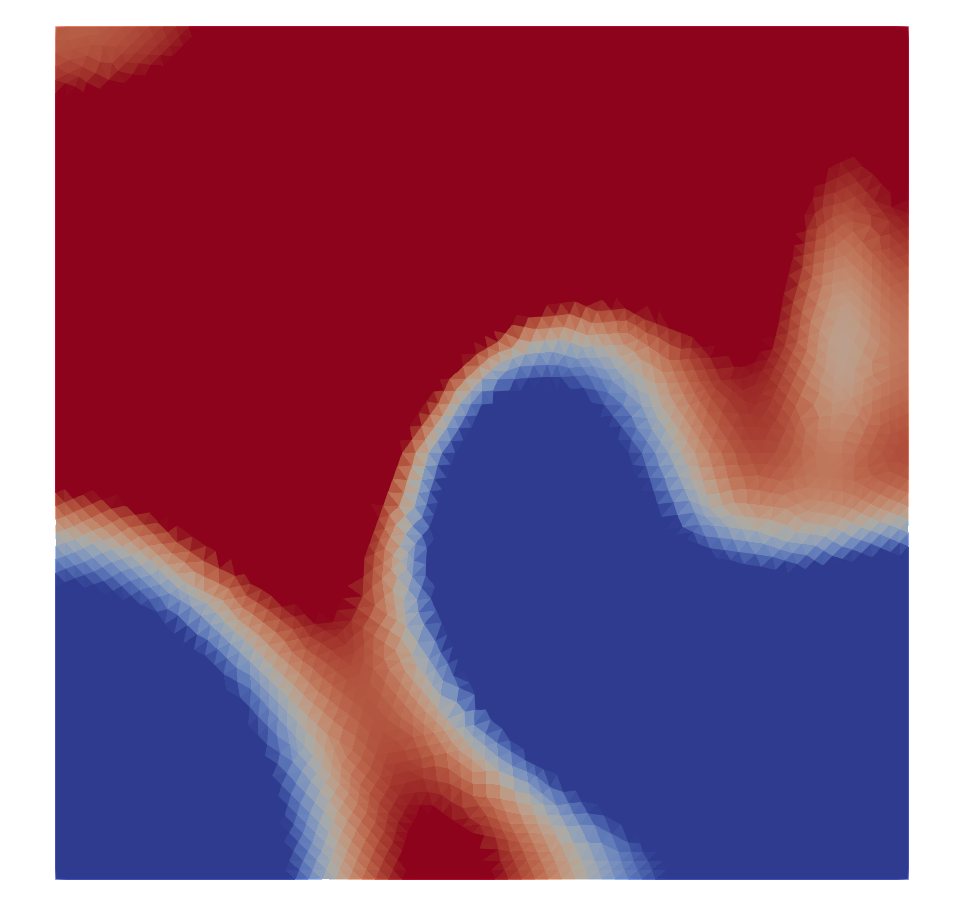} &
\includegraphics[width=0.155\textwidth]{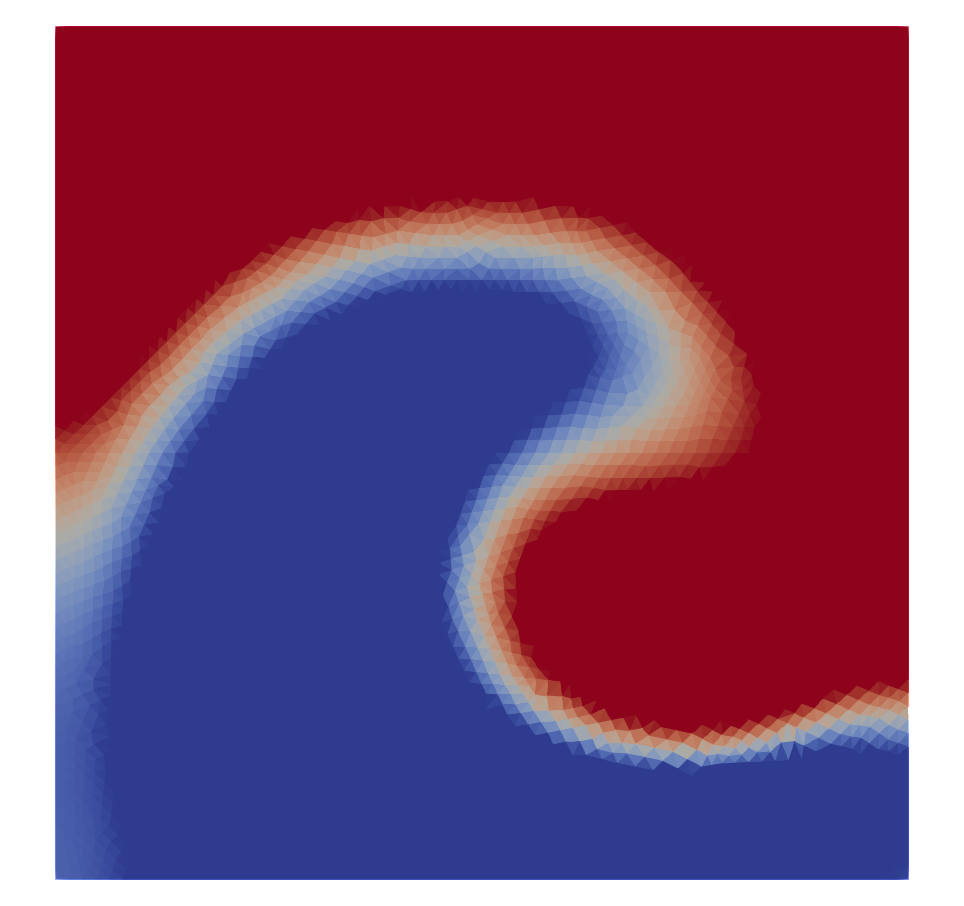} &
\includegraphics[width=0.155\textwidth]{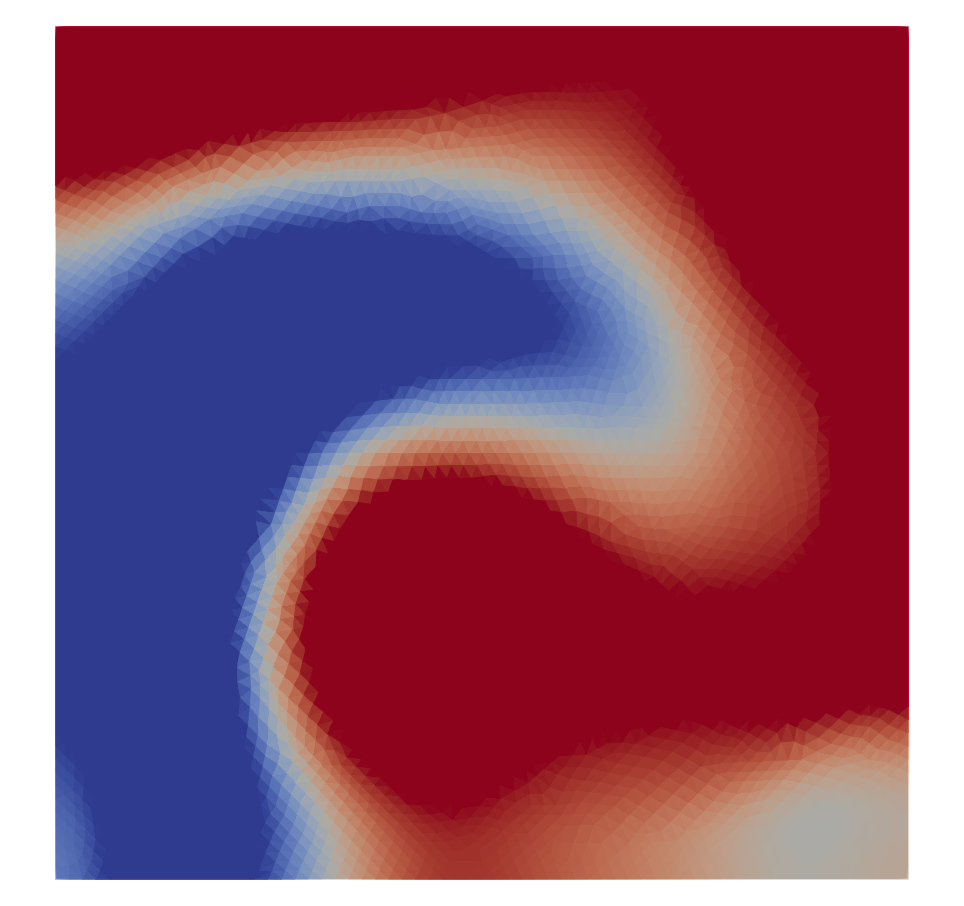} &
\includegraphics[width=0.155\textwidth]{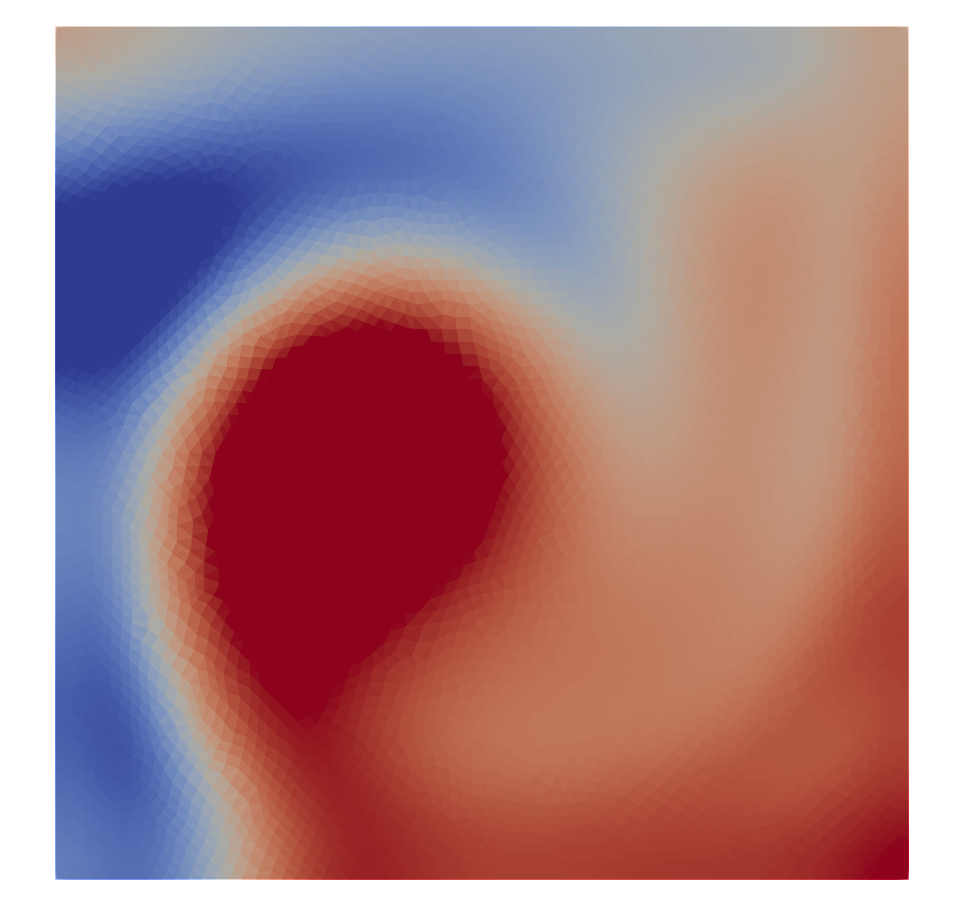}
& \includegraphics[width=0.04\textwidth]{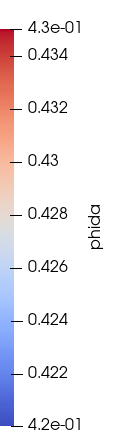}
\\[2mm]

\raisebox{0.2\height}{\makebox[0pt][c]{\rotatebox{90}{\textbf{No Nudging}}}} &
\includegraphics[width=0.155\textwidth]{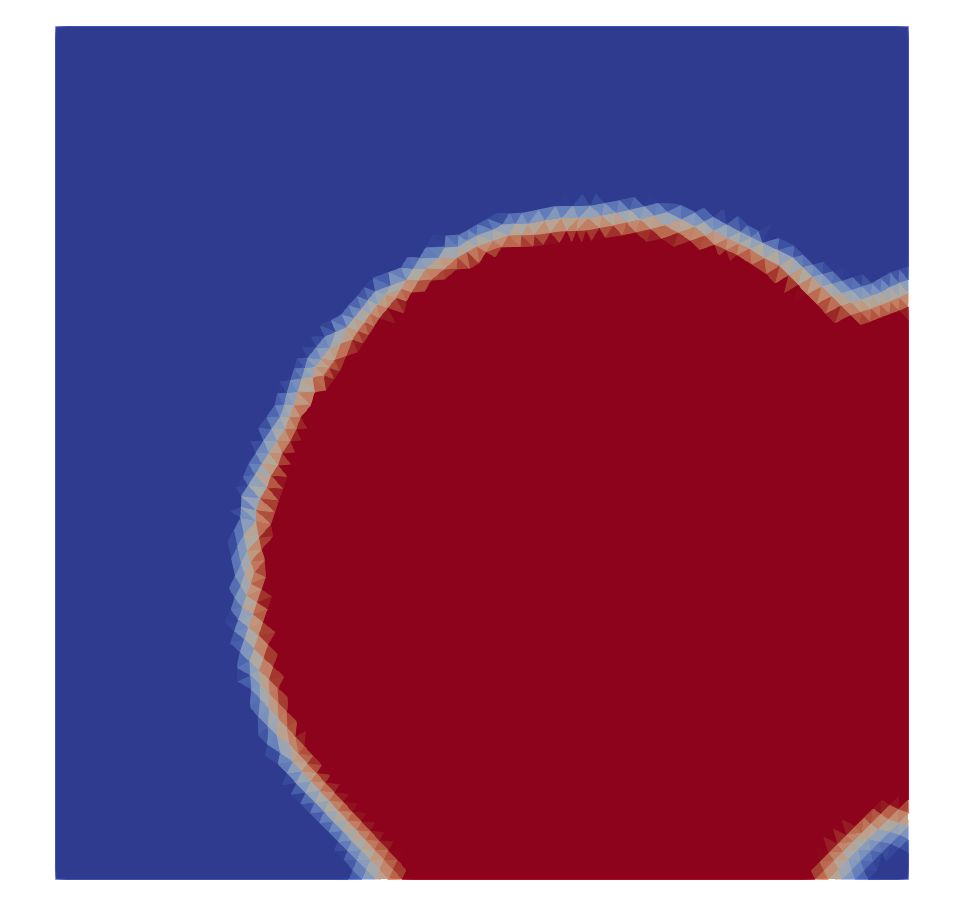} &
\includegraphics[width=0.155\textwidth]{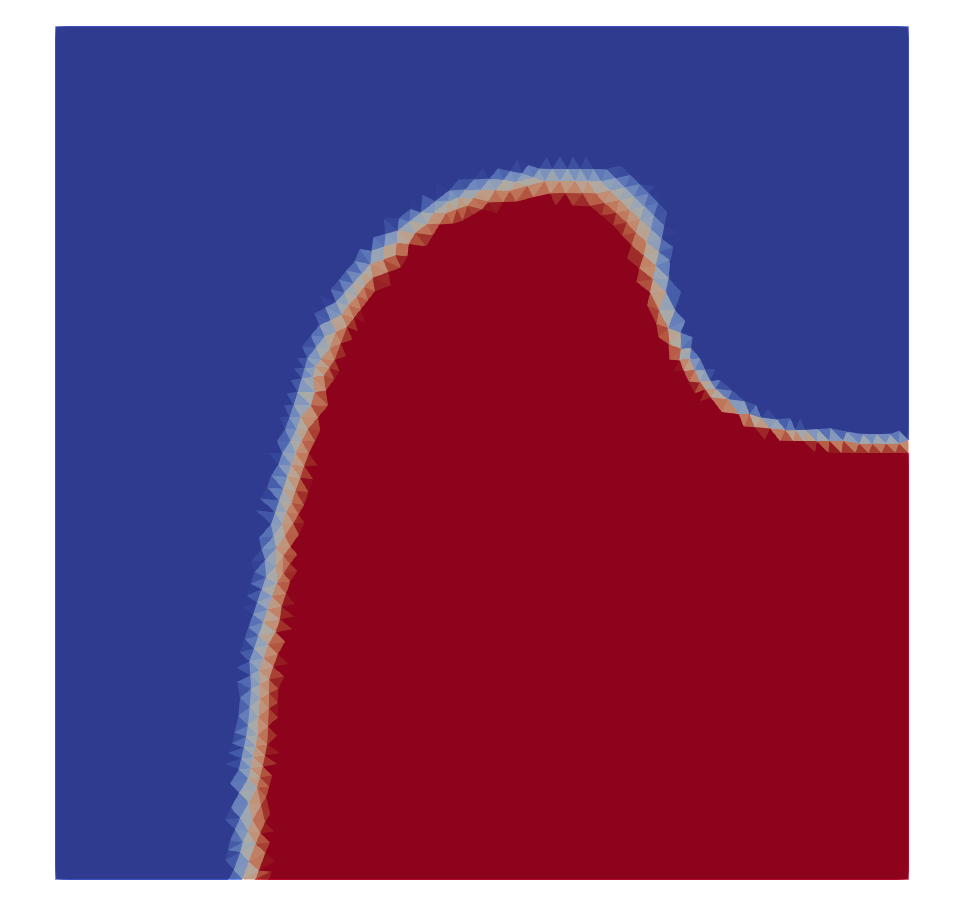} &
\includegraphics[width=0.155\textwidth]{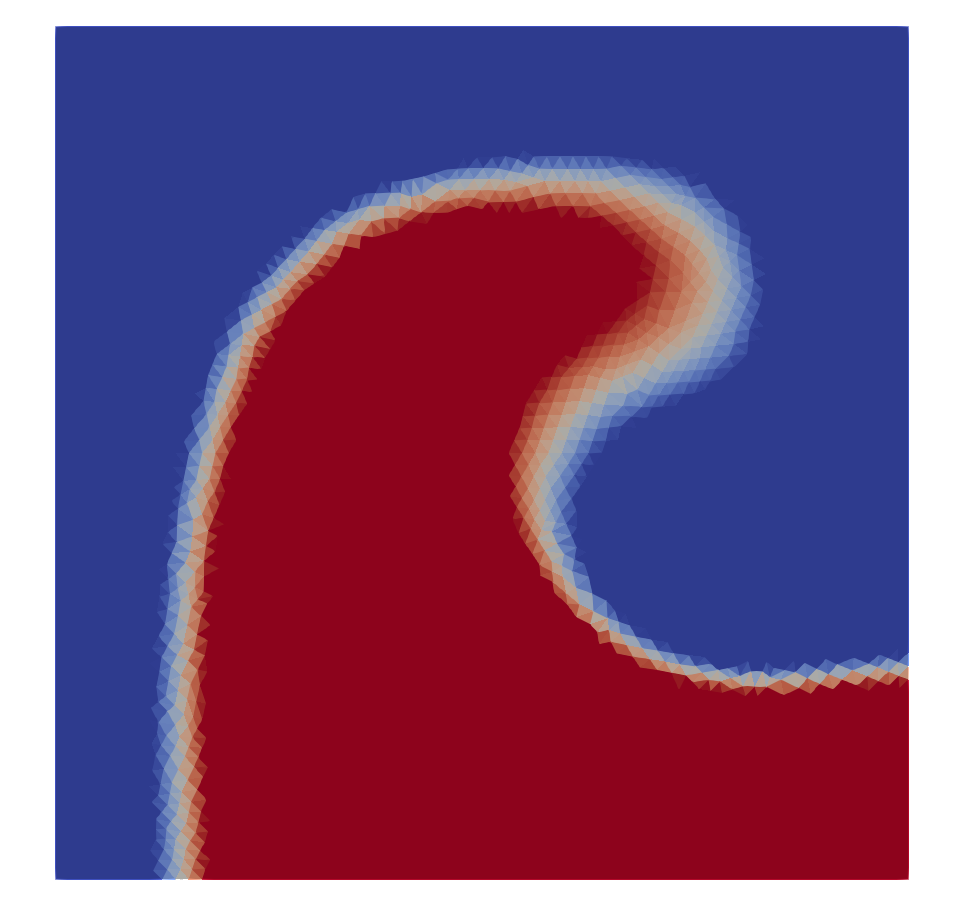} &
\includegraphics[width=0.155\textwidth]{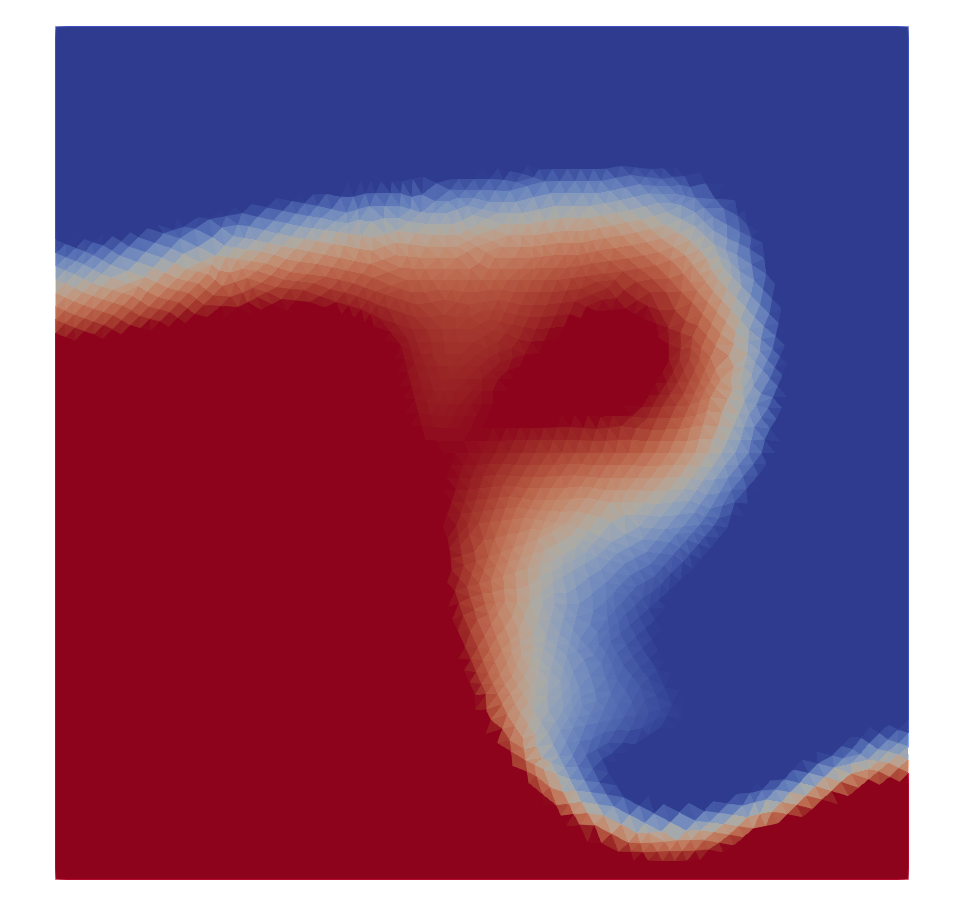} &
\includegraphics[width=0.155\textwidth]{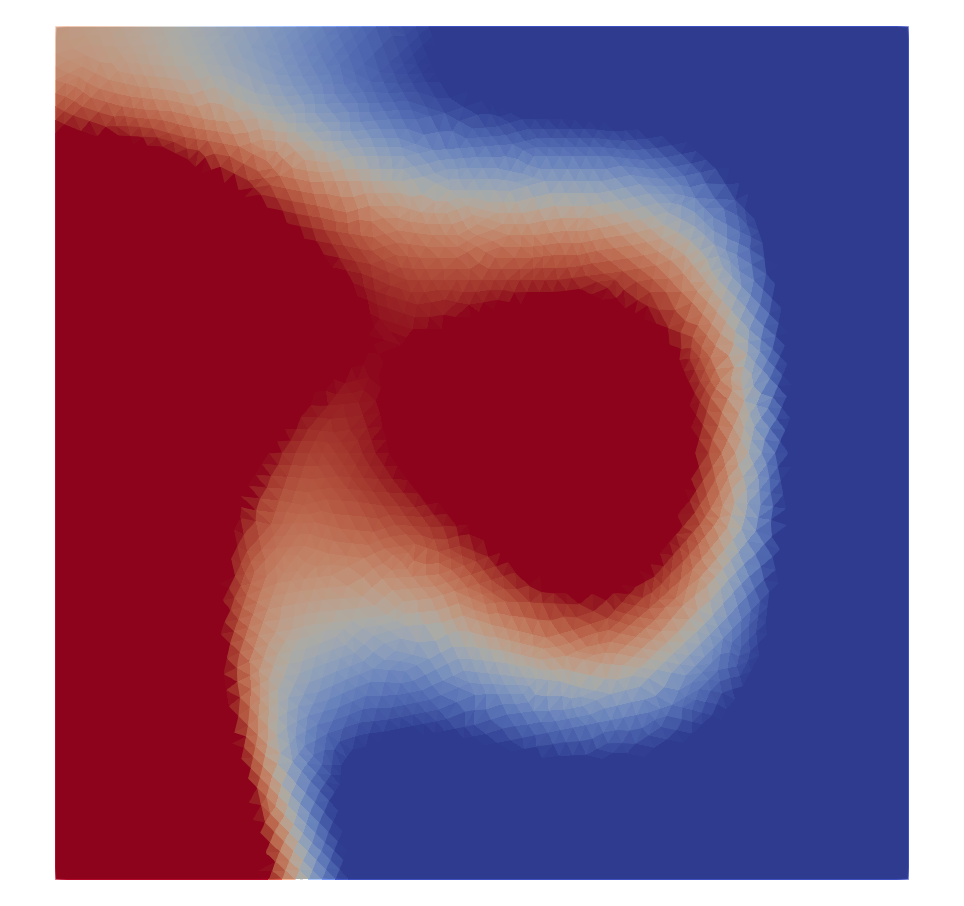} &
\includegraphics[width=0.155\textwidth]{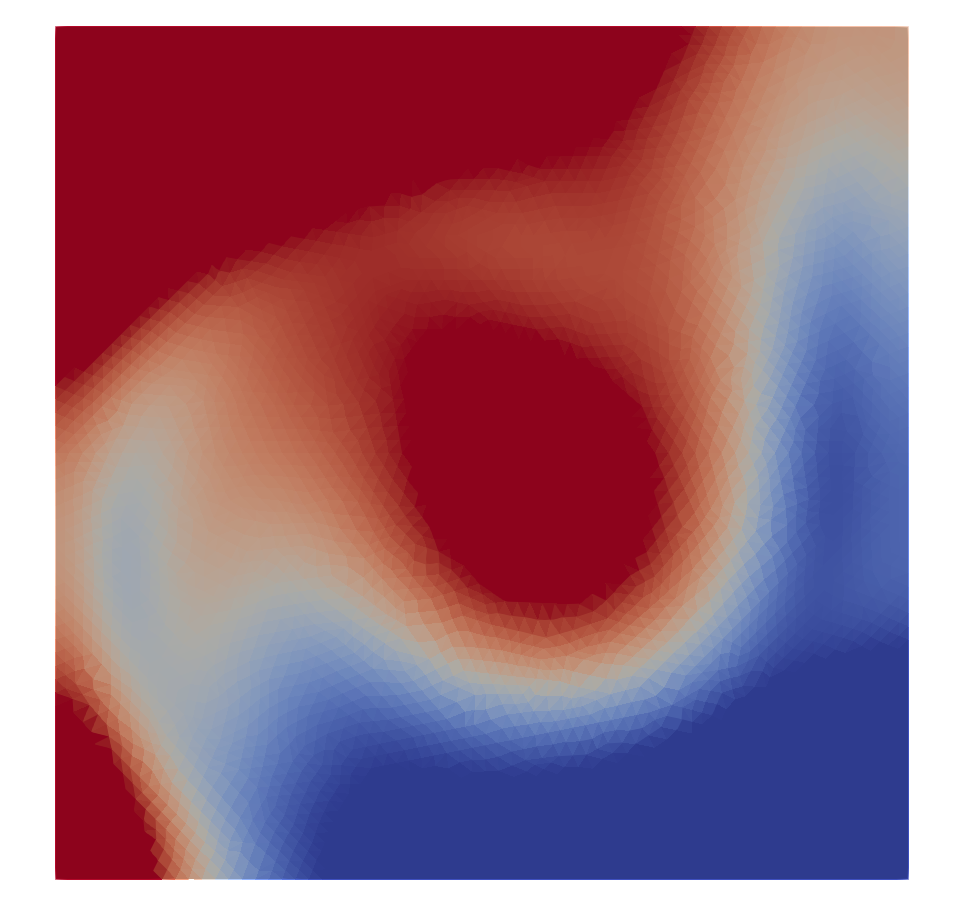}
& \includegraphics[width=0.04\textwidth]{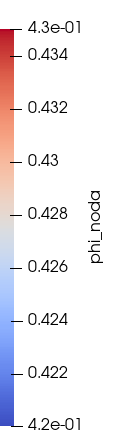}\\[1mm]

& $t=0$ & $t=2$ & $t=4$ & $t=6$ & $t=8$ & $t=10$
\end{tabular}
\caption{Order parameter snapshots. Top row: reference $\phi(t)$. Middle row: assimilated $\varphi(t)$ with
$\alpha_u=\alpha_\phi=\alpha_\psi=1$. Bottom row: $\varphi(t)$ without nudging
($\alpha_u=\alpha_\phi=\alpha_\psi=0$).}
\label{fig:shear-phi-compare}
\end{figure}

\paragraph{Test 4: validation of data assimilation}
\label{subsec:nudging-valid-test}

In this test, we check whether coarse-resolution initial information alone is sufficient to determine the future evolution.
We build two fine-grid initial conditions that are indistinguishable after applying the interpolation operator $I_H$,
but that differ on the fine scales. We then evolve both reference systems forward with the same parameters and boundary
conditions as in Test~3.

Let $\phi_{\rm naive}$ denote the base initial condition defined on the fine mesh $\mathcal T_H$. We apply the observation operator $I_H$ to $\phi_{\rm naive}$ and denote the reconstructed field by $\phi_{\rm coarse} := I_H \phi_{\rm naive}.$ The difference $\delta\phi := \phi_{\rm coarse}-\phi_{\rm naive}$ contains only fine-scale components that are lost when passing through $I_H$. We then define two fine-grid initial conditions
\[
\phi_0^{(1)} := \phi_{\rm naive}+\varepsilon\,\delta\phi,
\qquad
\phi_0^{(2)} := \phi_{\rm naive}-\varepsilon\,\delta\phi,
\]
with $\varepsilon>0$ chosen so that the fine-grid separation $\|\phi_0^{(1)}-\phi_0^{(2)}\|_{L^2(\omega)}$ is visible.
Therefore, $I_H\phi_0^{(1)}=I_H\phi_0^{(2)}$ up to machine precision, while $\phi_0^{(1)}\neq \phi_0^{(2)}$ on the fine mesh. In this test, we choose $\varepsilon=10.0$. We use a relatively coarse observation mesh $\mathcal T_h$ of size $8\times 8$ such that most information on the fine mesh are not resolved by the observation operator $I_H$, thus producing two separated initial conditions $\phi_0^{(1)}, \phi_0^{(2)}$. For the other variables, we use the same procedure as in the baseline setup, and evolve both reference systems independently.

We report four log-scale $L^2$-differences:
(i) Ref$_1$--DA$_1$ and (ii) Ref$_2$--DA$_2$, which measure synchronization of CDA to the driven reference trajectory; and
(iii) Ref$_1$--Ref$_2$ and (iv) DA$_1$--DA$_2$, which measure the separation between the two reference solutions and the
separation between the two assimilated solutions. Here, DA$_1$ denotes the CDA run driven by coarse observations from Ref$_1$,
and DA$_2$ is driven by coarse observations from Ref$_2$ (both started from the same assimilated initial guess).

The tables and plots Figures~\ref{fig:test4tot}--\ref{fig:test4u} show that Ref$_1$--DA$_1$ and Ref$_2$--DA$_2$ decay to near machine precision for $\phi$, $\vec\psi$,
$\vec u$, and the combined error, while Ref$_1$--Ref$_2$ remains nonzero over the time window and DA$_1$--DA$_2$ overlaps
with Ref$_1$--Ref$_2$. Thus, even though the two reference runs have the same coarse-grid initial data, their fine-grid
evolutions separate, and CDA tracks whichever reference solution supplies the observations.

\begin{figure}[H]
	\centering
	\subfloat[\textbf{Last six values (total).}]{
		\begin{tabular}{lrrrr}
\toprule
$t$ & Ref1-DA1 & Ref2-DA2 & Ref1-Ref2 & DA1-DA2 \\
\midrule
99.95 & 1.61E-05 & 1.75E-05 & 1.61E-02 & 1.61E-02 \\
99.96 & 1.61E-05 & 1.75E-05 & 1.61E-02 & 1.61E-02 \\
99.97 & 1.61E-05 & 1.75E-05 & 1.61E-02 & 1.61E-02 \\
99.98 & 1.61E-05 & 1.75E-05 & 1.61E-02 & 1.61E-02 \\
99.99 & 1.61E-05 & 1.75E-05 & 1.61E-02 & 1.61E-02 \\
100.00 & 1.61E-05 & 1.75E-05 & 1.61E-02 & 1.61E-02 \\
\bottomrule
\end{tabular}

	}
	\hspace{0.0cm}
	\begin{subfigure}[t]{0.43\linewidth}
		\includegraphics[width=\linewidth]{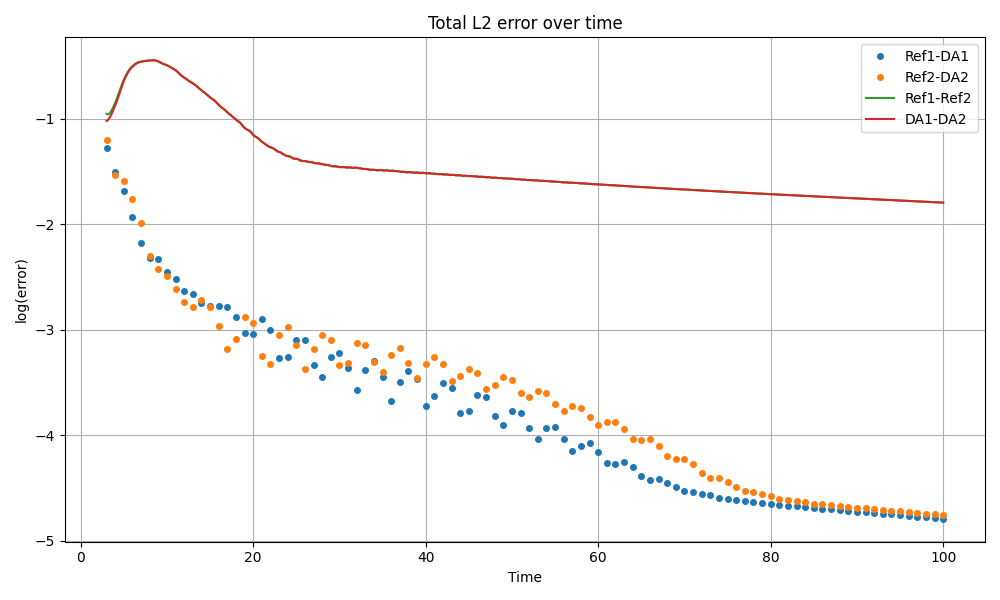}
		\subcaption{Total error: Ref1--DA1 and Ref2--DA2 decay to near machine precision, while Ref1--Ref2 and DA1--DA2 overlap and plateau.}
	\end{subfigure}
	\caption{Total error comparison showing CDA tracks the driven reference solution.}
	\label{fig:test4tot}
\end{figure}

\begin{figure}[H]
	\centering
	\subfloat[\textbf{Last six values ($\varphi$).}]{
		\begin{tabular}{lrrrr}
\toprule
$t$ & Ref1-DA1 & Ref2-DA2 & Ref1-Ref2 & DA1-DA2 \\
\midrule
99.95 & 4.89E-10 & 9.37E-10 & 2.33E-04 & 2.33E-04 \\
99.96 & 4.89E-10 & 9.25E-10 & 2.33E-04 & 2.33E-04 \\
99.97 & 4.89E-10 & 9.14E-10 & 2.33E-04 & 2.33E-04 \\
99.98 & 4.89E-10 & 9.04E-10 & 2.33E-04 & 2.33E-04 \\
99.99 & 4.89E-10 & 8.95E-10 & 2.33E-04 & 2.33E-04 \\
100.00 & 4.89E-10 & 8.88E-10 & 2.33E-04 & 2.33E-04 \\
\bottomrule
\end{tabular}

	}
	\hspace{0.0cm}
	\begin{subfigure}[t]{0.43\linewidth}
		\includegraphics[width=\linewidth]{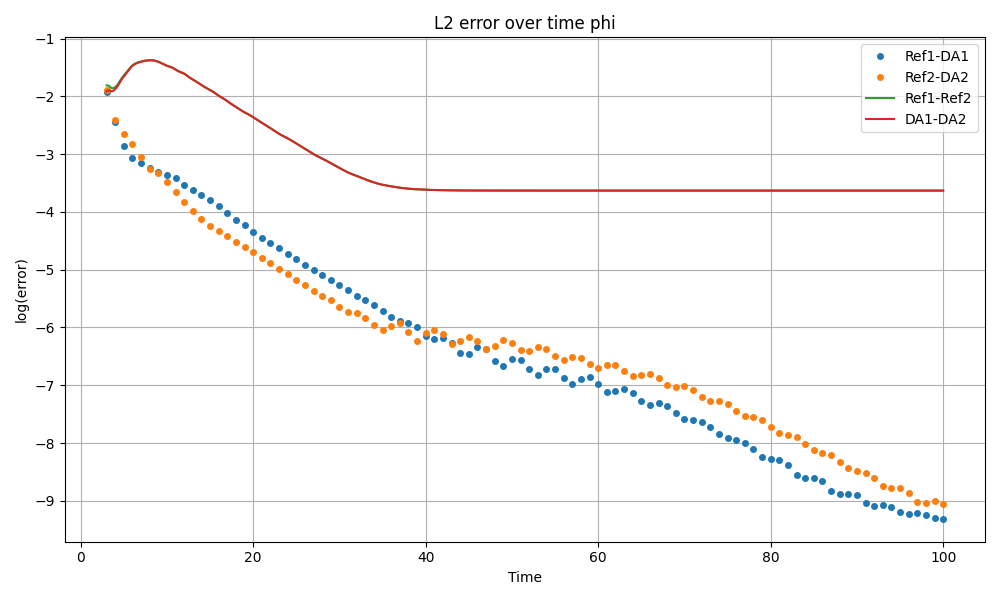}
		\subcaption{$\varphi$ error: Ref1--DA1 and Ref2--DA2 decay, while Ref1--Ref2 and DA1--DA2 remain matched and plateau.}
	\end{subfigure}
	\caption{$\phi$ error comparison.}
	\label{fig:test4phi}
\end{figure}

\begin{figure}[H]
	\centering
	\subfloat[\textbf{Last six values ($\vec\psi$).}]{
		\begin{tabular}{lrrrr}
\toprule
$t$ & Ref1-DA1 & Ref2-DA2 & Ref1-Ref2 & DA1-DA2 \\
\midrule
99.95 & 1.17E-12 & 2.79E-12 & 7.04E-03 & 7.04E-03 \\
99.96 & 1.17E-12 & 2.79E-12 & 7.04E-03 & 7.04E-03 \\
99.97 & 1.16E-12 & 2.78E-12 & 7.04E-03 & 7.04E-03 \\
99.98 & 1.16E-12 & 2.78E-12 & 7.04E-03 & 7.04E-03 \\
99.99 & 1.16E-12 & 2.77E-12 & 7.04E-03 & 7.04E-03 \\
100.00 & 1.16E-12 & 2.77E-12 & 7.04E-03 & 7.04E-03 \\
\bottomrule
\end{tabular}

	}
	\hspace{0.0cm}
	\begin{subfigure}[t]{0.43\linewidth}
		\includegraphics[width=\linewidth]{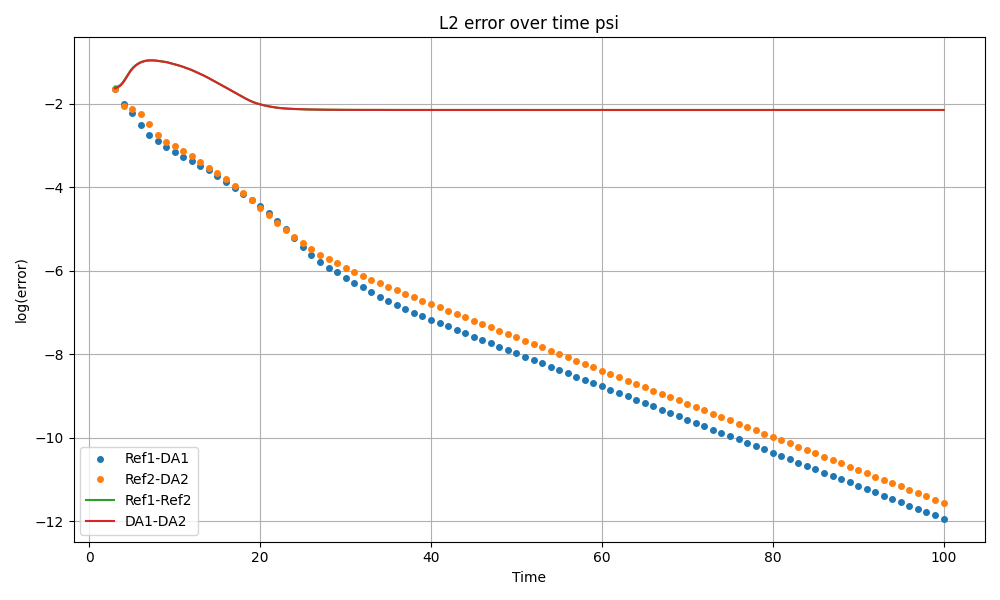}
		\subcaption{$\vec\psi$ error: Ref1--DA1 and Ref2--DA2 decay, while Ref1--Ref2 and DA1--DA2 overlap and plateau.}
	\end{subfigure}
	\caption{$\vec\psi$ error comparison.}
	\label{fig:test4psi}
\end{figure}

\begin{figure}[H]
	\centering
	\subfloat[\textbf{Last six values ($\vec u$).}]{
		\begin{tabular}{lrrrr}
\toprule
$t$ & Ref1-DA1 & Ref2-DA2 & Ref1-Ref2 & DA1-DA2 \\
\midrule
99.95 & 1.61E-05 & 1.75E-05 & 8.81E-03 & 8.80E-03 \\
99.96 & 1.61E-05 & 1.75E-05 & 8.80E-03 & 8.80E-03 \\
99.97 & 1.61E-05 & 1.75E-05 & 8.80E-03 & 8.80E-03 \\
99.98 & 1.61E-05 & 1.75E-05 & 8.80E-03 & 8.80E-03 \\
99.99 & 1.61E-05 & 1.75E-05 & 8.80E-03 & 8.80E-03 \\
100.00 & 1.61E-05 & 1.75E-05 & 8.80E-03 & 8.80E-03 \\
\bottomrule
\end{tabular}

	}
	\hspace{0.0cm}
	\begin{subfigure}[t]{0.43\linewidth}
		\includegraphics[width=\linewidth]{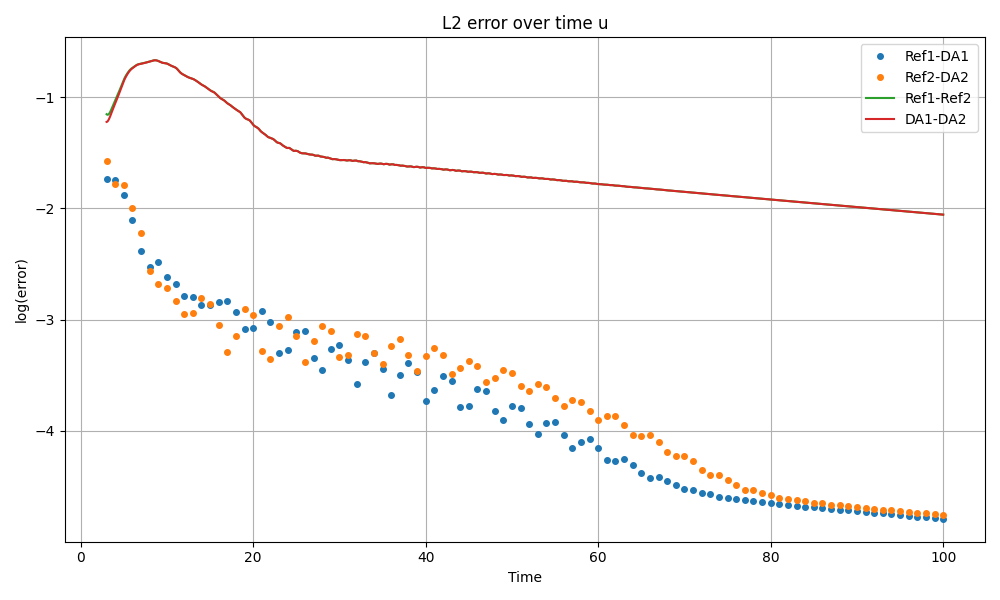}
		\subcaption{$\vec u$ error: Ref1--DA1 and Ref2--DA2 decay, while Ref1--Ref2 and DA1--DA2 remain matched and plateau.}
	\end{subfigure}
	\caption{Velocity $\vec u$ error comparison.}
	\label{fig:test4u}
\end{figure}

\begin{figure}[H]
\centering
\setlength{\tabcolsep}{1pt}
\renewcommand{\arraystretch}{0}
\scriptsize
\begin{tabular}{@{}c*{7}{c}@{}}

\raisebox{0.4\height}{\makebox[0pt][c]{\rotatebox{90}{\textbf{Ref1-Ref2}}}} &
\includegraphics[width=0.155\textwidth]{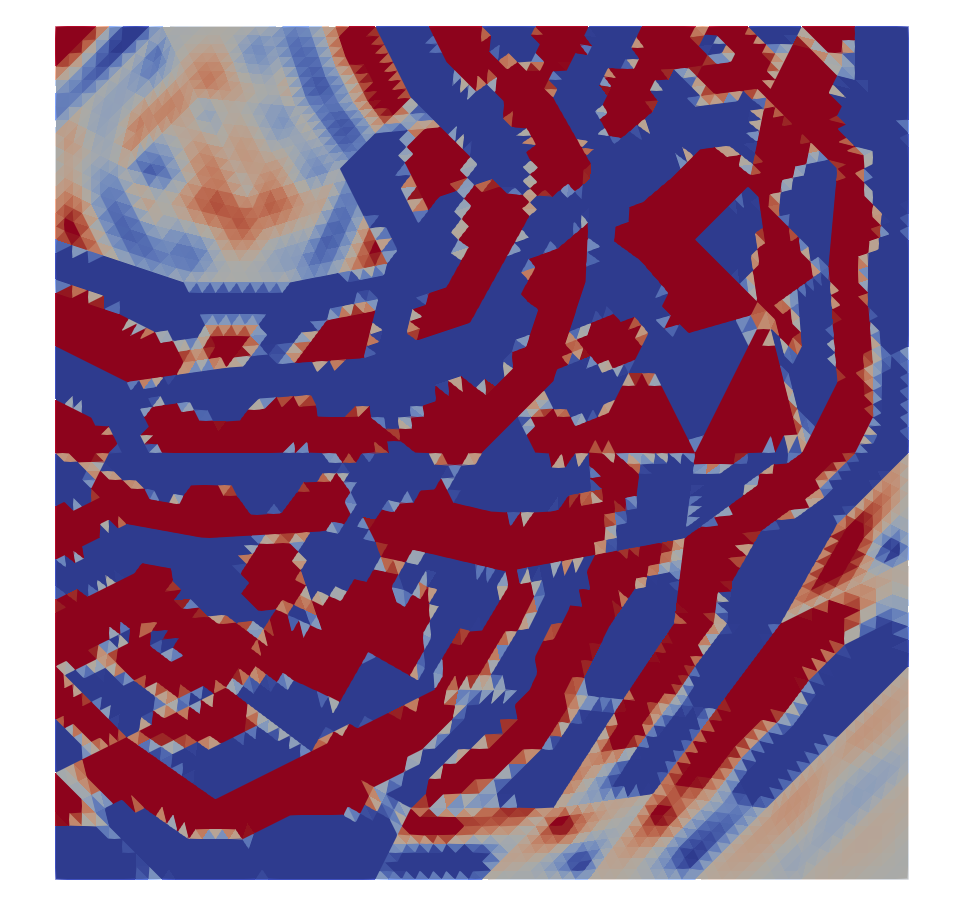} &
\includegraphics[width=0.155\textwidth]{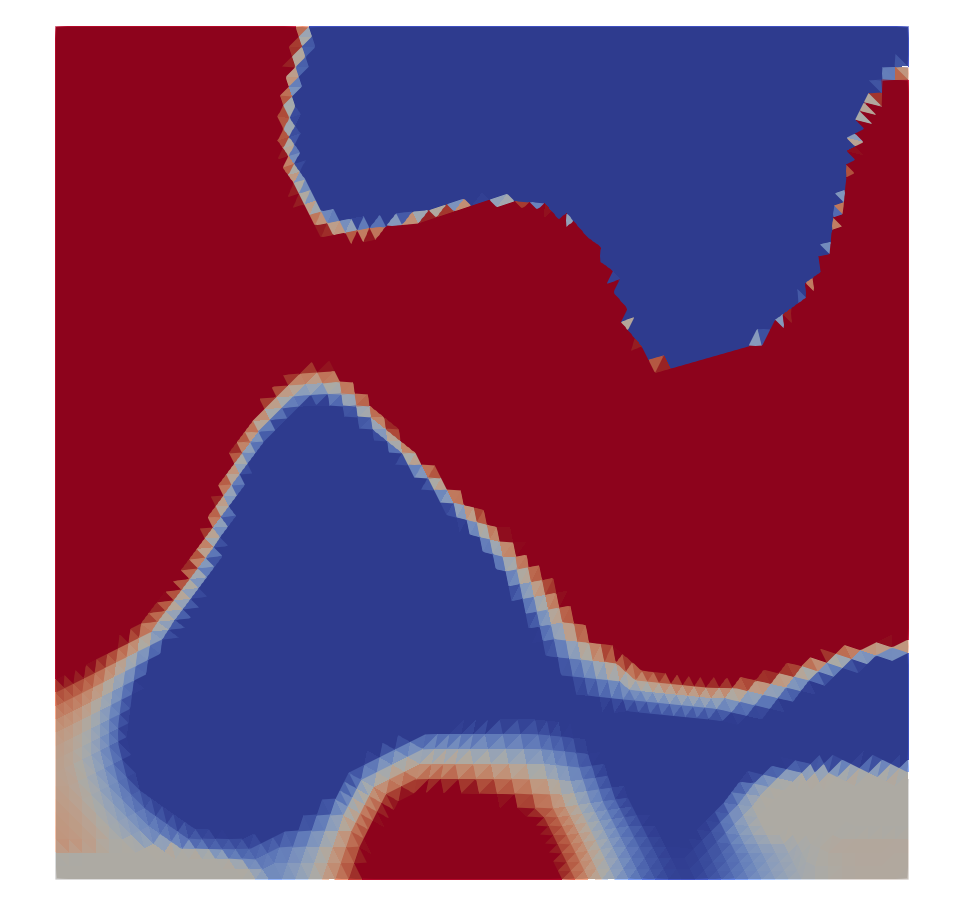} &
\includegraphics[width=0.155\textwidth]{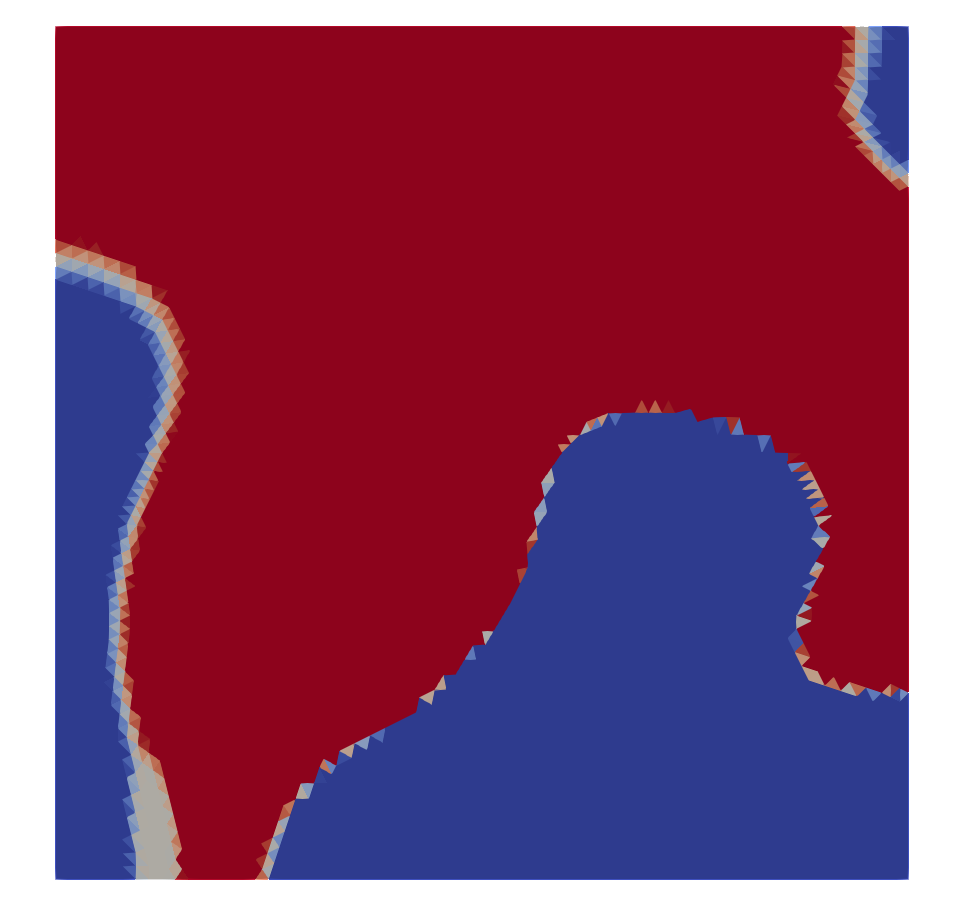} &
\includegraphics[width=0.155\textwidth]{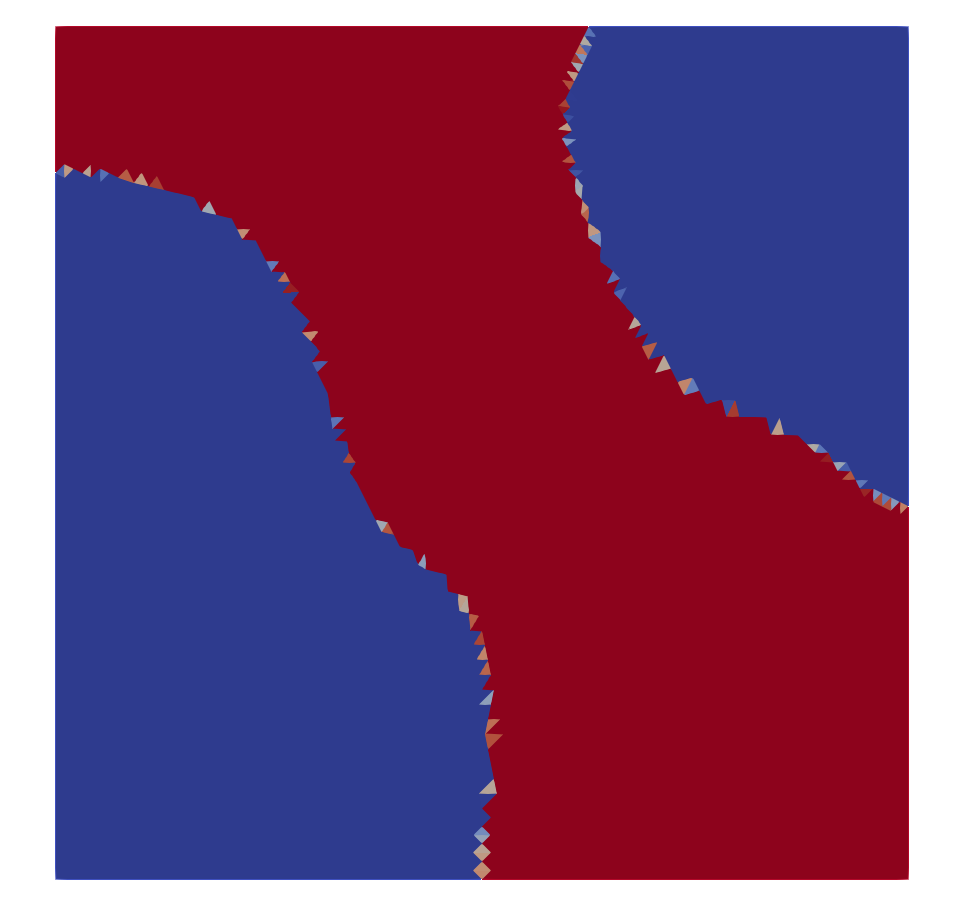} &
\includegraphics[width=0.155\textwidth]{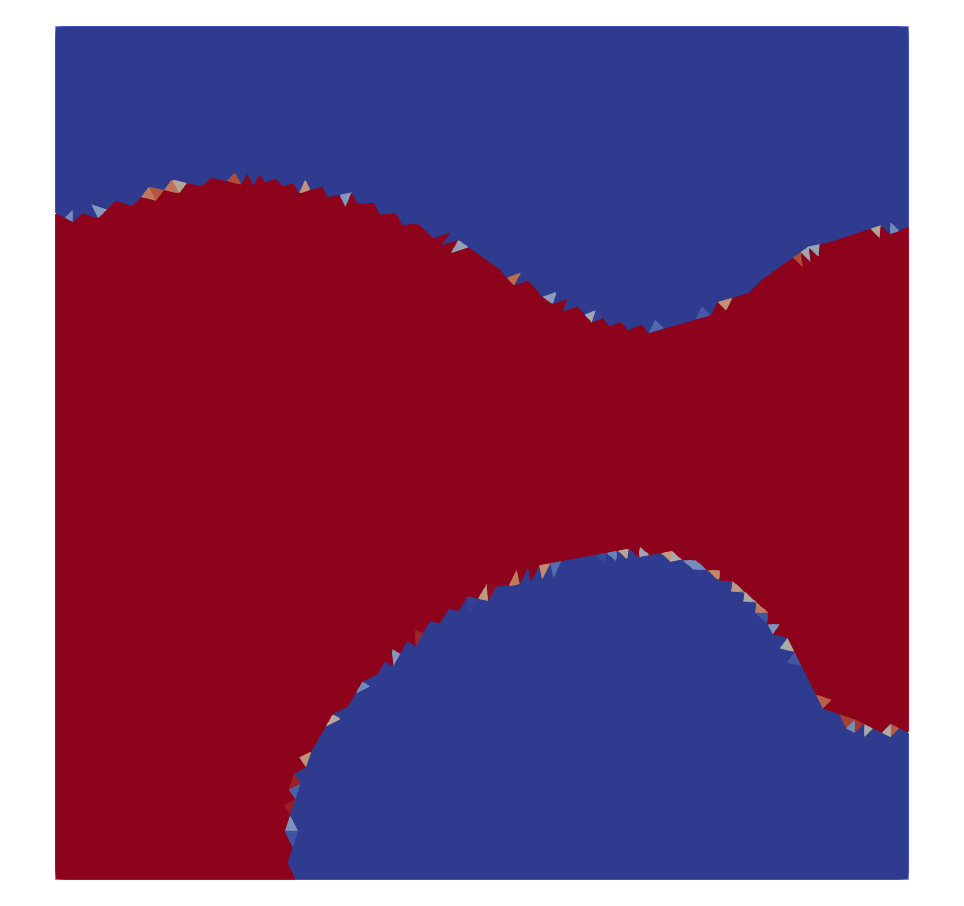} &
\includegraphics[width=0.155\textwidth]{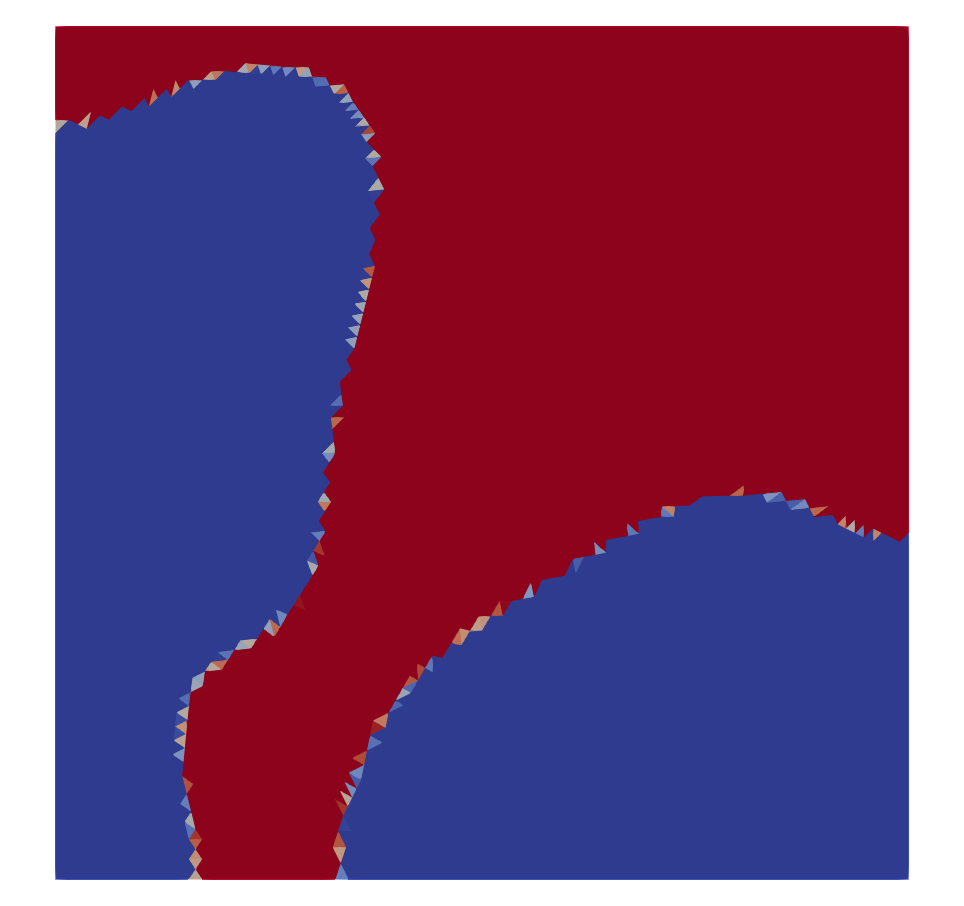} & \includegraphics[width=0.04\textwidth]{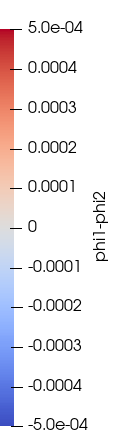}
\\[2mm]

\raisebox{0.4\height}{\makebox[0pt][c]{\rotatebox{90}{\textbf{DA1-DA2}}}} &
\includegraphics[width=0.155\textwidth]{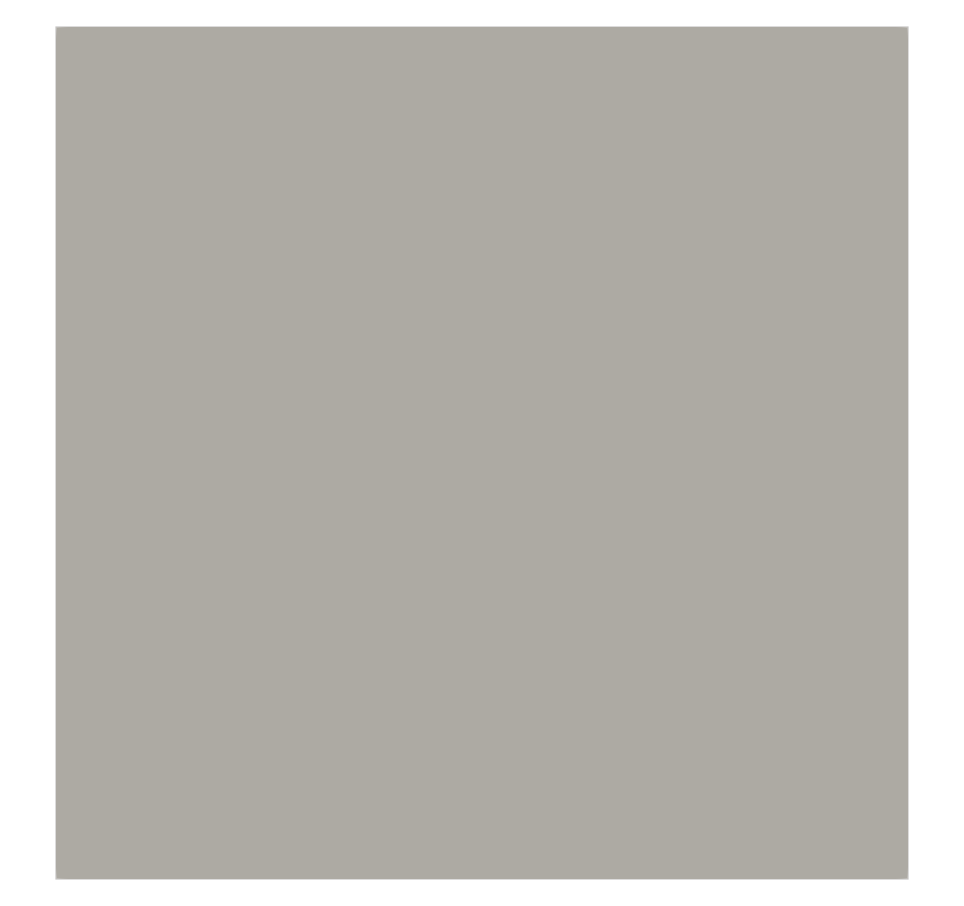} &
\includegraphics[width=0.155\textwidth]{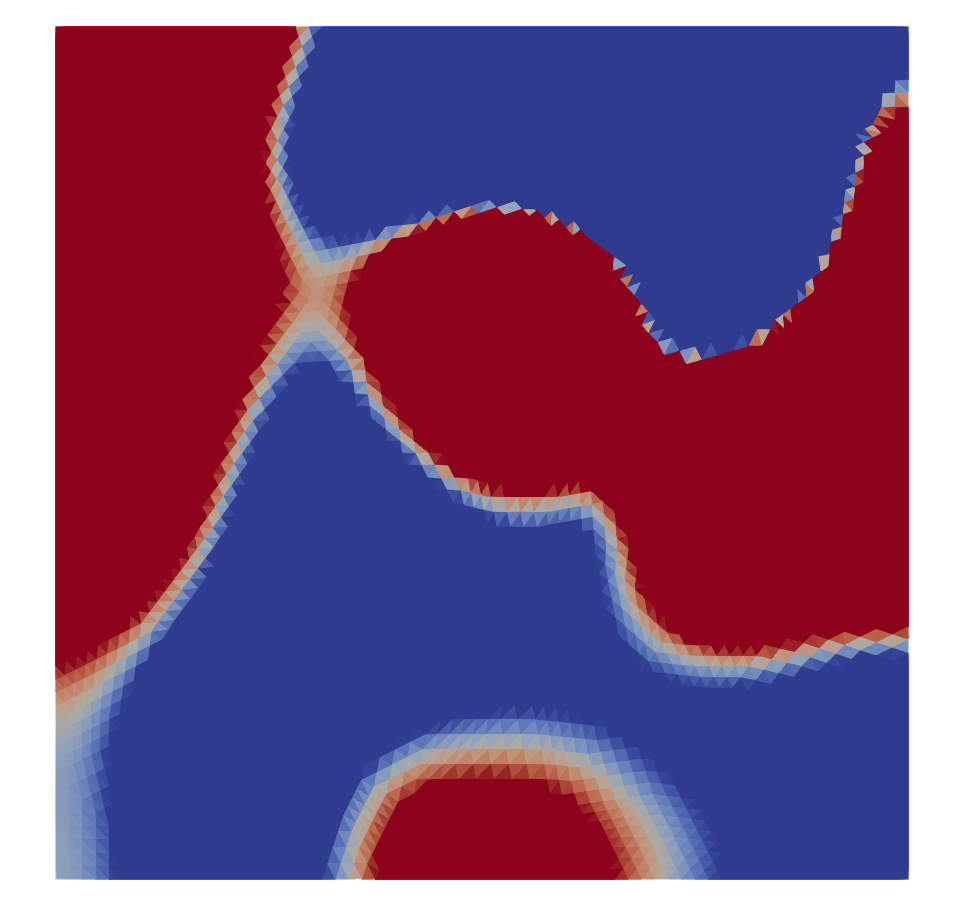} &
\includegraphics[width=0.155\textwidth]{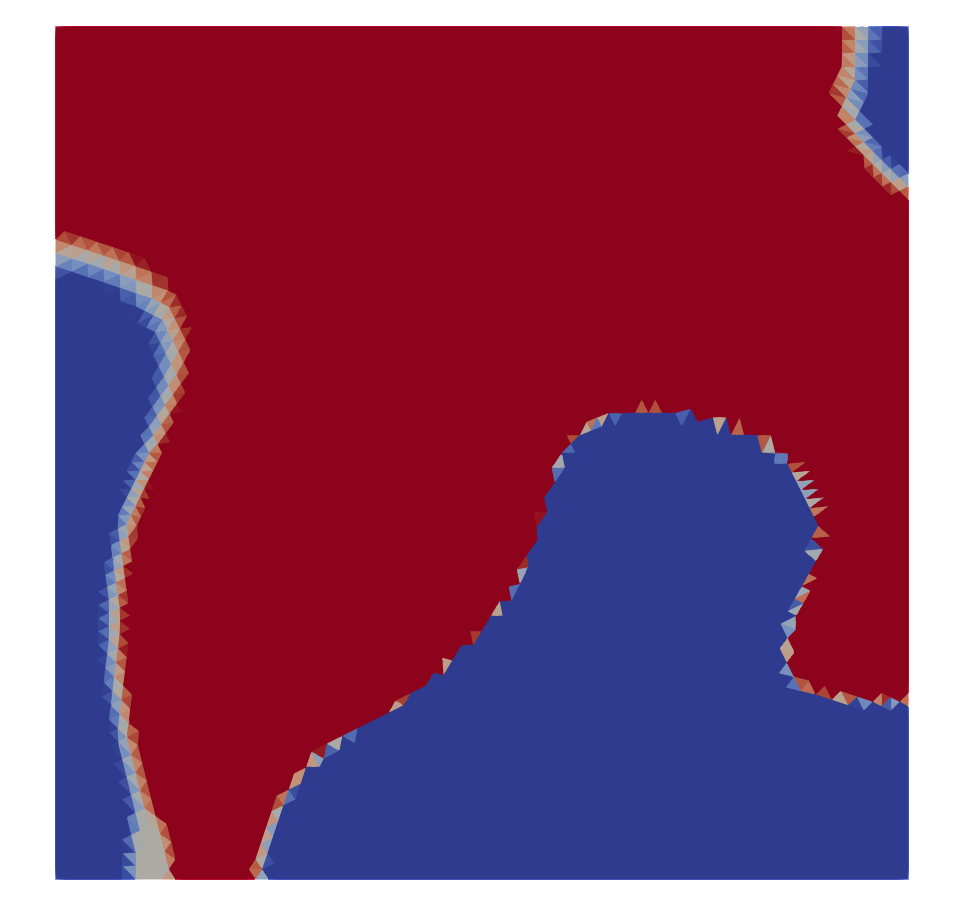} &
\includegraphics[width=0.155\textwidth]{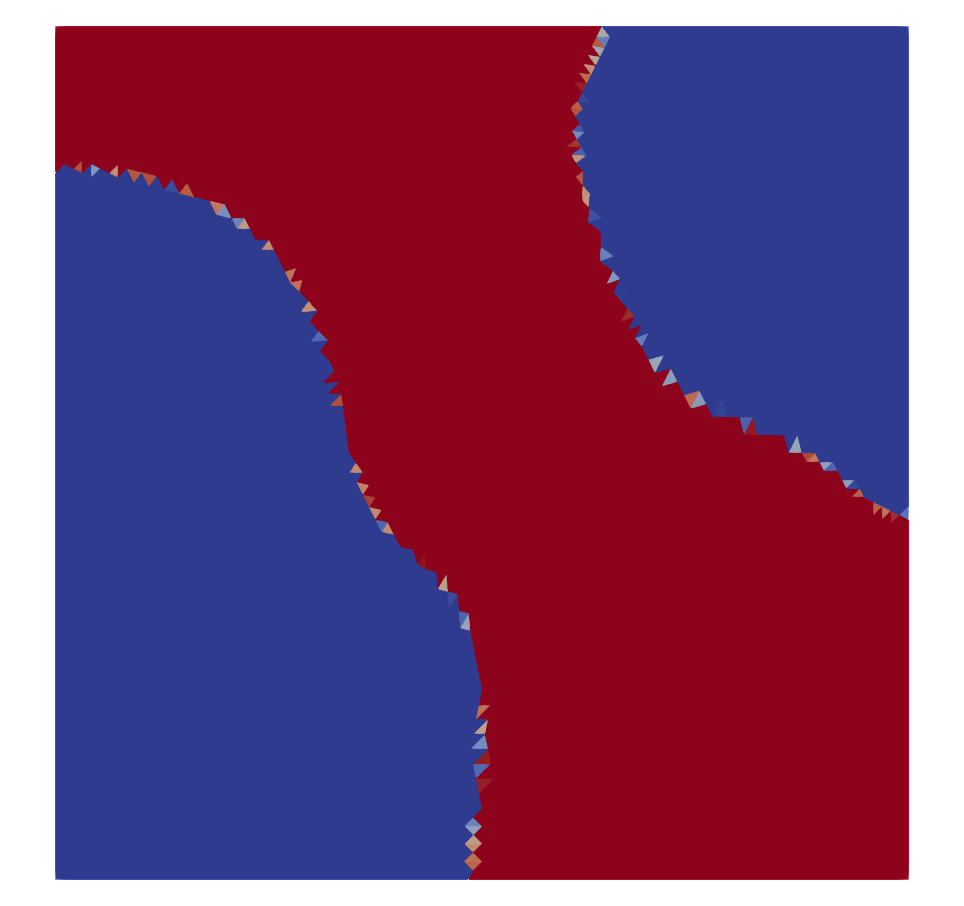} &
\includegraphics[width=0.155\textwidth]{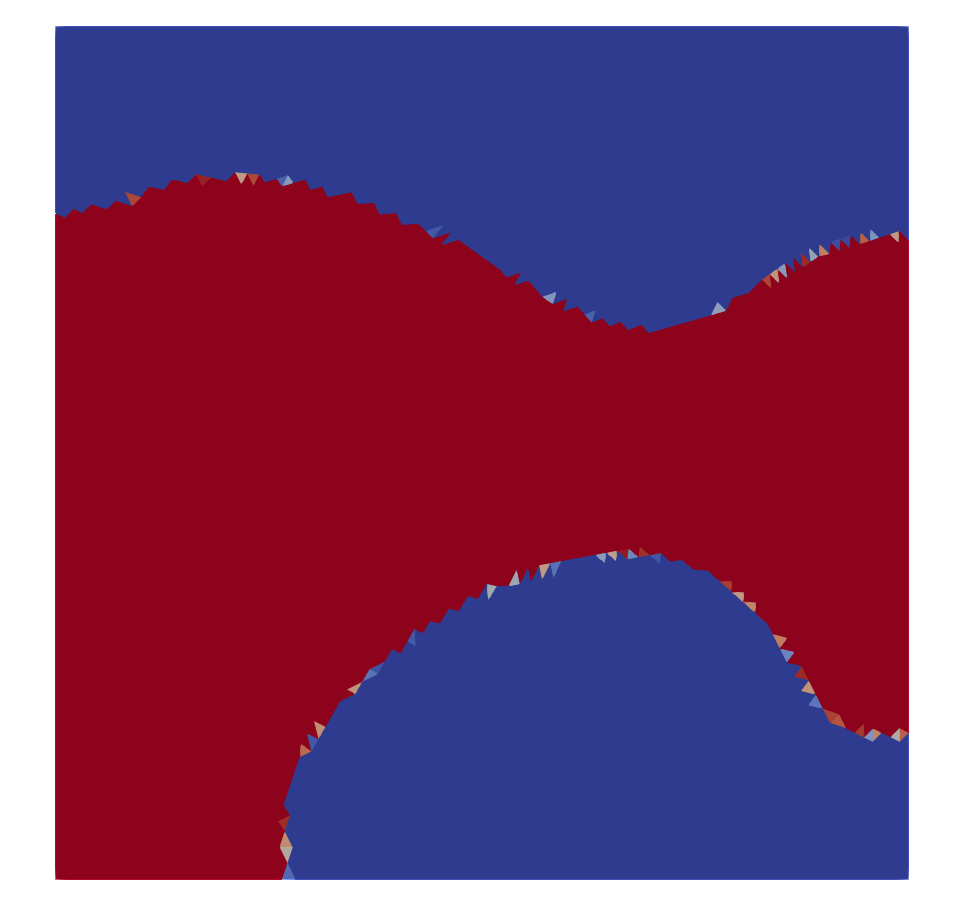} &
\includegraphics[width=0.155\textwidth]{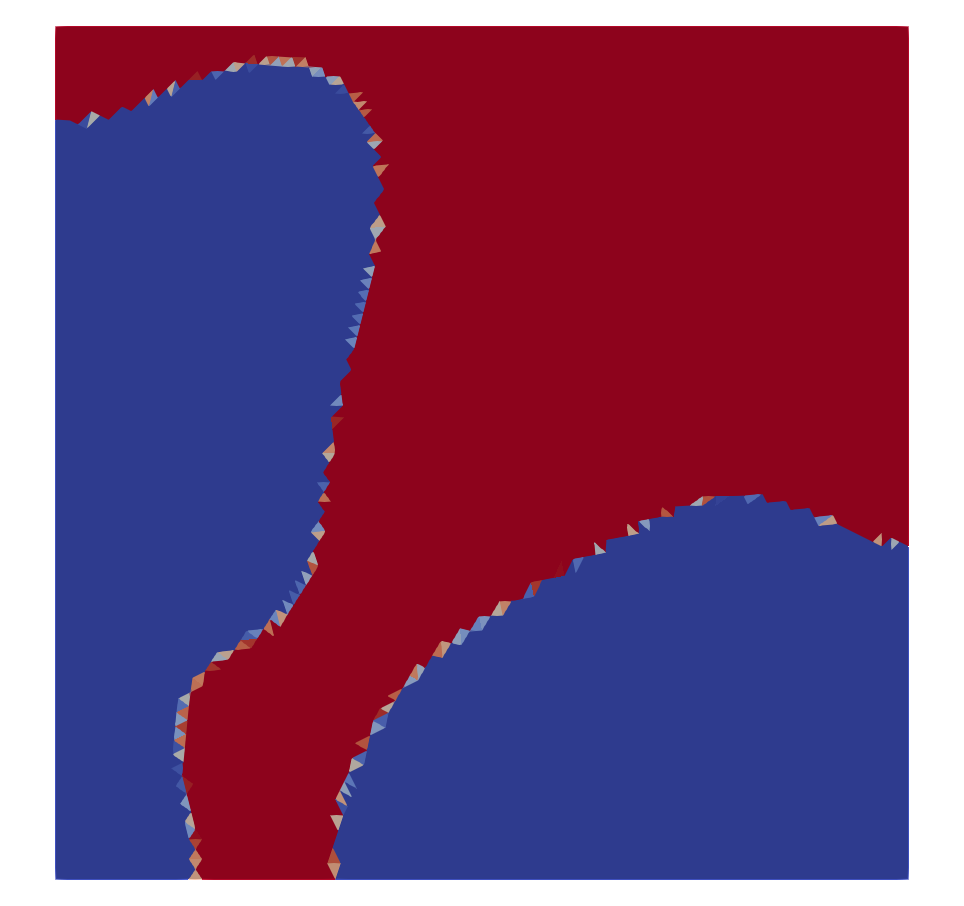}
& \includegraphics[width=0.04\textwidth]{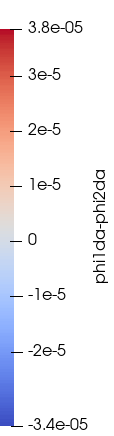}\\[2mm]

& $t=0$ & $t=2$ & $t=4$ & $t=6$ & $t=8$ & $t=10$
\end{tabular}
\caption{Order parameter difference snapshots. Top: reference solution $\phi^{(1)}-\phi^{(2)}$ (Ref1--Ref2). Bottom: nudging solution
$\varphi^{(1)}-\varphi^{(2)}$ (DA1--DA2), where DA1 and DA2 start from the same assimilated initial guess but are driven by coarse observations from Ref1 and Ref2, respectively.}
\label{fig:shear-phi-diff-compare}
\end{figure}

\paragraph{Test 5: effects of nudging coefficient}
\label{subsec:nudging-coef-test}

In this test, we investigate how the synchronization rate depends on the strength of the feedback.
We repeat the setup of Test~3 (same domain, discretization, observation operator $I_H$, parameters, and mismatched
initial data), but vary the nudging parameters. For simplicity, we use a single coefficient
\[
\alpha_u=\alpha_\phi=\alpha_\psi=:\alpha,
\qquad \alpha\in\{1,\;10^{-1},\;10^{-2}\}.
\]
We plot the errors on a logarithmic scale versus time.

Figure~\ref{fig:alpha-errors} shows a clear dependence on $\alpha$.
For $\alpha=1$, all primary variables synchronize rapidly: the curves display a steep initial drop and then
enter a slower decay regime, consistent with reaching an observation/discretization-limited accuracy.
For $\alpha=0.1$, the decay is still robust but noticeably slower; over most of the time interval, the trajectories
are close to straight lines on the log scale, indicating approximately exponential decay with a smaller rate.
For $\alpha=0.01$, the feedback is too weak on the time horizon $[0,100]$: the errors decrease only mildly and full synchronization is not observed within the final time.

We also report the pressure iterate difference in Figure~\ref{fig:alpha-errors} (top right). Since pressure is only
defined up to an additive constant and is obtained here via a pressure-correction update, its curve can remain highly
oscillatory even when the primary variables synchronize. Accordingly, the most reliable synchronization indicators are
the velocity and phase-field errors.

\begin{figure}[H]
\centering
\includegraphics[width=0.48\textwidth]{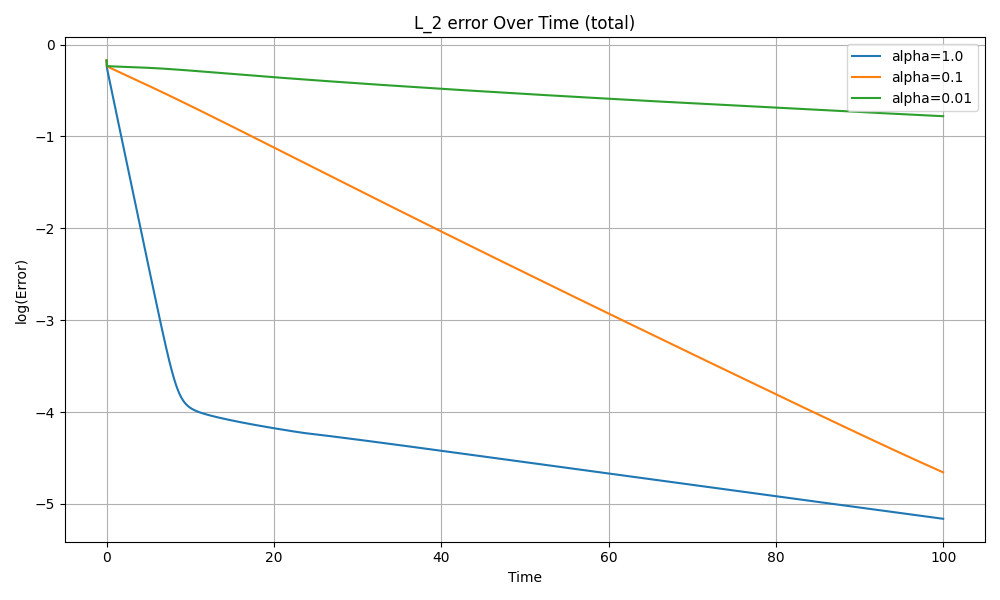}\hfill
\includegraphics[width=0.48\textwidth]{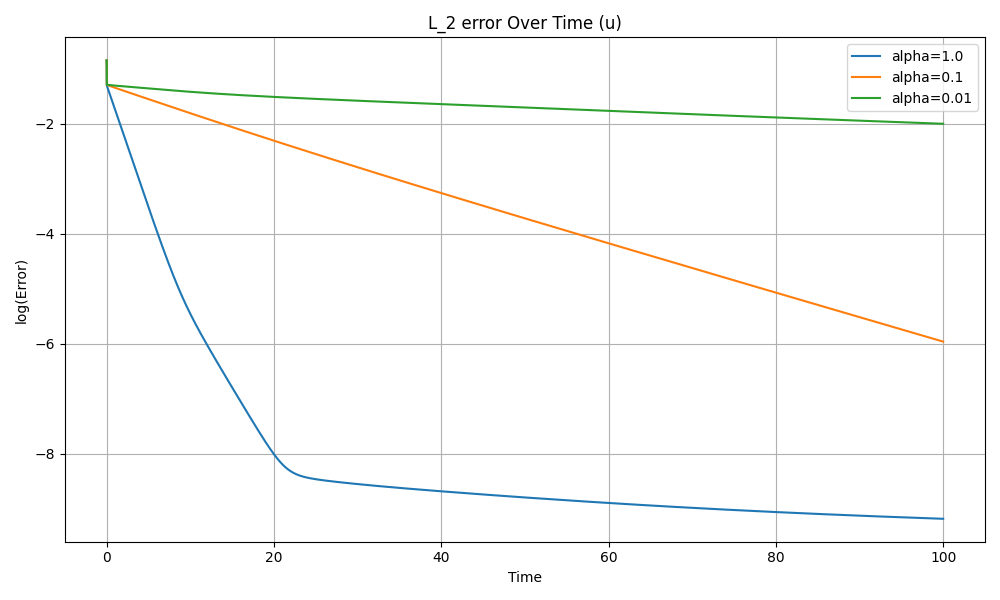}\hfill\\[0.6ex]
\includegraphics[width=0.48\textwidth]{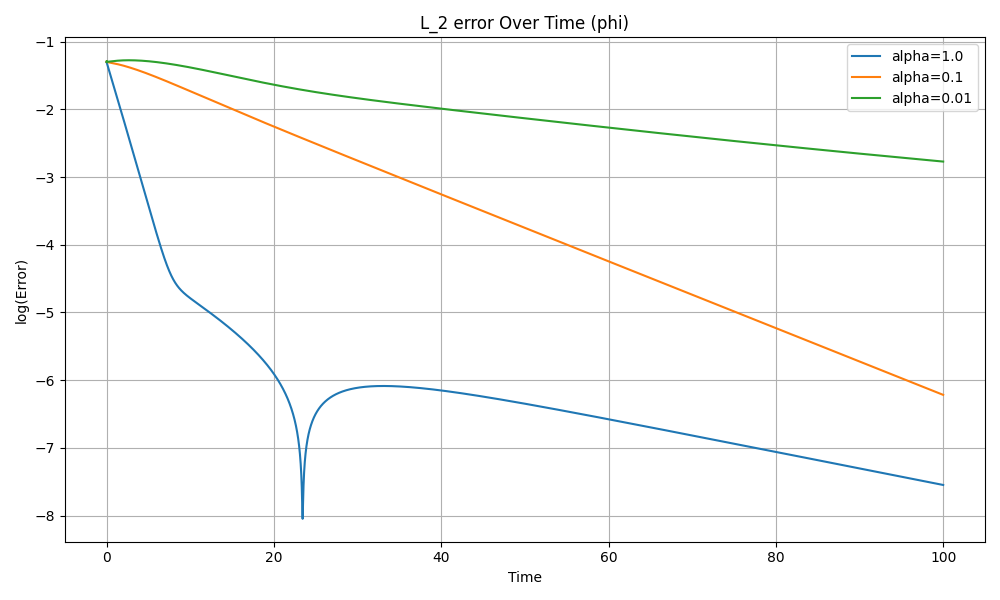}\hfill
\includegraphics[width=0.48\textwidth]{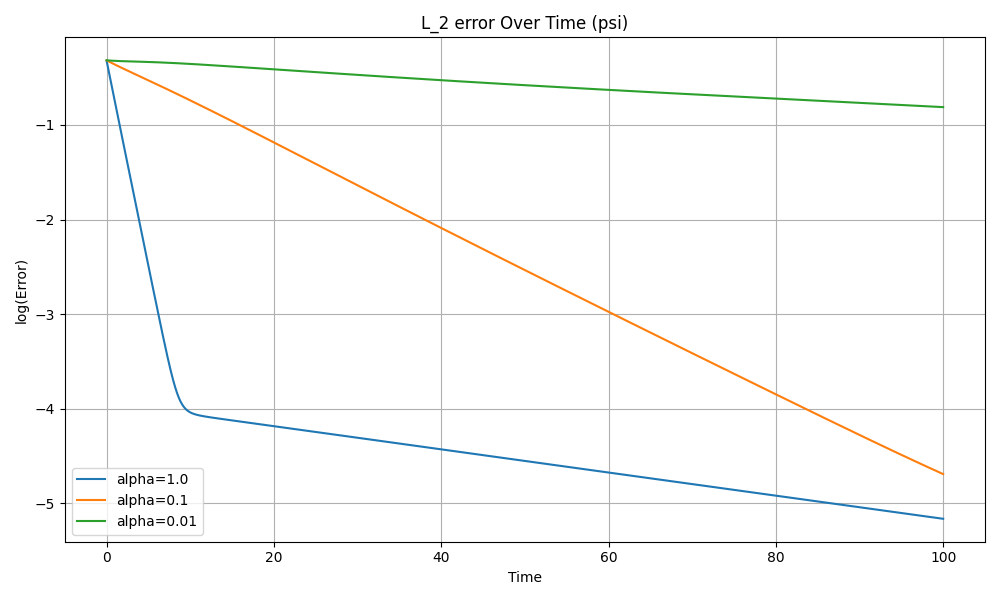}
\caption{Log-scale errors for $\alpha_u=\alpha_\phi=\alpha_\psi=\alpha\in\{1,0.1,0.01\}$.. Top left: combined error.Top right: velocity error. 
Bottom row: componentwise errors in $\phi$ and $\vec\psi$.}
\label{fig:alpha-errors}
\end{figure}

\paragraph{Test 6: effects of interpolation mesh size}
\label{subsec:interp-mesh-test}

In this test, we examine the dependence of synchronization on the resolution of the observation/interpolation operator
$I_H$. We keep the physical parameters, time step, and nudging strengths fixed with
\[
\alpha_u=\alpha_\phi=\alpha_\psi=1,
\]
and vary only the observation mesh size $H$. We consider three observation resolutions (labeled $32\times 32$,
$16\times 16$, and $8\times 8$), while maintaining the same overall simulation setup and mismatched initial data as in
the preceding tests.

Figure~\ref{fig:H-errors} shows a clear monotone trend: finer observations lead to faster and stronger synchronization.
The most pronounced sensitivity to $H$ occurs in the velocity component. The velocity error exhibits a steep initial decay and reaches a very small plateau by approximately $t\approx 20$ on the finest observation mesh ($32\times 32$). For $16\times 16$ observations, the same qualitative behavior persists but with a noticeably reduced
decay rate. For the coarsest $8\times 8$ observations, the velocity error decreases much more slowly and remains several
orders of magnitude larger over the entire interval. This behavior is consistent with the fact that coarser observations
cannot constrain finer-scale features of the flow, thereby limiting the achievable synchronization accuracy.

In contrast, the phase-field errors in $\phi$ and $\vec\psi$ are less sensitive to $H$ on the time horizon considered,
and the combined error is largely governed by the late-time behavior of the auxiliary field $\vec\psi$. The energy
diagnostics further support this picture: the elastic, mixing, and total energies remain nearly indistinguishable across
all observation resolutions, while the kinetic energy exhibits the clearest $H$-dependence, with coarser observations
associated with a larger residual kinetic component.

\begin{figure}[H]
\centering
\includegraphics[width=0.48\textwidth]{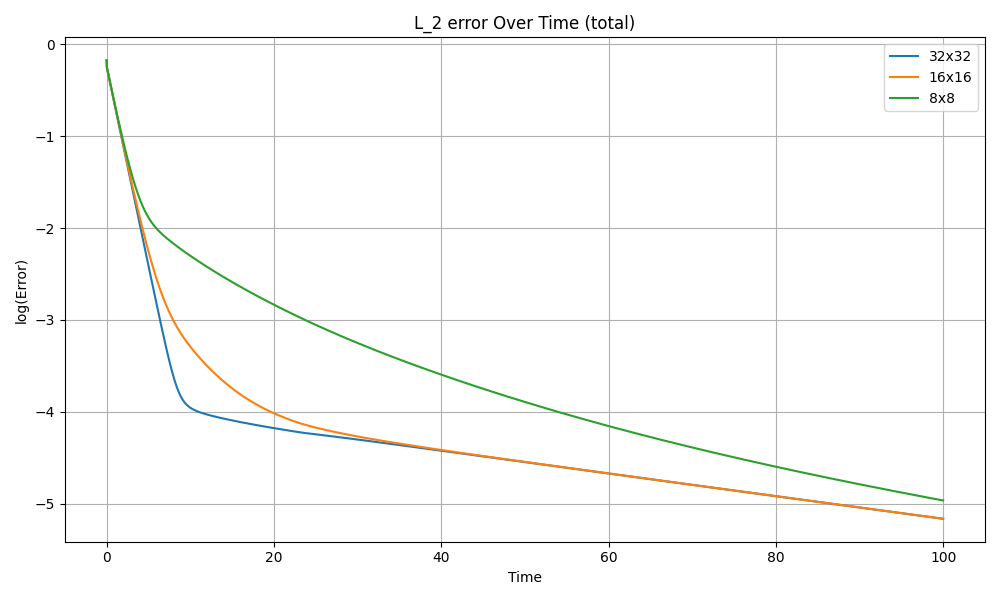}\hfill
\includegraphics[width=0.48\textwidth]{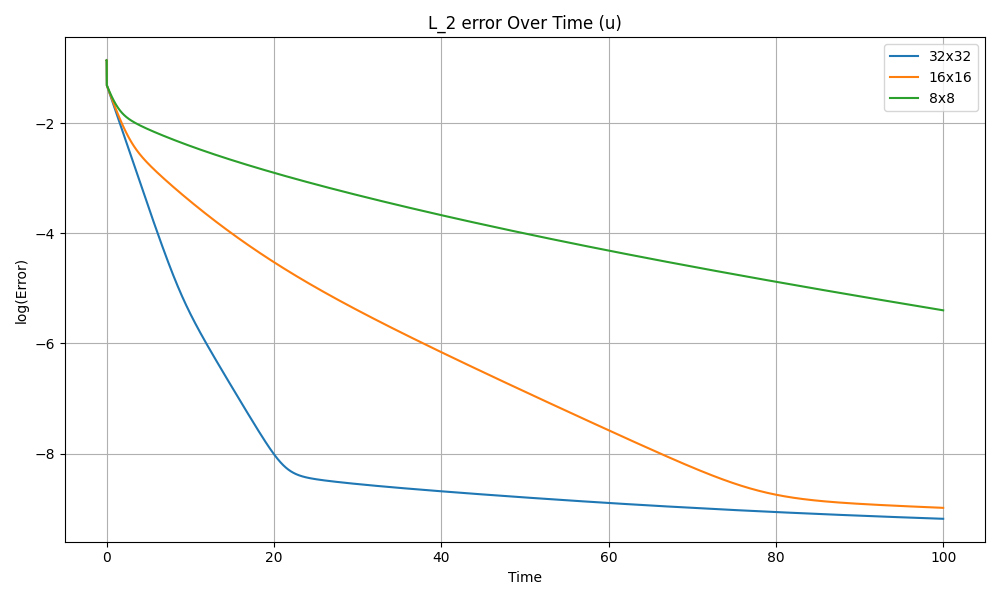}\\[0.6ex]
\includegraphics[width=0.48\textwidth]{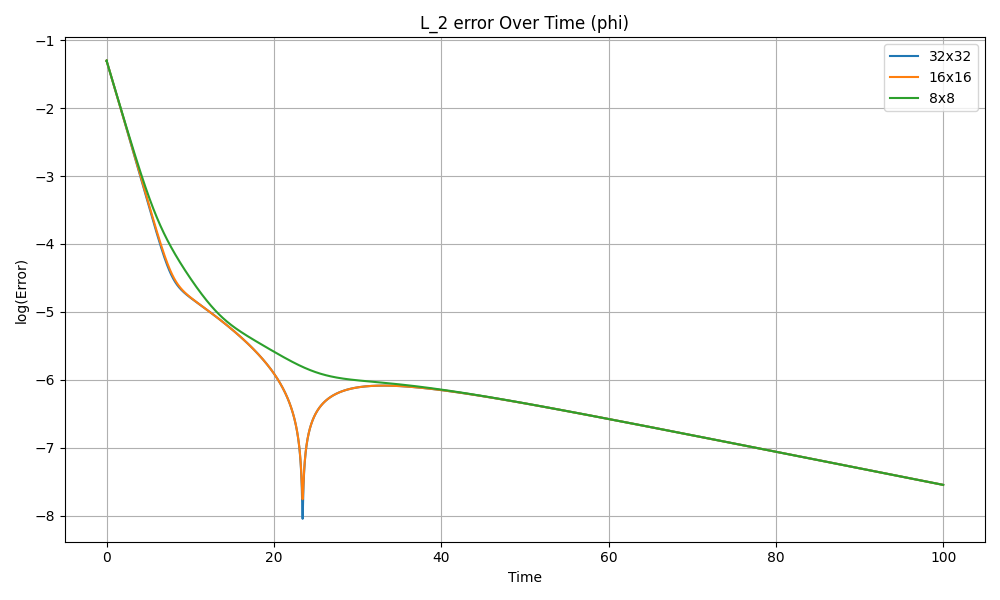}\hfill
\includegraphics[width=0.48\textwidth]{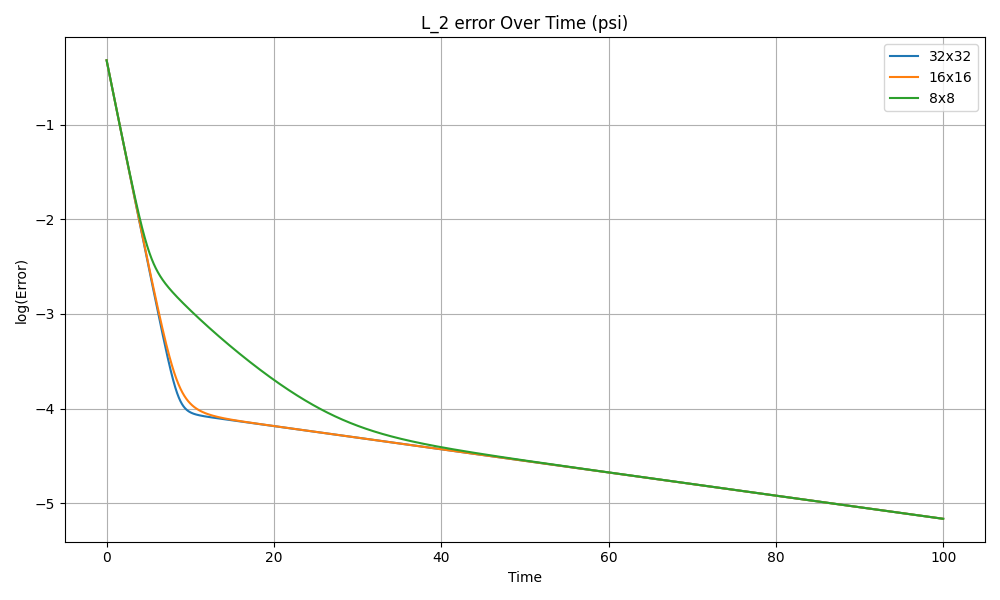}
\caption{Log-scale errors for three interpolation mesh sizes $H$
(labeled $32\times 32$, $16\times 16$, and $8\times 8$). Top left: combined error. Top right: velocity error. 
Bottom row: componentwise errors in $\phi$ and $\vec\psi$.}
\label{fig:H-errors}
\end{figure}

\addtocontents{toc}{\protect\setcounter{tocdepth}{1}}

\section*{Conclusions and Commentary}

In this paper we formulated a continuous data assimilation framework for a two-dimensional NSCH-type phase-field model coupled with a transported auxiliary field. The additional transported field contributes an extra stress term in the momentum balance and provides a natural extension of more classical NSCH dynamics arising in complex-fluid and thrombus-inspired phase-field modeling. Our goal was to develop a coarse-observation CDA formulation for this coupled system, identify its basic structural properties, and study its practical behavior through a fully discrete finite element realization.

On the continuous level, we introduced the nudged NSCH--$\psi$ system and recorded two structural properties relevant to the assimilation dynamics. First, we derived a formal energy law for the reference system, showing that the coupled model retains a dissipative structure compatible with the underlying phase-field formulation. Second, we analyzed the behavior of the phase mean and showed that, although the reference phase mass is conserved, the assimilated phase mean is generally influenced by the observation operator and the nudging term. These identities clarify the roles played by the coupled stresses and by the coarse-observation feedback mechanism.

On the discrete level, we proposed a capped fully discrete finite element splitting scheme for the CDA system. The method uses continuous quadratic elements for the phase, chemical potential, velocity, and auxiliary field, together with continuous linear elements for the pressure. The capped phase is used in the evaluation of the phase-dependent coefficients in order to reflect the numerical implementation and preserve admissible coefficient values. Within this framework, we proved one-step well-posedness of the fully discrete problem and established a stepwise a priori estimate for the capped scheme. Although this estimate does not constitute a closed discrete free-energy law and does not provide a full convergence theory, it gives a basic stability bound for the numerical method employed in the experiments.

The numerical experiments illustrate several features of the proposed CDA approach. In particular, they show recovery from strongly mismatched initial data, sensitivity of the synchronization behavior to the nudging strength and the observation resolution, and robustness under boundary forcing. The coarse-indistinguishability experiment further demonstrates that coarse initial information alone may fail to determine the fine-scale evolution uniquely, while the assimilated trajectory is selected by the supplied time-dependent observations. Taken together, these computations indicate that the proposed nudging strategy is effective in practice for this coupled NSCH-type system.

Several directions remain open. One natural next step is to establish a rigorous well-posedness and synchronization theory for the continuous assimilated dynamics under assumptions compatible with the present model. Another is to strengthen the discrete analysis by developing a closed stability theory or a convergence analysis for the capped fully discrete scheme. It would also be of interest to investigate partial-observation scenarios by nudging only the velocity or only the phase field and to characterize the minimal observation content needed for synchronization in this coupled setting. Finally, extending the analysis to three dimensions, or to additional physical couplings motivated by thrombus biomechanics and complex-fluid structure, remains an important and challenging topic.

\subsection*{Declarations}

\subsection*{Data Availability Statement}

Data sharing is not applicable to this article as no datasets were generated or analysed during the current study.

\subsection*{Funding Statement}

This work was partly supported by the Research Fund of Indiana University. 

\subsection*{Conflict of Interest Disclosure}

All authors certify that they have no affiliations with or involvement in any organization or entity with any competing interests in the subject matter or materials discussed in this manuscript.

\subsection*{Ethics Approval Statement}

Not applicable.

\subsection*{Patient Consent Statement}

Not applicable.

\subsection*{Permission to Reproduce Material from other Sources}

None of the presented material was drawn from other sources.

\subsection*{Clinical Trial Registration}

Not applicable.

\subsection*{Authors' Contribution}

The author contributed to all parts of this manuscript.

\subsection*{Acknowledgements}

T.S. is greatly thankful to Professor Michael Jolly for the insightful discussions and advice.

\bibliographystyle{abbrvnat} %unsrt
\bibliography{references}

\addresseshere

\clearpage

\appendix

\end{document}